\documentclass{amsart}
\usepackage{amsmath}
\usepackage{amsfonts}
\usepackage{graphics}
\usepackage{epsfig}
\usepackage{amssymb}
\usepackage{amscd}
\usepackage[all]{xy}
\usepackage{latexsym}
\usepackage{graphicx}
\usepackage{multirow}
\usepackage{geometry}
\usepackage{color}
\usepackage[usenames,dvipsnames]{pstricks}
\usepackage{epsfig}
\usepackage{pst-grad} 
\usepackage{pst-plot}

\newtheorem{thm}{Theorem}[section]
\newtheorem{cor}[thm]{Corollary}
\newtheorem{lem}[thm]{Lemma}

\newtheorem{conj}[thm]{Conjecture}
\newtheorem{quest}[thm]{Question}
\newtheorem{prob}[thm]{Problem}

\theoremstyle{definition}
\newtheorem{Def}[thm]{Definition}
\newtheorem{rem}[thm]{Remark}
\newtheorem*{ack}{Acknowledgement}
\newtheorem{ex}{Example}[section]
\newtheorem{case}{Case}

\numberwithin{equation}{section}
\numberwithin{figure}{section}

\def\End{{\text{\rm{End}}}}

\def\tr{{\text{\rm{tr}}}}
\def\rchi{{\hbox{\raise1.5pt\hbox{$\chi$}}}}
\def\Aut{{\text{\rm{Aut}}}}

\def\isom{\cong}
\def\tensor{\otimes}
\def\dsum{\oplus}

\def\a{\alpha}
\def\b{\beta}
\def\lam{\lambda}

\def\gam{\gamma}
\def\Gam{\Gamma}

\def\Arg{{\text{\rm{Arg}}}}

\def\supp{{\text{\rm{supp}}}}

\def\Catalan{{\text{\rm{Catalan}}}}
\def\Pic{{\text{\rm{Pic}}}}

\def\NS{{\text{\rm{NS}}}}
\def\erf{{\text{\rm{erf}}}}
\def\Gauss{{\text{\rm{Gau\ss}}}}

\def\Ext{{\text{\rm{Ext}}}}

\newcommand{\bea}{\begin{eqnarray}}
\newcommand{\eea}{\end{eqnarray}}
\newcommand{\be}{\begin{equation}}
\newcommand{\ee}{\end{equation}}

\newcommand{\Mbar}{{\overline{\mathcal{M}}}}
\newcommand{\bA}{{\mathbb{A}}}
\newcommand{\bP}{{\mathbb{P}}}
\newcommand{\bC}{{\mathbb{C}}}

\newcommand{\bF}{{\mathbb{F}}}
\newcommand{\bL}{{\mathbb{L}}}

\newcommand{\bQ}{{\mathbb{Q}}}
\newcommand{\bR}{{\mathbb{R}}}

\newcommand{\bZ}{{\mathbb{Z}}}

\newcommand{\cM}{{\mathcal{M}}}
\newcommand{\cN}{{\mathcal{N}}}

\newcommand{\cD}{{\mathcal{D}}}

\newcommand{\cO}{{\mathcal{O}}}

\newcommand{\cU}{{\mathcal{U}}}

\newcommand{\la}{{\langle}}
\newcommand{\ra}{{\rangle}}
\newcommand{\half}{{\frac{1}{2}}}
\newcommand{\bp}{{\mathbf{p}}}
\newcommand{\bx}{{\mathbf{x}}}

\newcommand{\rar}{\rightarrow}
\newcommand{\lrar}{\longrightarrow}

\begin{document}
\large
\setcounter{section}{0}

\allowdisplaybreaks

\title[Topological recursion for Higgs bundles and 
quantum curves]
{Lectures on  the topological recursion for
Higgs bundles and 
quantum curves}

\author[O.\ Dumitrescu]{Olivia Dumitrescu}
\address{
Olivia Dumitrescu:
Department of Mathematics\\
Central Michigan University\\
Mount Pleasant, MI 48859}
\address{and Simion Stoilow Institute of Mathematics\\
Romanian Academy\\
21 Calea Grivitei Street\\
010702 Bucharest, Romania}
\email{dumit1om@cmich.edu}

\author[M.\ Mulase]{Motohico Mulase}
\address{Motohico Mulase:
Department of Mathematics\\
University of California\\
Davis, CA 95616--8633, U.S.A.\\}
\address{
and Kavli Institute for Physics and Mathematics of the 
Universe\\
The University of Tokyo\\
Kashiwa, Japan}
\email{mulase@math.ucdavis.edu}

\begin{abstract}
The paper aims at giving an introduction to 
the notion of \emph{quantum curves}.
The main purpose  is to describe the new
 discovery of the relation between
the following two disparate subjects:
one is the topological recursion, that has
its origin in random matrix theory and  has been
effectively applied to many
 enumerative geometry problems; and
 the  other is
  the quantization
of Hitchin spectral curves associated with 
Higgs bundles. 
Our emphasis 
is   on explaining
the motivation and examples. Concrete
examples of the direct relation 
between Hitchin spectral curves
and enumeration problems are given.
A general geometric framework of
quantum curves is also discussed.
\end{abstract}



\maketitle
\tableofcontents

\begin{flushright}
We're not going to tell you
the story 
\\
the way it happened.
\\
We're going to tell it 
 \\
the way we remember it.
\end{flushright}

\section{Introduction}
\label{sect:intro}

Mathematicians often keep their childhood
dream for a long time. When you saw a perfect
rainbow as a child, you might have
 wondered what awaited you
when you went  over the arch. 
In a  lucky situation, you might have seen the
double, or even triple, rainbow arches spanning 
above the  brightest one, with the 
reversing color patterns on the higher arches. 
Yet we see nothing underneath
the brightest arch.

One of the purposes of these lectures 
is to  offer you 
a vision: \textbf{on the other side of the rainbow,
you see quantum invariants.} This statement 
describes only the tip of the iceberg. We 
believe something like the following is 
happening: Let $C$ be a
smooth projective curve over $\bC$, and 
\begin{equation}
\label{spectral-intro}
\xymatrix{
\Sigma \ar[dr]_{\pi}\ar[r]^{i} 
&\overline{T^*C}\ar[d]^{\pi}
\\
&C		}
\end{equation} 
be an arbitrary
 Hitchin \emph{spectral curve} associated
with a particular
  \emph{meromorphic} Higgs
bundle $(E,\phi)$ 
on $C$. Then the asymptotic expansion
at an essential singularity
of a solution (flat section)
 of the $\hbar$-connection on $C$,
that is the image of the \textbf{quantization} applied to 
$\Sigma$, carries the information of 
quantum invariants of a totally different
geometric structure, which should be considered
as the \textbf{mirror} to the geometric context
\eqref{spectral-intro}.

In this introduction,
we are going to tell you a story
of  an example to this
mirror correspondence using the rainbow
integral of Airy. The Hitchin spectral curve
is a singular compactification 
of a parabola $x=y^2$ in a Hirzebruch 
surface. 
The corresponding 
 quantum invariants, the ones  hiden
 underneath the rainbow,
 are the cotangent class intersection numbers
of the moduli space $\Mbar_{g,n}$. These 
numbers then determine the coefficients
of the tautological  relations among the 
generators of the tautological rings of 
$\cM_g$ and $\Mbar_{g,n}$. 
The uniqueness of the asymptotic expansion 
relates  the WKB analysis
of the quantization of the parabola at  infinity
 to the intersection 
 numbers  on $\Mbar_{g,n}$,
 through a combinatorial  estimate of the
 \emph{volume} of the moduli space
 $\cM_{g,n}$.

The story begins in 1838.

\subsection{On the other side of the rainbow}

\begin{figure}[hbt]
\includegraphics[height=1.5in]{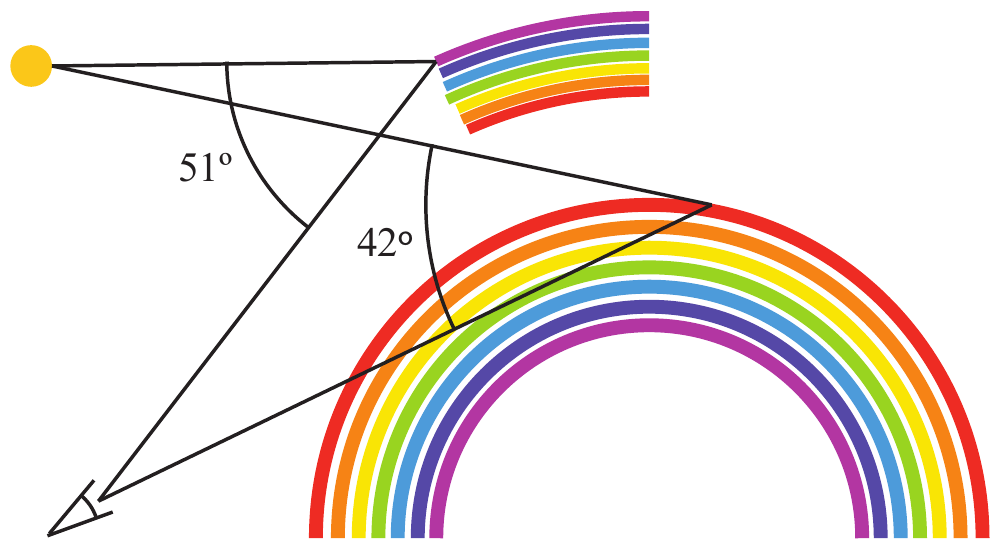}
\caption{Rainbow archs}
\label{fig:ranbow}
\end{figure}

Sir George Biddel Airy devised 
a simple formula, which he called 
the \emph{rainbow
integral} 
\be\label{rainbow}
Ai(x) = \frac{1}{2\pi} \int_{-\infty} ^\infty
e^{ipx}e^{i\frac{p^3}{3}}dp
\ee
and now carries his name,
in his attempt of explaining the rainbow phenomena
in terms of  wave optics \cite{Airy}. 
The angle between the sun and
the observer measured at the brightest arch is always
about $42^\circ$. The higher arches also have 
definite angles, independent of the rainbow. 
Airy tried to explain these angles and the 
brightness of the rainbow arches by
 the peak points of 
the rainbow integral.

\begin{figure}[hbt]
\includegraphics[width=3in]{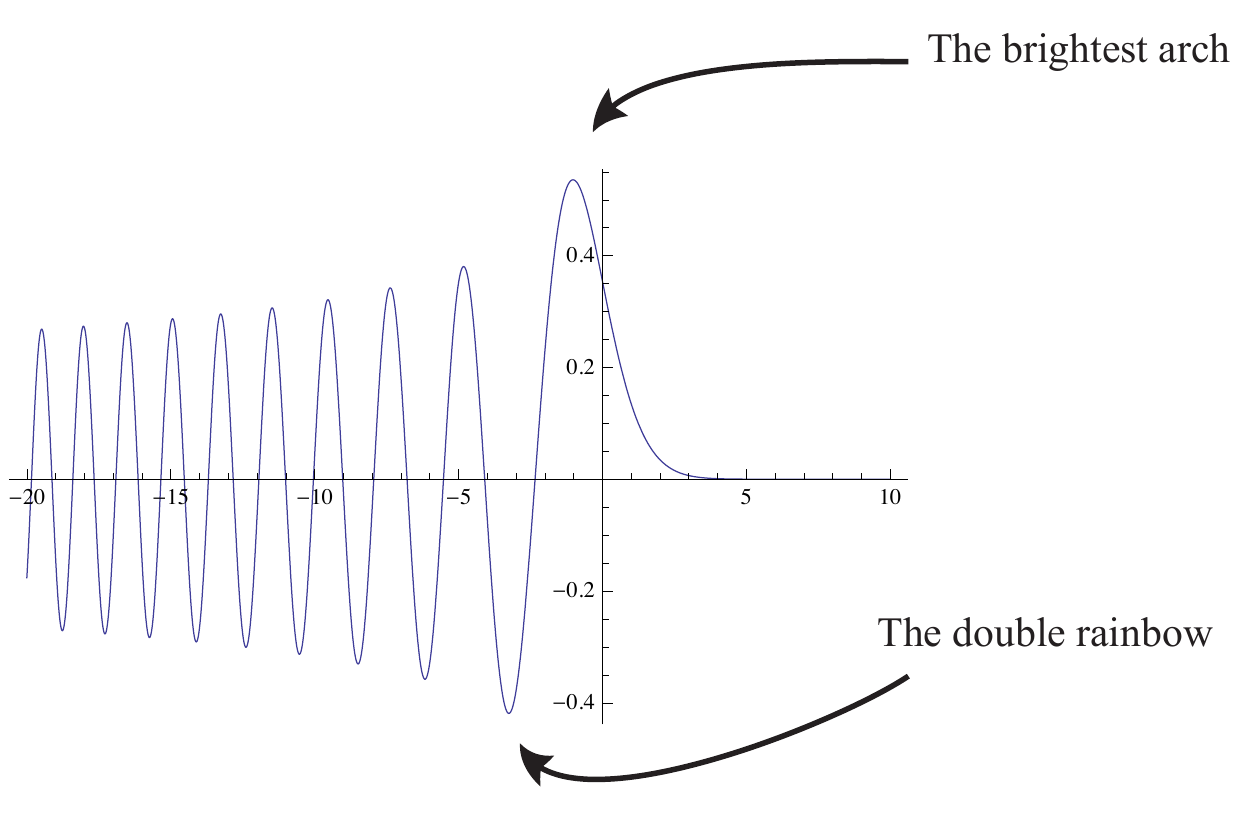}
\caption{The Airy function}
\label{fig:Airy}
\end{figure}

We note that \eqref{rainbow} is an 
\emph{oscillatory integral}, and determines a 
real analytic function in $x\in \bR$. It is easy to see,
by integration by parts and taking care of the 
boundary contributions in oscillatory 
integral, that 
\begin{align*}
\frac{d^2}{dx^2}Ai(x) 
&=
\frac{1}{2\pi} \int_{-\infty} ^\infty (-p^2)
e^{ipx}e^{i\frac{p^3}{3}}dp
=
\frac{1}{2\pi} \int_{-\infty} ^\infty 
e^{ipx}\left(i\frac{d}{dp}e^{i\frac{p^3}{3}}\right)dp
\\
&=
-\frac{1}{2\pi} \int_{-\infty} ^\infty 
\left(i\frac{d}{dp}e^{ipx}\right)e^{i\frac{p^3}{3}}dp
=
\frac{1}{2\pi}\; x \int_{-\infty} ^\infty e^{ipx}
e^{i\frac{p^3}{3}}dp.
\end{align*}
Thus the Airy function satisfies a second order
differential equation (known as the 
Airy differential equation)
\be\label{Airy diff}
\left(\frac{d^2}{dx^2}-x\right) Ai(x)=0.
\ee
Now we consider $x\in \bC$ as a complex variable.
Since the coefficients of  \eqref{Airy diff} are
entire functions (they are just $1$ and $x$),
any solution of this differential equation is automatically
entire, and has a convergent power series expansion
\be\label{Airy power}
Ai(x) = \sum_{n=0}^\infty a_n x^n
\ee
at the origin with the radius of convergence $\infty$.
Plugging \eqref{Airy power} into \eqref{Airy diff},
we obtain a recursion formula 
\be
\label{Airy coeff}
a_{n+2} = \frac{1}{(n+2)(n+1)} \; a_{n-1}, 
\qquad n\ge 0, 
\ee
with the initial condition $a_{-1} = 0.$
Thus we find
\begin{equation*}
 a_{3n} =a_0\cdot
\frac{\prod_{j=1}^n (3j-2)}{(3n)!},
\qquad
 a_{3n+1} =a_1\cdot
\frac{\prod_{j=1}^n (3j-1)}{(3n+1)!},
\qquad
 a_{3n+2} = 0.
\end{equation*}
Here, $a_0$ and $a_1$ are arbitrary constants,
so that the Airy differential equation has
a two-dimensional space of solutions. These coefficients
do not seem to be particularly interesting.
The
oscillatory integral \eqref{rainbow}
tells us that 
as $x\rar +\infty$ on the real axis,
the Airy function defined by the rainbow integral
vanishes, because  $e^{ipx+ip^3/3}$ oscillates so much
that the integral cancels. More precisely, $Ai(x)$
satisfies a limiting formula
\be
\label{airyasym}
\lim_{x\rar +\infty}
\frac{Ai(x)}{\frac{1}{2\sqrt{\pi}}
\cdot  \frac{1}{\sqrt[4]{x}}
\exp\left({-\frac{2}{3}{x^{\frac{3}{2}}}}\right)}
 = 1.
\ee
Hence it exponentially decays, as 
$x\rar +\infty$,
 $x\in \bR$.
Among the Taylor series solutions
\eqref{Airy power},  there is 
\emph{only one} solution
that satisfies this exponential decay
property,  which is given by the following
initial condition for \eqref{Airy coeff}:
$$
a_0 = \frac{1}{3^{\frac{2}{3}}\Gam(\frac{2}{3})},
\qquad 
a_1 = - \frac{1}{3^{\frac{1}{3}}\Gam(\frac{1}{3})}.
$$
The exponential decay on the positive real axis 
explains 
why we do not see any rainbows under the brightest
arch. Then what do we really
see underneath the rainbow? Or on the \emph{other 
side} of the rainbow?

The differential equation \eqref{Airy diff} tells us 
 that 
the Airy function has an essential singularity
at $x=\infty$. Otherwise, the solution
would be a polynomial in $x$, but 
\eqref{Airy coeff} does not terminate
at a finite $n$. How do we analyze the behavior
of a holomorphic function at its essential
singularity? And what kind of information
does it tell us?

\begin{Def}[Asymptotic expansion]
Let $f(z)$ be a holomorphic
function defined on an open domain $\Omega$
 of the complex plane ${\mathbb{C}}$ having the
origin $0$ on its boundary.
 A formal
power series
$$
\sum_{n=0} ^{\infty} a_{n} z^{n}
$$
is  an \emph{asymptotic expansion} of $f(z)$
on $\Omega$ at $z=0$ if
\begin{equation}
\label{def asymptotic}
\lim_{\substack{z\rightarrow 0 \\ z\in \Omega}}
\frac{1}
{z^{m+1}}
\left(
f(z) - \sum_{n = 0} ^m a_{n} z^{n}
\right)   = a_{m+1}
\end{equation}
holds for every $m\ge 0$.
\end{Def}

The asymptotic expansion is a
domain specific notion. For example,
$f(z) = e^{-1/z}$ is holomorphic
on $\bC^* = \bC\setminus \{0\}$, 
but it does not have any asymptotic 
expansion on all of $\bC^*$. 
However, it has an asymptotic expansion
$$
e^{-1/z} \sim 0
$$
on 
$$
\Omega=\left\{z\in \bC^*\;\left|\; |\Arg(z)|<
\frac{\pi}{2}-\epsilon\right.\right\},
\qquad \epsilon>0.
$$
If there is an asymptotic expansion of $f$ on 
a domain $\Omega$, then it is unique, 
and if $\Omega'\subset \Omega$ with 
$0\in \partial \Omega'$, then 
obviously the same asymptotic
expansion holds on $\Omega'$.

The Taylor expansion \eqref{Airy power}
shows that  
$Ai(x)$ is real valued on the real axis, and from
\eqref{airyasym}, we see that the value is
positive for $x>0$. Therefore, $\log Ai(x)$
is a holomorphic function on $Re(x)>0$.

\begin{thm}[Asymptotic expansion of the Airy function]
Define \be\label{Airy 0,1}
S_0(x) = - \frac{2}{3} x^{\frac{3}{2}},
\qquad
S_1(x) = -\frac{1}{4} \log x-\log (2\sqrt{\pi}).
\ee
Then $\log Ai(x) - S_0(x)-S_1(x)$
has the following asymptotic expansion on
$Re(x)>0$.
\be
\label{logAiry asym}
\log Ai(x) - S_0(x)-S_1(x) = 
\sum_{m=2}^\infty
S_m(x),
\ee
\begin{equation}
\label{Airy Sm}
S_m(x) :=x^{-\frac{3}{2}(m-1)} \cdot
\frac{1}{2^{m-1}}
\sum_{\substack{g\ge 0,n>0\\2g-2+n=m-1}} 
\frac{(-1)^n}{n!}
\sum_{\substack{d_1+\dots+d_n\\
=3g-3+n}}
\la \tau_{d_1}\cdots \tau_{d_n}\ra_{g,n}
\prod_{i=1}^n |2d_i-1|!!
\end{equation}
for $m\ge 2$.
The key coefficients are defined by
\be\label{intersection}
\la \tau_{d_1}\cdots \tau_{d_n}\ra_{g,n}
:=\int_{\Mbar_{g,n}}
\psi_1^{d_1}\cdots \psi_n^{d_n},
\ee
where, $\Mbar_{g,n}$ is the moduli space of
stable curves of genus $g$ with $n$ smooth
marked
points. Let $[C,(p_1,\dots,p_n)] \in \Mbar_{g,n}$
be a point of the moduli space. 
We can construct a line bundle $\bL_i$
on the smooth
Deligne-Mumford stack $\Mbar_{g,n}$ by
attaching the cotangent line $T^*_{p_i}C$
at the point $[C,(p_1,\dots,p_n)]$ of the moduli
space. The symbol
$$
\psi_i=c_1(\bL_i) \in H^2(\Mbar_{g,n},\bQ)
$$ 
denotes its first Chern class.
Since $\Mbar_{g,n}$ has dimension $3g-3+n$, 
the integral \eqref{intersection} is automatically
$0$ unless $d_1+\dots+d_n
=3g-3+n$.
\end{thm}

Surprisingly,  on the other side of the rainbow, i.e.,
when $x>0$, 
we see the intersection numbers 
\eqref{intersection}!

\begin{rem}
The relation between the Airy function and 
the intersection numbers was discovered by
Kontsevich \cite{K1992}. He replaces the 
variables $x$ and $p$ in \eqref{rainbow}
by \emph{Hermitian matrices}. It is a general
property that a large class of Hermitian
matrix integrals satisfy an integrable system
of KdV and KP type
(see, for example, \cite{M1994}). 
Because of the cubic polynomial
in the integrand, the matrix Airy function
of \cite{K1992} counts trivalent ribbon graphs
through the Feynman diagram expansion,
which represent open dense subsets 
of $\Mbar_{g,n}$. By identifying the intersection 
numbers and the Euclidean volume of these open
subsets, Kontsevich proves the Witten conjecture
\cite{W1991}. Our formulas \eqref{logAiry asym}
and \eqref{Airy Sm} are a consequence of
the results reported in \cite{CMS,DMSS, MP2012,MS}.
We will explain the relation 
more concretely in these lectures.
\end{rem}

\begin{rem}
The numerical value of the
asymptotic expansion \eqref{logAiry asym}
is given by
\begin{multline}
\label{Ai numerical}
\log Ai(x) 
= -\frac{2}{3} x^{\frac{3}{2}}
-\frac{1}{4}\log x -\log(2\sqrt{\pi})
\\
-\frac{5}{48} x^{-\frac{3}{2}}
+\frac{5}{64}x^{-3}
-\frac{1105}{9216}x^{-\frac{9}{2}}
+\frac{565}{2048}x^{-6}
-\frac{82825}{98304}x^{-\frac{15}{2}}
+\frac{19675}{6144}x^{-9}
\\
-\frac{1282031525}{88080384} x^{-\frac{21}{2}}
+\frac{80727925}{1048576}x^{-12}
-\frac{1683480621875}{3623878656}
x^{-\frac{27}{2}}
+\cdots .
\end{multline}
This follows from the WKB analysis of the
Airy differential equation, which will be
explained  in this introduction. 
\end{rem}

\begin{rem}
Although the asymptotic expansion is not
\emph{equal} to the holomorphic function itself,
we use the equality sign in these lectures to avoid
further cumbersome notations.
\end{rem}

The Airy differential equation appears in 
many different places, 
showing the feature of a \emph{universal object} in the 
WKB analysis. It reflects the fact that
the intersection numbers 
\eqref{intersection} are the most fundamental 
objects in Gromov-Witten theory. 
In  contrast to the Airy differential equation,
the gamma function is a universal object 
in the context of \emph{difference equations}. 
We recall that 
$$
\Gam(z+1) = z\Gam(z), 
$$
and 
its asymptotic expansion for $Re(z)>0$
is given by
\be
\label{Stirling}
\log \Gam(z) = 
z\log z-z-\half \log z +\half \log(2\pi) 
+ \sum_{m=1}^\infty \frac{B_{2m}}{2m(2m-1)}
z^{-(2m-1)},
\ee
where $B_{2m}$ is the $(2m)$-th
Bernoulli number defined by the generating function
$$
\frac{x}{e^x-1}=\sum_{n=0}^\infty \frac{B_n}{n!}
x^n.
$$
This  is called
 \textbf{Stirling's formula}, and its main part gives 
 the well-known approximation of the factorial:
 $$
 n! \sim \sqrt{2\pi n}\; \frac{n^n}{e^n}.
 $$
The asymptotic expansion of the gamma
function is deeply related to the moduli
theory of algebraic curves. For example,
from 
Harer and Zagier \cite{HZ} we learn that
the orbifold Euler characteristic of the
moduli space of smooth algebraic curves
is given by the formula
\be
\label{Mgn Euler}
\rchi(\cM_{g,n})=(-1)^{n-1}
\frac{(2g-3+n)!}{(2g-2)!n!}\zeta(1-2g).
\ee 
Here, the special value of the Riemann zeta function
is the Bernoulli number
$$
\zeta(1-2g) = -\;\frac{B_{2g}}{2g}.
$$
The expression \eqref{Mgn Euler} is valid for
$g=0, n\ge 3$ if we use the gamma function 
for $(2g-2)!$.
Stirling's formula \eqref{Stirling} 
is  much simpler  than 
the expansion of $\log Ai(x)$. 
As the prime factorization of one of the coefficients
$$
\frac{1683480621875}{3623878656}
= \frac{5^5\cdot 13\cdot 17\cdot 2437619}
{2^{27}\cdot 3^3}
$$
shows, we do not expect any simple closed formula
for the coefficients
of  \eqref{Ai numerical}, like Bernoulli numbers.
Amazingly, still there is a  close relation between
these two asymptotic expansions
\eqref{Ai numerical} and \eqref{Stirling}
through the work on \emph{tautological
relations} of Chow classes on the moduli
space $\cM_g$ by Mumford \cite{Mumford},
followed by recent exciting developments on the
Faber-Zagier conjecture
\cite{F, I, PP}. In Theorem~\ref{thm:FgnC}, we will
see yet another close relationship between 
the Euler characteristic of $\cM_{g,n}$ and
the intersection numbers on $\Mbar_{g,n}$,
through two special values of the \emph{same}
 function.

The asymptotic expansion of the Airy
function $Ai(x)$ itself for 
$Re (x)>0$ has actually a rather simple expression:
\be
\label{Ai asym}
\begin{aligned}
Ai(x) &= \frac{e^{-\frac{2}{3}x^{\frac{3}{2}}}}
{2\sqrt{\pi} x^{\frac{1}{4}}}
\sum_{m=0}^\infty 
\frac{\big(-\frac{3}{4}\big)^m
\Gam\big(m+\frac{5}{6}\big)
\Gam\big(m+\frac{1}{6}\big)}
{2\pi m!}\; x^{-\frac{3}{2}m}
\\
&=\frac{e^{-\frac{2}{3}x^{\frac{3}{2}}}}
{2\sqrt{\pi}x^{\frac{1}{4}}}
\sum_{m=0}^\infty 
(-1)^m \left(\frac{1}{576}\right)^m
\frac{(6m)!}{(2m)!(3m)!} \;x^{-\frac{3}{2}m}.
\end{aligned}
\ee
The expansion in terms of the
gamma function values  of the first line
of \eqref{Ai asym} naturally arises from a
hypergeometric function. 
The first line is equal to the second line
because 
$$
\Gam(z) \Gam(1-z) = \frac{\pi}{\sin \pi z},
$$
and  induction on $m$.
Since the $m=0$ term in the 
summation is $1$, we can 
apply the formal logarithm expansion
$$
\log(1-X) = -\sum_{j=1}^\infty \frac{1}{j}X^j
$$
 to  \eqref{Ai asym} with
 $$
 X = -\sum_{m=1}^\infty 
(-1)^m \left(\frac{1}{576}\right)^m
\frac{(6m)!}{(2m)!(3m)!} \;x^{-\frac{3}{2}m},
 $$
and obtain
\begin{multline}
\label{Ai gamma}
\log Ai(x) 
= -\frac{2}{3} x^{\frac{3}{2}}
-\frac{1}{4}\log x -\log(2\sqrt{\pi})
\\
-\frac{1}{576}\;\frac{6!}
{2!3!} x^{-\frac{3}{2}}
+
\left(\frac{1}{576}\right)^2
\left(
\frac{12!}{4!6!}
-\half \left(
\frac{6!}
{2!3!}
\right)^2
\right)
x^{-3}
\\
-\left(\frac{1}{576}\right)^3
\left(
\frac{18!}{6!9!} 
-\frac{12!}{4!6!} \cdot
\frac{6!}
{2!3!}
+\frac{1}{3} \left(
\frac{6!}
{2!3!}
\right)^3
\right)
x^{-\frac{9}{2}}+\cdots.
\end{multline}
In general, for $m\ge 1$, we have
\be
\label{Sm in gamma}
S_{m+1}(x) = (-1)^{m} x^{-\frac{3}{2}m}
\left(\frac{1}{576}\right)^m
\sum_{\lam\vdash m}
(-1)^{\ell(\lam)-1}\frac{\big(\ell(\lam)-1\big)!}
{|\Aut(\lam)|} \prod_{i=1}^{\ell(\lam)}
\frac{(6\lam_i)!}{(2\lam_i)!(3\lam_i)!},
\ee
where $\lam$ is a partition of $m$, $\ell(\lam)$
its length,  and $\Aut(\lam)$ is the group of 
permutations of the parts of $\lam$
of equal length.  

Comparing  \eqref{Airy Sm}
 and \eqref{Sm in gamma}, 
we establish concrete  relations among the intersection 
numbers.

\begin{thm}[Rainbow formula]
The cotangent class intersection numbers
of $\Mbar_{g,n}$ 
satisfy the following relation for every $m\ge 1:$
\be
\label{rainbow formula}
\begin{aligned}
&\sum_{\substack{g\ge 0,n>0\\2g-2+n=m}} 
\frac{1}{n!}
\sum_{\substack{d_1+\dots+d_n\\
=3g-3+n}}
\la \tau_{d_1}\cdots \tau_{d_n}\ra_{g,n}
\prod_{i=1}^n |2d_i-1|!!
\\
&=
\left(\frac{1}{288}\right)^{m}
\sum_{\lam\vdash m}
(-1)^{\ell(\lam)-1}\frac{\big(\ell(\lam)-1\big)!}
{|\Aut(\lam)|} 
\prod_{i=1}^{\ell(\lam)}
\frac{(6\lam_i)!}{(2\lam_i)!(3\lam_i)!}.
\end{aligned}
\ee
\end{thm}

Here, we use the fact that $(-1)^n = (-1)^{m}$ if
$2g-2+n = m$.
For example, for $m=2$, we have
$$
\frac{1}{6}\la \tau_0^3\tau_1\ra_{0,4}
+
\half \la \tau_1^2\ra_{1,2} 
+ 3\la \tau_0\tau_2\ra_{1,2}
=
\left(\frac{1}{288}\right)^2
\left(\frac{12!}{4!6!}-\half 
\left(\frac{6!}{2!3!}\right)^2
\right)
= \frac{5}{16}.
$$
This can be verified by evaluating
$$
\la \tau_0^3\tau_1\ra_{0,4}=1,
\qquad \la \tau_1^2\ra_{1,2} 
=\la \tau_0\tau_2\ra_{1,2}=\frac{1}{24}.
$$

\begin{rem}
The main purpose of these lectures is to 
relate the topological recursion of \cite{EO2007}
and quantization of Hitchin spectral curves.
The left-hand side of \eqref{rainbow formula}
represents the topological recursion in this 
example, since the intersection numbers 
can be computed through this mechanism,
as explained below. Actually, this is an important
example that leads to the universal structure 
of the topological recursion. The right-hand side
is the asymptotic expansion of a function that is
coming from the geometry of the
Hitchin spectral curve of a Higgs bundle.
\end{rem}

The structure of the cohomology ring (or more 
fundamental Chow ring) of the moduli space
$\Mbar_{g,n}$, and its open part 
$\cM_{g,n}$ consisting of smooth curves, attracted
much attention since the publication of
the inspiring paper by Mumford
\cite{Mumford} mentioned above. 
Let us focus on a simple situation
$$
\pi:\cM_{g,1}\lrar \cM_g,
$$
which  \emph{forgets} the marked point on 
a smooth curve. By gluing the canonical
line bundle of the fiber of each
point on the base $\cM_g$, which is the
curve represented by the point on the moduli,
 we obtain the
relative dualizing sheaf $\omega$ on 
$\cM_{g,1}$. Its first Chern class, 
considered as a divisor on $\cM_{g,1}$
and an element of
the Chow group $A^1(\cM_{g,1})$, is denoted
by $\psi$. 
In the notation of \eqref{intersection}, 
this is the same as $\psi_1$.
Mumford defines 
\emph{tautological} classes
$$
\kappa_a := \pi_*(\psi^{a+1})\in A^a(\cM_g).
$$
One of the statements
of the Faber-Zagier conjecture of \cite{F},
now a theorem due to Ionel \cite{I}, 
says
 the following.

\begin{conj}[A part of the
Faber-Zagier Conjecture \cite{F}]
Define rational numbers $a_j\in \bQ$ by
\be
\sum_{j=1}^\infty a_j t^j
=-\log \left(
\sum_{m=0}^\infty
\frac{(6m)!}{(2m)!(3m)!} t^m
\right).
\ee
Then the coefficient of $t^\ell$ of the
expression 
$$
\exp\left(
\sum_{j=1}^\infty a_j \kappa_j t^j
\right)
\in \big(\bQ[\kappa_1,\kappa_2,\dots]\big)[[t]]
$$
for each $\ell\ge 1$ gives the unique
codimension $\ell$ tautological relation
among the $\kappa$-classes on the 
moduli space $\cM_{3\ell -1}$.
\end{conj}

We see from \eqref{rainbow formula},
these coefficients $a_j$'s are given by
the intersection numbers \eqref{intersection},
by a change of the variables
$$
t=- \frac{1}{576}\; x^{-\frac{3}{2}}.
$$
Indeed, we have
\be
\label{FZ conj}
a_j = - 288^j
\sum_{\substack{g\ge 0,n>0\\2g-2+n=j}} 
\frac{1}{n!}
\sum_{\substack{d_1+\dots+d_n\\
=3g-3+n}}
\la \tau_{d_1}\cdots \tau_{d_n}\ra_{g,n}
\prod_{i=1}^n |2d_i-1|!!\; .
\ee
These tautological relations are generalized
for the moduli spaces $\Mbar_{g,n}$,
and are proved in \cite{PP}. Amazingly, 
still the rainbow 
integral \eqref{rainbow} plays the essential 
role in determining the tautological 
relations in this generalized work.

\begin{rem}
As we have seen above, the asymptotic 
expansions of the gamma function and the Airy
function carry information of quantum invariants,
in particular, certain topological information 
of $\cM_{g,n}$ and $\Mbar_{g,n}$. 
We note that these quantum invariants are stored
in the \textbf{insignificant} part of the 
asymptotic expansions. 
\end{rem}

Here come questions.

\begin{quest}
\label{quest: intersection}
The Airy function is a one-variable function.
It cannot be a generating function of 
all the intersection numbers \eqref{intersection}.
Then how do we obtain all intersection 
numbers from the Airy function, or the Airy 
differential equation, alone?
\end{quest}

\begin{quest}
\label{quest: Higgs}
The relations between the Airy function, the
gamma function, and intersection numbers
are all great. 
But then how does this relation have anything
to do with \textbf{Higgs bundles}?
\end{quest}

\begin{quest}
\label{quest: spectral}
As we remarked, the information 
of the quantum invariants is stored in the 
insignificant part of the asymptotic expansion.
In the Airy example, they correspond to
$S_m(x)$ of \eqref{Airy Sm} for $m\ge 2$. 
Then what does the main part of the asymptotic 
behavior of the function,  i.e.,  those functions in
\eqref{Airy 0,1}, determine?
\end{quest}

As we have remarked earlier,
Kontsevich \cite{K1992} utilized  matrix integral 
techniques 
to answer Question~\ref{quest: intersection}. 
The key idea  is to replace the 
variables in \eqref{rainbow} by Hermitian
matrices, and then use the asymptotic expansion on
 the result. Through the Feynman diagram expansion,
 he was able to obtain a generating function 
 of all the intersection numbers.
 
 What we explain in these lectures is the 
 concept of \textbf{topological recursion}
 of \cite{EO2007}. Without going into 
 matrix integrals, we can directly obtain 
 (a totally different set of)
 generating functions of the intersection numbers
 from the Airy function. 
 Here, the Airy differential equation is 
 identified as a \textbf{quantum curve},
 and  application of the \textbf{semi-classical limit}
 and the topological recursion enable us to 
 calculate generating functions of the intersection 
 numbers.

 But before going into  detail, let us 
 briefly answer Question~\ref{quest: Higgs}
 below. The point is that 
 the geometry of the Airy function is a special
 example of Higgs bundles.

 For the example of the Airy
 differential equation, the topological 
 recursion is exactly the same as the
 \emph{Virasoro constraint conditions}
 for the intersection numbers 
 \eqref{intersection}, and the 
 semi-classical limit recovers the
 \textbf{Hitchin spectral curve} of the 
 corresponding Higgs bundle.
 The information stored in the main part
 of the asymptotic expansion \eqref{Airy 0,1}, 
 as asked in Question~\ref{quest: spectral},
 actually determines the
 spectral curve and its geometry. 
 We can turn the story in the
 other way around: we will see that the functions
 $S_0(x)$ and $S_1(x)$ corresponding to
 \eqref{Airy 0,1} in the general context
 are indeed determined by the geometry
 of the Hitchin spectral curve
 of an appropriate Higgs bundle.

The stage setting  
is the following. As a base curve,
we have $\bP^1$. On this curve we have a
 vector bundle 
 \be
 \label{E}
 E =  \cO_{\bP^1}(-1)\dsum\cO_{\bP^1}(1) 
 = 
K_{\bP^1}^{\half}\dsum K_{\bP^1}^{-\half}
\ee
 of rank $2$. The main character of the Second Act is 
 a \emph{meromorphic} Higgs field
\begin{equation}
\label{Airy Higgs}
\phi = \begin{bmatrix}
& x(dx)^2\\
1
\end{bmatrix}:
E\lrar K_{\bP^1}(5)\tensor E.
\end{equation}
Here, $x$ is an affine coordinate of 
$\bA^1\subset 
\bP^1$, $1$ on the $(2,1)$-component of 
$\phi$ is the natural morphism
$$
1:K_{\bP^1}^{\half}\overset{=}{\lrar}
K_{\bP^1}^{\half}\lrar K_{\bP^1}^{\half}\tensor
\cO_{\bP^1}(5),
$$
and 
$$
x(dx)^2\in 
H^0\big(\bP^1,K_{\bP^1}^{\tensor 2}(5-1)\big)
\isom \bC
$$ 
is 
the unique (up to a constant factor) meromorphic
\emph{quadratic differential}
 on $\bP^1$ that has one zero at $x=0$
and a pole of order $5$ at $x=\infty$.
We use $K_C$ to denote the canonical sheaf on 
a projective algebraic curve $C$.
The data $(E,\phi)$ is called  a \emph{Higgs pair}.
Although $\phi$ contains a quadratic
differential in its component, because of the 
shape of the vector bundle $E$, we see that
$$
E = 
K_{\bP^1}^{\half} \dsum
K_{\bP^1}^{-\half}
\overset{\phi}{\lrar}
\left(K_{\bP^1}^{\frac{3}{2}}\dsum
K_{\bP^1}^{\half}
\right)\tensor \cO_{\bP^1}(5)
=
K_{\bP^1}(5)\tensor \left(K_{\bP^1}^{\half}\dsum
K_{\bP^1}^{-\half}\right),
$$
hence
$$
\phi\in H^0\big(\bP^1,K_{\bP^1}(5)\tensor \End(E)
\big)
$$
is indeed an $\End(E)$-valued meromorphic 
$1$-form on $\bP^1$. 

The cotangent 
bungle 
$$
\pi:T^*\bP^1\lrar \bP^1
$$
 is the total space
of $K_{\bP^1}$. Therefore, the pull-back bundle
$\pi^*K_{\bP^1}$ has a tautological section
$\eta\in H^0(T^*\bP^1,\pi^*K_{\bP^1})$,
which is a globally defined
holomorphic $1$-form on 
$T^*\bP^1$. 
The global holomorphic $2$-form
$-d\eta$ gives the 
 holomorphic symplectic structure,
 hence a hyper-K\"ahler structure, on 
$T^*\bP^1$.
If we trivialize the cotangent 
bundle on the affine  neighborhood of
$\bP^1$ with a coordinate $x$, 
and use a fiber coordinate $y$, then 
$\eta = ydx$. We wish to define the spectral curve
of this Higgs pair. Due to the fact that
$\phi$ is singular at $x=\infty$, we cannot capture
the whole story within the cotangent bundle.
We note that the cotangent bundle $T^*\bP^1$ 
has a natural
compactification
$$
\overline{T^*\bP^1} :=
\bP(K_{\bP^1}\dsum \cO_{\bP^1}) =\bF_2
\overset{\pi}{\lrar} \bP^1,
$$
which is  known as a Hirzebruch surface.
The holomorphic $1$-form $\eta$ extends to 
$\overline{T^*\bP^1}$ as a meromorphic $1$-form
with simple poles along the divisor at 
infinity.

Now we can consider the characteristic polynomial
$$
\det(\eta - \pi^*\phi) \in 
H^0\left(\overline{T^*\bP^1},
\pi^*\big(K_{\bP^1}^{\tensor 2}(5)\big)
\right)
$$
as a meromorphic section of the line bundle
$\pi^* K_{\bP^1} ^{\tensor 2}$ on the compact
space $\overline{T^*\bP^1}$.
It defines the Hitchin spectral curve
\be
\label{Sigma Airy}
\Sigma = \big(\det(\eta - \pi^*\phi)\big)_0
\subset \overline{T^*\bP^1}
\ee
as a divisor. Again in terms of the local coordinate
$(x,y)$ of $\overline{T^*\bP^1}$, the spectral 
curve $\Sigma$ is simply given by
\be\label{x=y^2}
x=y^2.
\ee
It is a perfect parabola in the $(x,y)$-plane. But
our $\Sigma$ is in the Hirzebruch surface,
not in the projective plane.  
Choose the coordinate $(u,w)\in \bF_2$ 
around $(x,y)=(\infty,\infty)$ defined by
\be
\label{uw}
\begin{cases}
u = \frac{1}{x}\\
\frac{1}{w}du = ydx
\end{cases}.
\ee
Then the local expression of $\Sigma$ around 
$(u,w)=(0,0)$ becomes a quintic cusp 
\be
\label{quintic cusp}
w^2=u^5.
\ee
So the spectral curve $\Sigma$ is indeed highly
singular at  infinity!

\subsection{Quantum curves, semi-classical
limit, and the WKB analysis}

At this stage we have
come to the point to 
introducing the notion of \emph{quantization}. We wish
to quantize the spectral curve $\Sigma$
of \eqref{Sigma Airy}. 
In terms of the affine coordinate $(x,y)$, the
\textbf{quantum curve} of $y^2-x=0$ should be
the Airy differential equation
\be\label{Airy quantum}
\left(\left(\hbar \frac{d}{dx}\right)^2 -x\right)
\Psi(x,\hbar) = 0.
\ee
This is the Weyl quantization, in 
which we change the commutative algebra
$\bC[x,y]$ to a Weyl algebra $\bC[\hbar]\la x,y\ra$
defined by the commutation relation
\be\label{commutation}
[x,y]=-\hbar.
\ee
We consider $x\in \bC[\hbar]\la x,y\ra$ as the 
multiplication operator by the coordinate $x$,
and $y = \hbar\frac{d}{dx}$ as a differential 
operator. 

How do we know that  
\eqref{Airy quantum} is the right quantization
of the spectral curve \eqref{x=y^2}? 
Apparently, the limit $\hbar \rar 0$ of the
differential operator does not reproduce
the spectral curve.
Let us now recall the WKB method for 
analyzing 
differential equations like \eqref{Airy quantum}.
This is a method that relates classical mechanics
and quantum mechanics. As we see below,
the WKB method is not for finding a
convergent  analytic 
solution to the differential equation.
Since the equation we wish to solve is
considered to be a quantum equation,  the
corresponding classical problem, if it exists,
should be recovered by taking $\hbar \rar 0$.
We denote by an unknown  function $S_0(x)$ 
the  ``solution'' to the 
corresponding classical 
problem.
To emphasize the classical behavior
 at the $\hbar \rar 0$ limit,
we expand the solution to the quantum 
equation as
\be\label{Psi expansion}
\Psi(x,\hbar) = \exp\left(
\sum_{m=0} ^\infty \hbar^{m-1}S_m(x)
\right):=
\exp\left(
\frac{1}{\hbar} S_0(x)
\right)
\cdot
\exp\left(
\sum_{m=1} ^\infty \hbar^{m-1}S_m(x)
\right).
\ee
The idea is
that  as $\hbar\rar 0$, the effect of
$S_0(x)$ is magnified. But as
a series in $\hbar$, \eqref{Psi expansion}
is ill defined because the coefficient of 
each power of $\hbar$ is an infinite sum.
It is also clear that $\hbar\rar 0$ does not
make sense for $\Psi(x,\hbar)$.
Instead of expanding \eqref{Psi expansion}
immediately in $\hbar$ and take its $0$ limit, 
we use the following 
 standard procedure. 
First we note that \eqref{Airy quantum}
is equivalent to 
$$
\left[
\exp\left(
-\frac{1}{\hbar} S_0(x)
\right)\cdot 
\left(\left(\hbar \frac{d}{dx}\right)^2 -x\right)
\cdot
\exp\left(
\frac{1}{\hbar} S_0(x)
\right)
\right]
\exp\left(
\sum_{m=1} ^\infty \hbar^{m-1}S_m(x)
\right)
=
0.
$$
Since
the conjugate differential 
operator
\begin{multline*}
\exp\left(
-\frac{1}{\hbar} S_0(x)
\right)\cdot 
\left(\left(\hbar \frac{d}{dx}\right)^2 -x\right)
\cdot
\exp\left(
\frac{1}{\hbar} S_0(x)
\right)
\\
=
\left(\hbar \frac{d}{dx}\right)^2
+2\hbar S_0'(x) \frac{d}{dx}+  \big(S_0'(x)\big)^2
-x+\hbar S_0''(x)
\end{multline*}
is a well-defined differential operator, 
\emph{and} its limit $\hbar\rar 0$ makes sense,
we interpret \eqref{Airy quantum}
as the following differential equation:
\be
\label{conjugate}
\left[
\left(\hbar \frac{d}{dx}\right)^2
+2\hbar S_0'(x) \frac{d}{dx} + \big(S_0'(x)\big)^2
-x+\hbar S_0''(x)
\right]
\exp\left(
\sum_{m=1} ^\infty \hbar^{m-1}S_m(x)
\right)
=0.
\ee
Here, $'$ indicates the $x$-derivative.
This equation is equivalent to 
\be\label{WKB Airy}
\left(\sum_{m=0} ^\infty \hbar^m S_m'(x)
\right)^2 + \sum_{m=0} ^\infty \hbar^{m+1}
 S_m''(x) -x =0
\ee
for every $m\ge 0$.
The coefficient of the $\hbar^0$,
or the $\hbar\rar 0$ limit of
 \eqref{WKB Airy}, then gives
\be\label{S0}
\big(S_0'(x)\big)^2 - x=0,
\ee
and that of $\hbar^1$ gives
\be\label{S1}
S_0''(x) + 2S_0'(x)S_1'(x)=0.
\ee
The $\hbar^0$ term is what we call
the \textbf{semi-classical limit} of the
differential equation \eqref{Airy quantum}.
From \eqref{S0} we obtain
\be
\label{S0 solved}
S_0(x) = \pm \frac{2}{3} x^{\frac{3}{2}} +
c_0,
\ee
with a constant of integration $c_0$. 
Then plugging $S_0(x)$ into \eqref{S1}
we obtain 
$$
S_1(x) = -\frac{1}{4}\log x -\log(2\sqrt{\pi})+c_1,
$$
again with a constant of integration $c_1$.
Note that these solutions
are  consistent with \eqref{Airy 0,1}.
For $m\ge 1$, the coefficient 
of  $\hbar^{m+1}$
 gives
\be\label{Sm'}
S_{m+1}'(x) = -\frac{1}{2S_0'(x)}
\left(
S_m''(x)+\sum_{a=1}^{m}S_a'(x)S_{m+1-a}'(x)
\right),
\ee
which can be solved recursively, term by term
from $S_0(x)$. This mechanism is the
method of Wentzel-Kramers-Brillouin (WKB)
approximation.

We can ignore the constants of integration 
when solving \eqref{Sm'}
because it is easy to
restore them, if necessary,
just by adding $c_m$ to each $S_m(x)$
in \eqref{Psi expansion}. The solution then simply
changes to another one
$$
\left(\exp\left(\frac{1}{\hbar}c_0\right)
\exp\left(\frac{1}{\hbar}S_0(x)\right)
\right)\cdot
\exp\left(\sum_{m=1}^\infty h^{m-1}c_m\right)
\exp\left(\sum_{m=1}^\infty h^{m-1}S_m(x)\right).
$$
In terms of the main variable $x$, the above
solution is a constant multiple of the original
one. The two choices of the
 sign in \eqref{S0 solved} lead to  
two linearly independent
solutions of \eqref{Airy quantum}.
 If we impose 
\be
\label{boundary}
\lim_{x\rar \infty}S_m(x)=0, \qquad m\ge 2,
\ee
then the differential equation \eqref{Sm'}
uniquely determines
all terms $S_m(x)$. Thus, with the choice
of the negative sign in \eqref{S0 solved}
and imposing $c_0=c_1=0$
and \eqref{boundary}, we obtain
the unique  exponentially
decaying solution for $x\rar \infty$ along the 
real line. This solution 
 agrees with 
\eqref{Airy Sm} and \eqref{Sm in gamma}.
Thus we obtain the second line of the
Rainbow formula \eqref{rainbow formula}.

We also see from the semi-classical 
limit \eqref{S0} that if we put
$y=S_0'(x)$, then  we recover the
Hitchin spectral curve
$x=y^2$. The functions $S_m(x)$  actually
live on the  spectral curve rather than the 
base $\bP^1$, because of the appearance 
of $\sqrt{x}$ in \eqref{Airy Sm}.

\begin{rem} One can ask a question: 
\textbf{Does \eqref{Psi expansion} give a
convergent solution?} The answer is
a flat \textbf{No!} 
Suppose we solve the Airy differential 
equation with the WKB method explained
above, and define a ``solution'' by 
\eqref{Psi expansion}. Expand the second
 exponential factor as a 
power series in $\hbar$, and write the solution as
$$
\Psi(x,\hbar) = \exp\left(\frac{1}{\hbar}S_0(x)
\right)
\sum_{n=0}^\infty f_n(x) \hbar^n.
$$
Then for any compact subset
 $K\subset \bC\setminus \{0\}$,
 there is a constant $C_K$ such that
 $$
 \sup_{x\in K}|f_n(x)|
 \le C_K^n n!.
 $$
 Therefore, unless we are in an 
 extremely special case,
 the second exponential factor 
 in the expression \eqref{Psi expansion}
 does not converge as a power series in 
 $\hbar$ at all! Since 
 $$
 \Psi(x,\hbar):= Ai\left(x/\hbar^{\frac{2}{3}}\right)
 $$
 is a solution that is entire in $x$ and any $\hbar$
 for which $1/\hbar^{\frac{2}{3}}$ makes sense, 
 the WKB method around $x\in \bC\setminus \{0\}$
 is the same as the asymptotic expansion of 
 the Airy function $Ai(x)$ given in 
 \eqref{Ai asym}. There we see the factorial growth
 of the coefficients. Thus the WKB method is 
 \textbf{not} for finding a
 convergent analytic solution.
\end{rem}

\subsection{The topological recursion 
as quantization}

Then what is good about the WKB method and
the purely formal solution \eqref{Psi expansion}?
Let us examine \eqref{Airy Sm}.
We note that $S_m(x)$s are one variable functions, 
and  different values of $g$ and $n$ 
are summed  in its definition.
Therefore, knowing
the solution $\Psi(x,\hbar)$ of \eqref{Airy quantum}
that decays exponentially as $x\rar \infty$ along
the real axis, assuming $\hbar>0$, does not
seem to 
possibly recover  
 intersection numbers $\la \tau_{d_1}\cdots
\tau_{d_n}\ra_{g,n}$ for all values of
$(d_1,\dots,d_n)$ and $(g,n)$. 
Then how much information does
 the quantum curve
\eqref{Airy quantum} really have?

Here comes the idea of \textbf{topological recursion}
of Eynard and Orantin \cite{EO2007}.
This mechanism gives a refined expression 
of each $S_m(x)$, and computes all
intersection numbers. The solution 
$\Psi$ of \eqref{Psi expansion}
is never holomorphic, and it makes sense 
only as the asymptotic expansion of
a holomorphic solution at its essential singularity. 
The expansion of the Airy function $Ai(x)$ 
at a holomorphic point does not carry any
interesting information. The function's key 
information is concentrated in the expansion 
at the essential singularity. 
The topological recursion is for obtaining this
hidden information 
when applied at the essential singularity 
of the solution, by giving an explicit
formula for the WKB expansion.
And the WKB analysis
 is indeed a method that determines
 the relation between the quantum 
 behavior and the classical behavior of a
 given system, i.e., the process of quantization.

As we have seen above, the quantum curve
\eqref{Airy quantum} recovers the spectral
curve \eqref{x=y^2} by the procedure of
semi-classical limit.
We recall that the spectral curve lives
in the Hirzebruch surface $\bF_2$, 
and it has a quintic cusp singularity 
\eqref{quintic cusp} at $(x,y)=(\infty,\infty)$.
It requires two blow-ups of $\bF_2$ 
to resolve the singularity of $\Sigma$.
Let us denote  this minimal resolution by
$\widetilde{\Sigma}$.
The  proper transform
  is a smooth curve of genus $0$, 
 hence it is a $\bP^1$. 
Let $B\isom \bP^1$ be the $0$-section, 
and $F\isom \bP^1$ a fiber, of $\bF_2$.
Then after two blow-ups,
 $\widetilde{\Sigma}\subset Bl(\bF_2)$
is identified as a divisor by the
equation
$$
\widetilde{\Sigma} = 2B+5F-4E_2-2E_1
\in \Pic\big(Bl(\bF_2)\big),
$$ 
where $E_i$ is the exceptional divisor introduced
at the $i$-th blow-up 
\cite[Section 5]{OM2}.

\begin{figure}[htb]
\includegraphics[width=3in]{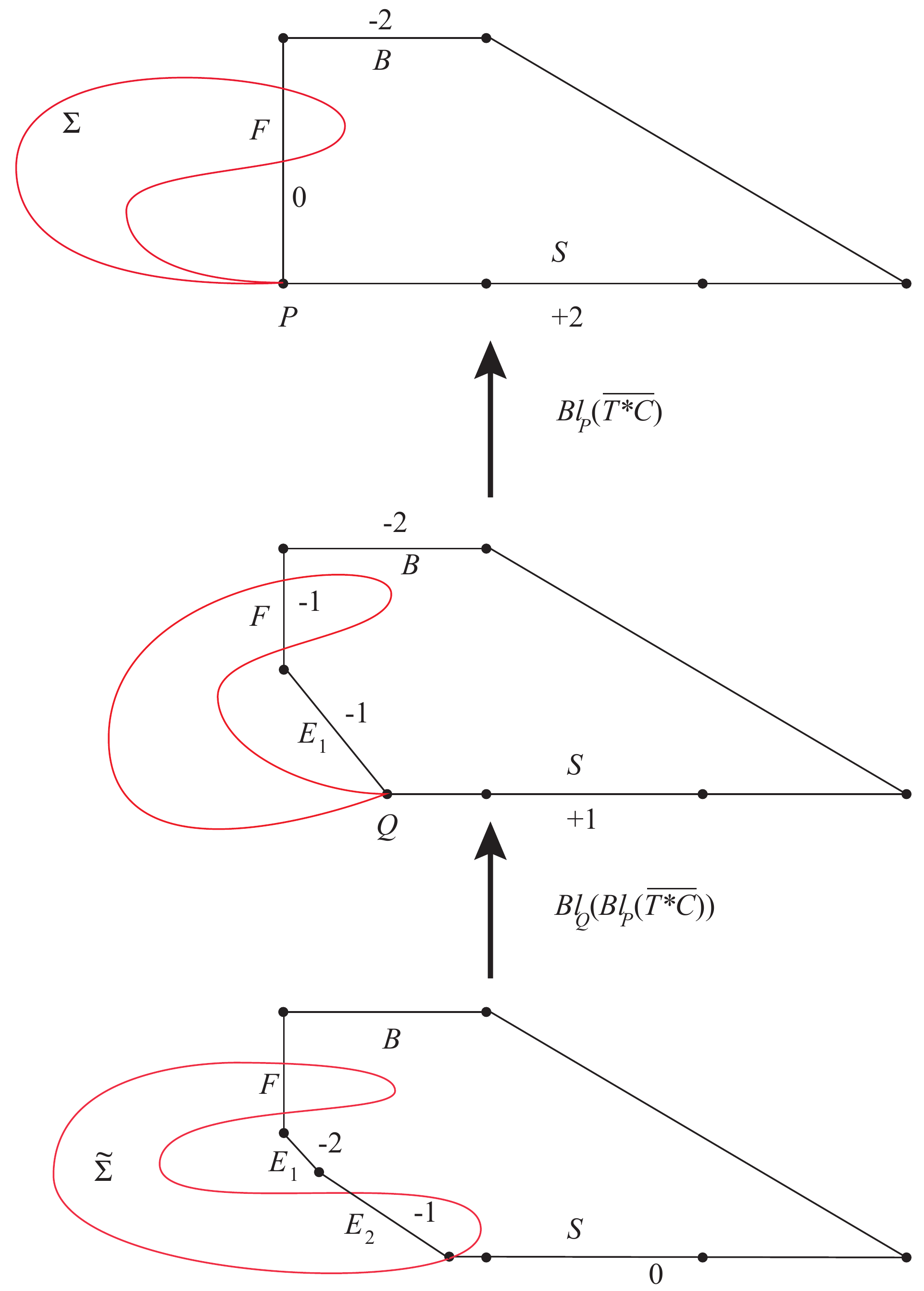}
\caption{Blowing up $\bF_2$ twice. The parabola
in $\bF_2$ has a quintic cusp singularity at
 infinity (top). After the second blow-up, 
the proper transform
becomes non-singular (bottom).
Since all fibers of $\bF_2\lrar \bP^1$ 
are equivalent, the toric picture has only one
representative of the fiber $F=\bP^1$. It is a stretch 
to place $\Sigma$ and $\widetilde{\Sigma}$
in this diagram, because the spectral curve
 is a double cover 
of the base $B$, rather than a degree $4$ covering
that the picture may suggest.}
\label{fig:blow-up}
\end{figure}

Since the desingularization $\widetilde{\Sigma}$
is just a copy of a $\bP^1$, we can choose 
a \textbf{normalization 
coordinate} $t$ on it so that the map
$\tilde{\pi}:\widetilde{\Sigma}\lrar \bP^1$ 
to the base curve $\bP^1$ is given by
\be
\label{t Airy}
\begin{cases}
x=\frac{4}{t^2}\\
y = -\frac{2}{t}
\end{cases},
\qquad
\begin{cases}
u = \frac{t^2}{4}\\
w = \frac{t^5}{32}
\end{cases}.
\ee
With respect to the normalization coordinate, 
define a homogeneous polynomial of
degree $6g-6+3n$ for $2g-2+n>0$ by
\be\label{Fgn Airy}
F_{g,n}^A(t_1,\cdots,t_n) 
:= \frac{(-1)^n}{2^{2g-2+n}}
\sum_{\substack{d_1+\dots+d_n\\
=3g-3+n}}
\la \tau_{d_1}\cdots \tau_{d_n}\ra_{g,n}
\prod_{i=1}^n 
|2d_i-1|!!\left(
\frac{t_i}{2}
\right)^{2d_i+1}
\ee
as a function on $(\widetilde{\Sigma})^n$,
and an $n$-linear differential form
\be
\label{Wgn Airy}
W_{g,n}^A(t_1,\dots,t_n):=d_{t_1}\cdots d_{t_n}
F_{g,n}^A(t_1,\dots,t_n).
\ee
For \emph{unstable geometries}
$(g,n) = (0,1)$ and $(0,2)$, we need to define
differential forms separately:
\begin{align}
\label{W01A}
&W_{0,1}^A(t) := \eta =ydx = \frac{16}{t^4}dt,
\\
\label{W02A}
&W_{0,2}^A(t_1,t_2):=
\frac{dt_1\cdot dt_2}{(t_1-t_2)^2}.
\end{align}
The definition of $W_{0,1}^A$ encodes 
the geometry of the singular spectral
curve $\Sigma\subset \bF_2$ embedded
in the Hirzebruch surface, and
$W_{0,2}^A$ depends only on the intrinsic 
geometry of the normalization $\widetilde{\Sigma}$.
Then we have

\begin{thm}[Topological recursion for
the intersection numbers, \cite{DMSS}]
\begin{multline}
\label{Airy TR}
W_{g,n}^A(t_1,\dots,t_n)
\\
=
-\frac{1}{2\pi i}\int_{\gam}
\left(
\frac{1}{t+t_1}+\frac{1}{t-t_1}
\right)
\frac{t^4}{64}\;\frac{1}{dt}\;dt_1
\Bigg[
W_{g-1,n+1}^A(t,-t,t_2, \dots, t_n)
\\
+
\sum_{\substack{g_1+g_2=g\\
I\sqcup J=\{2,\dots,n\}}} ^{\text{no $(0,1)$ terms}}
W_{g_1,|I|+1}^A(t,t_I)W_{g_2,|J|+1}^A(-t,t_J)
\Bigg],
\end{multline}
where the integral is taken with respect to 
the contour of Figure~\ref{fig:contourC},
and the sum is over all partitions of $g$ and
set partitions of $\{2,3,\dots,n\}$ without
including $g_1=0$ and $I=\emptyset$, or
$g_2=0$ and $J=\emptyset$. The notation 
$\frac{1}{dt}$ represents the ratio
operation, which acts on 
a differential $1$-form to produce a
global meromorphic function. When
acted on the quadratic differential, 
$\frac{1}{dt}$ yields a $1$-form.
\end{thm}

\begin{figure}[htb]
\centerline{\epsfig{file=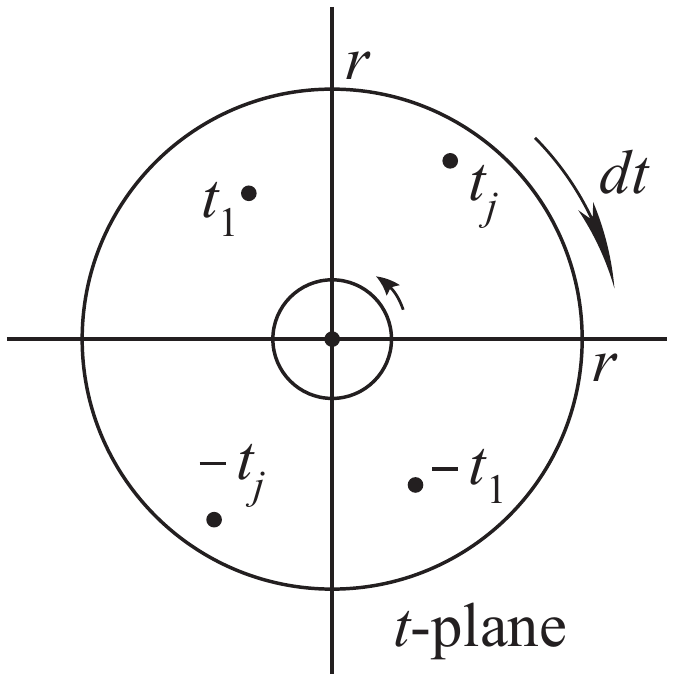, width=1.5in}}
\caption{The integration contour $\gamma$. This contour 
encloses an annulus bounded by two concentric 
circles centered at the origin. The outer one 
is a small circle around $\infty\in \widetilde{\Sigma}
=\bP^1$,
 and the inner one a small circle around the origin. 
 Both circles are positively oriented on 
 $\widetilde{\Sigma}$. Using the $t$-plane as 
 the affine chart, the small loop around
 $\infty$ becomes a circle of a large radius
 $r>\!>\infty$  with the
 opposite orientation, as in the figure.
 Geometrically, $t=0$ and $t=\infty$ correspond
 to the two simple ramification points
 of the Galois covering 
 $\tilde{\pi}:\widetilde{\Sigma}
 \lrar \bP^1$, and the contour $\gam$ consists
 of two small loops around these ramification points.}
\label{fig:contourC}
\end{figure}

This is an example of the 
axiomatic mechanism of \cite{EO2007} 
called the \emph{topological recursion}.
We give a more geometric interpretation
of the formula
in the later sections.
The derivation of the topological 
recursion \eqref{Airy TR}
from \eqref{Fgn Airy} is 
the subject of Section~\ref{sect:Catalan}
where the origin of the recursion 
formula is identified via graph
enumeration.
The power of this formula is that 
all $W_{g,n}^A$ for $2g-2+n>0$
are calculated from the initial values
\eqref{W01A} and \eqref{W02A}. For example,
\begin{align*}
W_{1,1}^A(t_1) &=
-\left[\frac{1}{2\pi i}\int_{\gam}
\left(
\frac{1}{t+t_1}+\frac{1}{t-t_1}
\right)
\frac{t^4}{64}\;\frac{1}{dt}
W_{0,2}^A(t,-t)
\right]dt_1
\\
&=
-\left[\frac{1}{2\pi i}\int_{\gam}
\left(
\frac{1}{t+t_1}+\frac{1}{t-t_1}
\right)
\frac{t^4}{64}\frac{(-dt)}{4t^2}
\right]dt_1
\\
&=
-\frac{1}{128} t_1^2dt_1
\\
&=
-\frac{3}{16}\la \tau_1\ra_{1,1}t_1^2dt_1.
\end{align*}
Thus we find $\la \tau_1\ra_{1,1} = \frac{1}{24}$.

The functions $F_{g,n}^A$ for 
$2g-2+n>0$ can be calculated 
by integration:
\be
\label{Fgn from Wgn}
F_{g,n}^A(t_1,\dots,t_n)
=\int_0 ^{t_1}\cdots\int_0^{t_n}
W_{g,n}^A(t_1,\dots,t_n).
\ee
Note that 
$$
t\rar 0 \Longleftrightarrow x\rar \infty.
$$
Therefore, we are considering the expansion 
of quantities at the essential singularity of 
$Ai(x)$. 
It is surprising to see that
 the topological recursion indeed 
determines all intersection numbers 
\eqref{intersection}!
Now define
\be
\label{principal specialization Airy}
S_m(x)=\sum_{2g-2+n=m-1}\frac{1}{n!}
F_{g,n}^A\big(t(x),\dots,t(x)\big),
\ee
where we choose a branch of $\pi:\Sigma \lrar \bP^1$,
and consider $t=t(x)$ as a function in $x$.
It coincides with \eqref{Airy Sm}.
Since the \emph{moduli} spaces $\Mbar_{0,1}$ and
$\Mbar_{0,2}$ do not exist, we 
do not have an expression \eqref{Fgn Airy} for
these unstable geometries.
So let us formally apply
\eqref{Fgn from Wgn} to 
\eqref{W01A}:
\be
\label{F01A}
F_{0,1}^A(t):= \int_0^{t}\frac{16}{t^4} dt
=-\frac{16}{3}t^{-3}
= -\frac{2}{3}x^{\frac{3}{2}} = S_0(x).
\ee
We do not have this type of integration procedure
to produce $S_1(x)$ from $W_{0,2}^A$. So we
simply define $S_1(x)$ by solving 
\eqref{S1}.
Then the residues in \eqref{Airy TR} can be
concretely computed, and produce
a system of recursive partial differential
equations among the $F_{g,n}^A$s. 
Their \textbf{principal
specialization}
\eqref{principal specialization Airy} produces 
 \eqref{Sm'}! Therefore, we 
 obtain \eqref{rainbow formula}.

In this context, the topological recursion is 
 the process of quantization, because it
actually constructs the function 
$\Psi(x,\hbar)$ by giving a closed
formula for the WKB analysis, and hence
the differential operator \eqref{Airy quantum}
that annihilates it, all from the classical curve
$x=y^2$.

\subsection{Non-Abelian Hodge 
correspondence and quantum curves}

Then what is the quantum curve?
Since it is a second order differential equation
with a deformation parameter $\hbar$, and 
its semi-classical limit is the spectral curve
of a Higgs bundle, which is a rank $2$ bundle
in our example, we can easily \emph{guess} that
it should be the result of the \textbf{non-Abelian
Hodge correspondence}. Since the 
quantization procedure is a \emph{holomorphic}
correspondence, while the non-Abelian Hodge
correspondence is \emph{not} holomorphic
as a map, we do not expect that these two 
are the same. 

To have a glimpse of the geometric effect 
of quantization, let us start with a Higgs bundle
$(E,\phi)$ of 
\eqref{E} and \eqref{Airy Higgs}.
The transition function of the vector bundle
$$
E=\cO_{\bP^1}(-1)\dsum \cO_{\bP^1}(1)
$$ 
on $\bP^1=U_\infty \cup U_0$
defined on $\bC^*=U_\infty \cap U_0$
is given by 
$
\begin{bmatrix}
x\\
&\frac{1}{x}
\end{bmatrix},
$
where $U_0=\bP^1\setminus \{\infty\}$ and
$U_\infty = \bA^1 = \bP^1\setminus \{0\}$.
The trivial extension 
$$
0\lrar \cO_{\bP^1}(-1)\lrar E\lrar \cO_{\bP^1}(1)\lrar 0
$$
has a unique $1$-parameter family of deformations
as the extension of $\cO_{\bP^1}(1)$ by $\cO_{\bP^1}(-1)$:
$$
0\lrar \cO_{\bP^1}(-1)\lrar E_\hbar\lrar \cO_{\bP^1}(1)\lrar 0,
$$
where $\hbar\in \Ext^1\left(\cO_{\bP^1}(1),\cO_{\bP^1}(-1)\right)
\isom H^1\left(\bP^1,K_{\bP^1}\right) \isom \bC$, and the
transition function of $E_{\hbar}$ is given by
$$
\begin{bmatrix}
x&\hbar\\
&\frac{1}{x}
\end{bmatrix}.
$$
Since 
$$
\begin{bmatrix}
1\\
-\frac{1}{\hbar x}&1
\end{bmatrix}
\begin{bmatrix}
x &\hbar\\
&\frac{1}{x}
\end{bmatrix}
\begin{bmatrix}
&-\hbar\\
\frac{1}{\hbar}&x
\end{bmatrix}
=
\begin{bmatrix}
1\\
&1
\end{bmatrix},
$$
\be
\label{Ehbar}
E_\hbar \isom 
\begin{cases}
\cO_{\bP^1}(-1)\dsum \cO_{\bP^1}(1) 
\quad \hbar =0\\
\cO_{\bP^1}\dsum \cO_{\bP^1}
\qquad \qquad  \; \hbar \ne 0.
\end{cases}
\ee
The Higgs field \eqref{Airy Higgs}
satisfies the transition relation
$$
-\begin{bmatrix}
&\frac{1}{u^5}\\
1&
\end{bmatrix}du
=\begin{bmatrix}
x&\\
&\frac{1}{x}
\end{bmatrix}
\begin{bmatrix}
&x\\
1
\end{bmatrix}dx
\begin{bmatrix}
x&\\
&\frac{1}{x}
\end{bmatrix}
^{-1},
$$
where $u=1/x$ is a coordinate on $U_\infty$.
Because of the relation
$$
-\begin{bmatrix}
&\frac{1}{u^5}\\
1&
\end{bmatrix}du
=\begin{bmatrix}
x&\hbar\\
&\frac{1}{x}
\end{bmatrix}
\begin{bmatrix}
&x\\
1
\end{bmatrix}dx
\begin{bmatrix}
x&\hbar\\
&\frac{1}{x}
\end{bmatrix}
^{-1}
-d \begin{bmatrix}
x&\hbar\\
&\frac{1}{x}
\end{bmatrix}
\begin{bmatrix}
x&\hbar\\
&\frac{1}{x}
\end{bmatrix}
^{-1},
$$
somewhat miraculously,
$\nabla^\hbar= \hbar d+\phi$ 
with the same Higgs fields as a connection 
matrix
becomes
an $\hbar$-connection in 
$E_\hbar$:
\be\label{hbar conn}
\nabla^\hbar= \hbar d+\phi :E_\hbar
\lrar K_{\bP^1}(5)\tensor E_\hbar.
\ee
It gives the Higgs field at $\hbar = 0$.
A flat section with respect to $\nabla^\hbar$ for
$\hbar\ne 0$ can be obtained by solving
\be\label{flat}
\left(\hbar \frac{d}{dx} +\begin{bmatrix}
&x\\
1
\end{bmatrix}
\right)
\begin{bmatrix}
-\hbar \Psi(x,\hbar)'\\
\Psi(x,\hbar)
\end{bmatrix}
=
\begin{bmatrix}
0\\
0
\end{bmatrix},
\ee
where $\Psi(x,\hbar)'$ is the $x$-derivative.
Clearly, \eqref{flat} is equivalent to 
\eqref{Airy quantum}.

For $\hbar=1$, \eqref{hbar conn}
is a holomorphic flat connection on 
$E_1|_{\bA^1}=\cO_{\bA^1}^{\dsum 2}$, the restriction of the 
vector bundle $E_1$ on the affine coordinate
neighborhood. Therefore, 
$\left.\left(E_\hbar,\nabla^{\hbar}\right)\right|_{\hbar = 1}$ defines a $\cD$-module
on $E_1|_{\bA^1}$. As a holomorphic $\cD$-module
over $\bA^1$,
we have an isomorphism
$$
\left.\left(\cO_{\bA^1}^{\dsum 2},\nabla^{\hbar=1}
\right)\right|_{\bA^1}\isom 
\cD\big/\big(\cD\cdot P(x,1)\big),
$$
where 
\be
\label{P Airy}
P(x,\hbar) := \left(\hbar\frac{d}{dx}\right)^2 -x,
\ee
and $\cD$ denotes the sheaf of linear differential
operators with holomorphic coefficients
on $\bA^1$.
The above consideration indicates how we
constructs a quantum curve as a $\cD$-module
from a given particular Higgs bundle.

Hitchin's original idea \cite{H1}
of constructing \emph{stable} Higgs 
bundles is to solve a system of differential 
equations, now known as \emph{Hitchin's equations}.
The stability condition for the Higgs bundles
can be translated into a system of nonlinear 
elliptic partial differential equations defined on  a \textbf{Hermitian} vector bundle 
$E\lrar C$ over a compact Riemann surface
$C$. It takes the following
form:
\be
\label{Hitchin}
\begin{aligned}
F(D) +[\phi,\phi^{\dagger_h}] &= 0
\\
D^{0,1}\phi &=0.
\end{aligned}
\ee
 Here, $h$ is
a Hermitian metric in $E$,
$D$ is a unitary connection in $E$ with respect
to $h$, $D^{0,1}$ is the $(0,1)$-component of 
the covariant differentiation with respect to the
complex structure of the curve $C$,
$F(D)$ is the curvature of $D$,  $\dagger_h$ is the 
Hermitian conjugation with respect $h$,
 and
$\phi$ is a differentiable Higgs field on $E$.
Solving Hitchin's equation is equivalent to 
constructing a $1$-parameter family of
flat connections of the form
\be
\label{nabla}
D(\zeta) = \frac{\phi}{\zeta}+D+\phi^{\dagger_h}
\zeta, \qquad \zeta\in \bC^*
\ee
(see \cite{GMN, N,S}).
The non-Abelian Hodge correspondence
is the association of 
$\left(\widetilde{E},D(1)^{1,0}\right)$
  to the given stable holomorphic
Higgs 
bundle $(E,\phi)$, i.e., a solution to Hitchin's equations.
Here, $D(1)= \phi + D + \phi^{\dagger_h}$,
and $\widetilde{E}$ is a holomorphic
vector bundle with the complex structure
given by the flat connection $D(1)^{0,1}$.
The connection $D(1)^{1,0}$ is then a
holomorphic connection in $\widetilde{E}$.

A new idea that relates the 
non-Abelian Hodge correspondence and
\emph{opers}
is emerging \cite{DFKMMN}. The role
of the topological recursion in this context, and
the identification of opers as 
globally defined quantum curves, are
being developed. Since it is beyond
our scope of the current lecture notes, it will be 
discussed elsewhere.

So far, we have considered $\Psi(x,\hbar)$ as
a formal ``function.'' What is it indeed? 
Since our vector bundle $E$ is of the form
\eqref{E}, the shape of the equation 
\eqref{flat}
suggests that 
\be
\label{Psi in Khalf}
\Psi (x,\hbar)\in 
\widehat{H}^0\big(\bP^1,K_{\bP^1}^{-\half}
\big),
\ee
where the hat sign indicates that we really 
do not have any good space to store
the formal wave function $\Psi (x,\hbar)$.
The reason for the appearance of 
$K_{\bP^1}^{-\half}$ as its home is 
understood as follows. 
Since $W_{g,n}^A$ for $2g-2+n$ is an
$n$-linear differential form, 
\eqref{Fgn from Wgn} tells us that
$F_{g,n}^A$ is just a number. 
Therefore, $S_m(x)$ that is determined
by \eqref{principal specialization Airy}
for $m\ge 2$ is also just a number. 
Here, by a ``number'' we mean a genuine
function in $x$. 
The differential equation \eqref{S0}
should be  written 
\be
\label{S0 geometry}
(dS_0(x))^{\tensor 2} = x(dx)^2 =q(x)
\ee
as an equation of quadratic differentials. 
Here, 
$$
q(x)\in K_{\bP^1}^{\tensor 2}
$$
 can be actually any meromorphic 
quadratic 
differential on $\bP^1$, including  
$x(dx)^2$, so that
$$
\phi = \begin{bmatrix}
&1
\\
q
\end{bmatrix}
\in K_{\bP^1}\tensor \End(E)
$$
is a meromorphic 
Higgs field on the vector bundle $E$ in the
same way.
Similarly, \eqref{S1} should be interpreted as
\be
\label{S1 geometry}
S_1(x) =-\half \int d\log (dS_0)
= -\half \log \sqrt{q(x)} = -\frac{1}{4}
\log q(x).
\ee
Now recall the conjugate 
differential equation \eqref{conjugate}.
Its solution takes the form
\be
\label{conjugate solution}
\frac{1}{\sqrt[4]{q(x)}}
\exp\left(\sum_{m=2}^\infty
S_m(x)\hbar^{m-1}
\right),
\ee
and as we have noted, the exponential factor
is just a number. 
Therefore, the geometric behavior of this 
solution is determined by the factor
$$
\frac{1}{\sqrt[4]{q(x)}} \in K_{\bP^1}^{-\half},
$$
which is a meromorphic section of the
negative half-canonical sheaf.

We recall that $\Psi(x,\hbar)$ has another 
factor $\exp\big(S_0(x)/\hbar\big)$.
It should be understood as a ``number''
defined on the spectral curve
$\Sigma$,
because $dS_0(x) = \eta$ is the tautological 
$1$-form on $T^*\bP^1$ restricted to the
spectral curve, and $S_0(x)$ is its integral
on the spectral curve. Therefore, 
this factor tells us that the equation
 \eqref{Airy quantum} should be considered
 on $\Sigma$. Yet its local $x$-dependence is
 indeed determined by 
 $K_{\bP^1}^{-\half}$.
 
We have thus a good answer for
Question~\ref{quest: spectral} now.  
The main part of the asymptotic expansion
\eqref{Airy 0,1} tells us what geometry
we should consider.
It tells us 
what  the Hitchin spectral curve should be, and it
also 
 includes the information of
Higgs bundle $(E, \phi)$ itself.

\begin{rem}
The Airy example, and another example we consider
in the next section, are in many ways special, in the
sense that the $S_2(x)$-term of the WKB expansion
is given by integrating the solutions $W_{1,1}$
and $W_{0,3}$ of the topological 
recursion \eqref{Airy TR}. In general, 
the topological recursion mechanism of
computing $W_{1,1}$
and $W_{0,3}$ from $W_{0,2}$ does not 
correspond to the WKB equation for $S_2$.
As discovered in \cite{OM1}, the topological 
recursion in its PDE form is equivalent to the 
WKB equations for all $S_m(x)$ in the
range of  $m\ge 3$. But the PDE recursion, 
which we discuss in detail in the later sections, 
does not determine $S_2$. It requires a new way 
of viewing the topological recursion
in its differential equation formulation:
\textbf{we consider $F_{1,1}$ and $F_{0,3}$
as the initial condition for topological recursion},
rather than $W_{0,1}$ and $W_{0,2}$,
which have been more 
commonly considered as the starting
point for the topological recursion.
\end{rem}

\subsection{The Lax operator for Witten-Kontsevich
KdV equation}

Surprisingly, the operator $P(x,1)$ 
of \eqref{P Airy} at $\hbar = 1$ 
is the 
initial value of the 
\emph{Lax operator} for the KdV equations
that appears in the work of Witten \cite{W1991}
and Kontsevich \cite{K1992}.
Witten considered a different generating function
of the intersection numbers \eqref{intersection}
given by
\be\label{Witten}
\begin{aligned}
F(s_0,s_1,s_2,\dots)
&=
\left<
\exp\left(
\sum_{d=0}^\infty {s_d \tau_d}
\right)
\right>
\\
&=
\sum_{k_0,k_1,k_2,\dots = 0} ^\infty
\left< \tau_0^{k_0}\tau_1^{k_1}\tau_2 ^{k_2}\cdots
\right>
\prod_{j=0}^\infty \frac{s_j^{k_j}}{k_j !}
\\
&=
\sum_{g=0}^\infty \sum_{n=1}^\infty
\frac{1}{n!}
\sum_{\substack{d_1+\cdots +d_n\\
=3g-3+n}}
\la \tau_{d_1}\cdots \tau_{d_n}\ra_{g,n}
s_{d_1}\cdots s_{d_n}
\\
&=
\la\tau_0^3\ra_{0,3}\frac{s_0^3}{3!}
+ \la \tau_1\ra_{1,1}s_1 +
\la \tau_0^4\ra_{0,4}\frac{s_0^4}{4!} +\cdots.
\end{aligned}
\ee
Define 
$$
t_{2j+1}:=\frac{s_j}{(2j+1)!!}, 
$$
and
\be
\label{Witten KdV}
u(t_1,t_3,t_5,\dots) := 
\left(\frac{\partial}{\partial s_0}\right)^2
F(s_0,s_1,s_2,\dots).
\ee
Then $u(t_1,t_3,t_5,\dots)$ satisfies the
system of KdV equations, whose first equation
is
\be
\label{KdV}
u_{t_3}=\frac{1}{4}u_{t_1t_1t_1}
+3 u u_{t_1}.
\ee

The system of KdV equations are the  deformation 
equation for the universal
iso-spectral family of  second order 
 ordinary differential operators of the form
\be\label{Lax}
L(X,t):=
\left(\frac{d}{dX}\right)^2 + 2u(X+t_1,t_3,t_5,\dots)
\ee
in $X$ and deformation parameters 
$t=(t_1,t_3,t_5,\dots)$. The operator is 
often referred to as the \emph{Lax operator}
for the KdV equations. The expression
$$
\sqrt{L} = \frac{d}{dX} + 
u\cdot  \left(\frac{d}{dX}\right)^{-1}
-\half u' \cdot \left(\frac{d}{dX}\right)^{-2}+\cdots
$$
makes
sense in the ring of pseudo-differential operators,
where $'$ denotes the $X$-derivative. 
The KdV equations are the system of 
\emph{Lax equations}
$$
\frac{\partial L}{\partial t_{2m+1}}
= \left[ \left(\sqrt{L}^{2m+1}\right) _+,L\right],
\qquad m\ge 0,
$$
where $_+$ denotes the differential operator part
of a pseudo-differential operator. The commutator
on the right-hand side explains the invariance of
the eigenvalues of the Lax operator $L$ with respect
to the deformation parameter $t_{2m+1}$. 
The $t_1$-deformation is the translation
$u(X)\longmapsto u(X+t_1)$, 
and the $t_3$-deformation is given by
the KdV equation \eqref{KdV}.

For the particular function \eqref{Witten KdV}
defined by the intersection numbers
\eqref{Witten}, 
 the \emph{initial}
value of the Lax operator is 
$$
\left(\frac{d}{dX}\right)^2+2u(X+t_1,t_3,t_5,\dots)
\big|_{t_1=t_3=t_5=\cdots = 0} 
=\left(\frac{d}{dX}\right)^2+2X.
$$
In terms of yet another
variable $x=-\frac{2}{ \sqrt[3]{4}}X$, the 
initial value becomes
$$
\frac{1}{\sqrt[3]{4}}
\left(\left(\frac{d}{dX}\right)^2+2X\right)
=
\left(\frac{d}{dx}\right)^2-x 
= P(x,\hbar)\big|_{\hbar=1},
$$
which is the Airy differential operator.

Kontsevich \cite{K1992}
used the matrix Airy function
to obtain all intersection numbers.
The topological recursion replaces
the asymptotic 
analysis of matrix integrations with 
a series of residue calculations on 
the spectral curve $x=y^2$.

\subsection{All things considered}

What we have in front of us is an interesting example
of a theory yet to be constructed. 
\begin{itemize}
\item
We have generating functions of 
quantum invariants, such as Gromov-Witten 
invariants. 
They are 
symmetric functions. 

\item We take the \emph{principal specialization}
of these functions, and form a generating function
of the specialized  generating functions. 

\item
This function then solves a $1$-dimensional 
stationary Schr\"odinger equation. The equation
is what we call a \emph{quantum curve}.

\item From this Schr\"odinger equation (or a
quantum curve), 
we construct an algebraic curve,  a
\emph{spectral curve} in the sense of Hitchin,
through the process of
\emph{semi-classical limit}.

\item The 
differential version of 
the \emph{topological recursion} \cite{OM1,OM2}
applied to the spectral curve then recovers
the starting quantum invariants.

\item The spectral curve can be also  
expressed as
the Hitchin spectral curve of a 
particular meromorphic
Higgs bundle.

\item Then the quantum curve is equivalent to 
the $\hbar$-connection in the $\hbar$-deformed
vector bundle
on the base curve, on which
the initial Higgs bundle is defined. 

\item The topological recursion therefore 
constructs a  flat section, although formal,
 of the
$\hbar$-connection from the Hitchin spectral curve.
At least locally, the $\hbar$-connection itself
is thus constructed by the topological recursion.

\end{itemize}
We do not have a general theory yet. In particular,
we do not have a global definition of 
quantum curves. As mentioned above,
right at this moment, the notion 
of \emph{opers} is emerging as a mathematical 
definition of quantum curves, at least for the
case of smooth spectral covers in 
the cotangent bundle $T^*C$ of a smooth 
curve $C$ of genus greater than $1$. 
We will report our finding in this exciting 
direction elsewhere.
Here, we present what we know as of 
now. 

The idea of topological recursion was
devised for a totally different context. 
In the authors' work \cite{OM1}, for the first time
the formalism of Eynard and Orantin was
placed in the Higgs bundle context. 
The formalism depends purely on the
geometry of the Hitchin spectral curve.
Therefore, quantities that the topological recursion 
computes should represent the geometric
information. Then in \cite{OM2}, we have shown
through examples
that quantization of singular spectral curves
are related to certain enumerative geometry 
problems, when the quantum curve is analyzed 
at its singular point and the function 
$\Psi(x,\hbar)$, which should be actually considered as
a formal section of $K_C^{-\half}$, is expanded at its
essential singularity. 
In the example described above, the
corresponding counting 
problem  is computing
 the intersection numbers of certain
 cohomology classes on $\Mbar_{g,n}$.
The original topological recursion of \cite{EO2007}
is generalized to singular spectral curves
in \cite{OM2} for this purpose.

The question we still do not know its answer
is how to directly connect the 
Higgs bundle information with the 
geometric structure whose quantum invariants
are captured by the topological recursion.

Since the time of inception of the topological 
recursion \cite{CEO, EO2007}, 
numerous papers  have been produced,
in both mathematics and physics.
It is far beyond the authors' ability to 
make any meaningful comments on this vast body
of literature in the present article. Luckily,
interested readers can find useful information
in Eynard's  talk 
at the ICM 2014 \cite{E2014}.
Instead of
attempting the impossible, we review here
a glimpse of \textbf{geometric developments} 
inspired by the topological recursion
that have taken place in the last few years.

The geometry community's 
keen attention was triggered when a concrete 
\textbf{remodeling conjecture} 
was formulated by string theorists, first by
 Mari\~no \cite{Mar},  and  then in a more
 precise and generalized framework by 
Bouchard, Klemm, Mari\~no and 
Pasquetti \cite{BKMP1, BKMP2}.
The conjecture states that  
open Gromov-Witten invariants of an arbitrary
toric Calabi-Yau orbifold of dimension $3$ 
can be calculated by the topological recursion
formulated on the \emph{mirror curve}.
A physical argument  for the correctness of
the conjecture is
pointed out in \cite{OSY}.
Bouchard and Mari\~no \cite{BM}
then derived a new
conjectural formula for simple Hurwitz numbers
from the remodeling conjecture.
The correctness of the
Hurwitz number conjecture can be easily
checked by a computer for many concrete
examples. At the same time, it was clear that
the conjectural formula was totally different
from the combinatorial 
formula known as the \textbf{cut-and-join equation}
of
\cite{Goulden,GJ,V}.

After many computer experiments,
one of the authors noticed that the conjectural 
formula of Bouchard and Mari\~no was
exactly the \emph{Laplace transform} of a 
particular variant of the 
cut-and-join equation. Once the precise
relation between the knowns and unknowns 
is identified, often the rest is straightforward, 
even though  a technical difficulty still remains. The 
conjecture for simple Hurwitz numbers
of \cite{BM} was  solved in 
\cite{EMS,MZ}. Its generalization 
to the orbifold Hurwitz numbers 
is then established in \cite{BHLM}.
In each case, the Laplace transform plays the 
role of the mirror symmetry, changing the
combinatorial problem on the A-model side to 
a complex analysis problem on the B-model side.
The first case of the remodeling conjecture
for $\bC^3$ was solved, using the same idea,
shortly afterwords in \cite{Zhou0}.
The remodeling conjecture in its full generality is 
recently solved in its final form by Fang, Liu, 
and Zong (announced in \cite{FLZ}), based on an earlier work of
\cite{EO3}.

Independent of these developments,
the relation between the topological recursion 
and combinatorics of enumeration of various
graphs drawn on oriented topological surfaces
has been studied by the Melbourne group of
mathematicians, including Do, Leigh, Manesco,
Norbury, and Scott (see, for example,
\cite{DoMan, N1,N2, NS}). The authors'
earlier papers \cite{CMS, DMSS, MP2012}
are inspired by their work.
A surprising  observation of the Melbourne group, 
formulated in a conjectural formula,
is that the Gromov-Witten invariants of 
$\bP^1$ themselves should
satisfy the topological recursion. 
Since $\bP^1$ is not a Calabi-Yau manifold,
this conjecture does not follow from the 
remodeling conjecture. 

The $GW(\bP^1)$ conjecture of \cite{NS}
is solved by the Amsterdam group of 
mathematicians, consisting of 
Dunin-Barkowski,  Shadrin, and 
Spitz, in collaboration with 
Orantin \cite{DOSS}. 
Their discovery,
that the topological recursion on 
a disjoint union of open discs as its spectral 
curve is equivalent to \emph{cohomological 
field theory},
has become a key technique of many later
works  \cite{ABO,DKOSS,FLZ}.

The technical difficulty of the topological 
recursion lies in the evaluation of residue calculations
involved in the formula. When the spectral 
curve is an open disc, this difficulty does not
occur. But if the global structure of the spectral 
curve has to be considered, then one needs 
a totally different idea. The work
of \cite{BHLM,EMS,MZ} has overcome the
 complex analysis difficulty
  in dealing with 
  simple and orbifold Hurwitz numbers.
  There, the key idea is the use of the 
  (piecewise) \emph{polynomiality}
  of these numbers through the ELSV
  formula \cite{ELSV} and its orbifold
  generalization \cite{JPT}. The paper 
   \cite{DKOSS} proves the  converse:
  they first prove the polynomiality 
  of the Hurwitz numbers without
  assuming the relation to 
  the intersection numbers over
  $\Mbar_{g,n}$, and then
  establish the ELSV formula from the 
  topological recursion, utilizing a technique of
  \cite{MZ}.

  The topological recursion is a byproduct of
  the study of random matrix theory/matrix
  models \cite{CEO, EO2007}. A recursion 
  of the same nature appeared earlier in the work of
  Mirzakhani \cite{Mir1, Mir2} on the
  Weil-Petersson volume of the moduli space
  of bordered hyperbolic surfaces. The Laplace 
  transform of the Mirzakhani recursion 
  is an example of the topological recursion. Its
  spectral curve, the sine curve, was first identified
  in \cite{MSaf} as the intertwiner  of 
  two representations of the Virasoro
  algebra.
  
  The notion of quantum curves goes back to 
\cite{ADKMV}. It is further developed in the
physics community by Dijkgraaf, Hollands,
 Su\l kowski, and Vafa \cite{DHS, DHSV, Hollands}.
The  geometry community
was piqued by \cite{DFM,GS}, which speculated on
the relation between the topological recursion,
quantization of the $SL(2,\bC)$-character
variety of the fundamental group of 
a knot complement in $S^3$,  the
AJ-conjecture due to \cite{Gar, GarLe},
and the $K_2$-group of algebraic $K$-theory.
Although it is tantalizingly interesting, so far
no mathematical results have been established in 
this direction for hyperbolic knots, or even it
may be impossible (see for example, \cite{BE}), 
mainly due to the 
reducible nature
of the spectral curve in this particular context. 
For torus knots, see also \cite{BEM}.
A rigorous construction of quantization of 
spectral curves was established for a few
examples in \cite{MS} for the first time,
but without any relation to knot
invariants. One of the examples
there will be treated in these lectures below. 
By now, we have
many more mathematical examples  of 
quantum curves \cite{BHLM,DoMan,
DMNPS,MSS,Zhou1,Zhou2}.
We note that in many of these examples, the 
spectral curves have a global parameter, 
even though the curves are not necessarily the rational
projective curve. Therefore, the situation is
in some sense still the ``genus $0$ case.''
The difficulty of quantization
 lies in 
dealing with complicated entire functions,
and the fact that the quantum curves are
\textbf{difference equations}, rather than 
differential equations of  finite order.
  
  We are thus  led to another question.
  
  \begin{quest} What  can we do when
  we have a different situation, where
  the spectral curve of the theory
  is global, compact, and 
  of a high genus?
  \end{quest}
  
  It comes as a surprise that 
  there is a system of recursive
  partial differential equations,
  resembling the residue calculation formula
  for the topological recursion,
  when the
  spectral curve of the topological recursion
  is precisely a Hitchin spectral curve
  associated with a rank $2$ Higgs bundle.
  The result of the calculation then leads us to 
  a  construction of a quantum
  curve. In this way a connection between
  quantization of Hitchin spectral curves 
  and the topological recursion is discovered in
  \cite{OM1}.
  
  If we regard the topological recursion as a
  method of calculation of quantum invariants, 
  then we need to allow singular spectral curves,
  as we have seen earlier. 
The simplest quantization of the 
singular Hitchin spectral 
curve is then obtained by the topological 
recursion again, but this time, it has to be
 applied to the
normalization of the singular spectral curve
constructed
by a particular way of using blow-ups of the
ambient compactified cotangent bundle. This is 
the content of \cite{OM2}, and we obtain
a \textbf{Rees $\cD$-module} as the result 
of quantization. 

We note here that 
what people call
by quantization is not unique. Depending
on the purpose, one needs to use a different
guideline for quantization. The result would be a
different differential equation, but still having the
same semi-classical limit.

For example, it has been rigorously proved  in
\cite{IS} that surprisingly
the quantization procedure of
\cite{OM1,OM2},
including the desingularization of the spectral curve,
  \emph{automatically} leads
to an iso-monodromic deformation family,
for the case of the Painlev\'e I equation. 
Their global parameter is essentially the 
normalization coordinate of the singular 
elliptic curve.
In their work \cite{IS},  
they ask what one obtains if the straightforward
topological recursion is applied for the quantization 
of a singular elliptic curve with a prescribed 
parameter in a particular way. 
They then find that the quantum curve 
is a Schr\"odinger equation whose coefficients
have nontrivial 
dependence on $\hbar$, yet it is an iso-spectral family
with respect to the parameter.
This work, and also the numerous 
mathematical examples of quantum curves
that have been already
constructed, suggest that the idea of 
 using
$\cD$-modules for the definition of
quantum curves (\cite{DHS,OM2})
is not the final word. Differential operators 
of an infinite order, or difference operators
mixed with differential operators, also have to be
 considered. 

For the mirror curves of toric Calabi-Yau
orbifolds of dimension $3$ appearing
in the context of the remodeling 
conjecture \cite{BKMP1, 
BKMP2, FLZ, Mar}, the 
conjectural quantum curves
acquire a very different nature. It has 
deep connections to number theory
and  quantum dilogarithm functions
\cite{KMar}.

A suggestion of using 
deformation quantization modules for
quantum curves is made by
Kontsevich in late 2013 in a private communication 
to the authors. An interesting 
work 
toward this direction is proposed by Petit \cite{Petit}.

The relation of the quantization discussed in these
lectures with the
\textbf{non-Abelian Hodge correspondence}
and \textbf{opers} is being investigated as of 
now \cite{DFKMMN}. 
A coordinate-free
global definition of  quantum curves
is emerging, and a direct relationship among
quantum curves, non-Abelian Hodge correspondence,
and opers is being developed.

The story is expanding its horizon. 
We have come to the other side of the rainbow.
And there we find ourselves on Newton's seashore.
So far we have  found only a few 
smoother pebbles or 
prettier shells, whilst...

\section{From Catalan numbers to
the topological recursion}
\label{sect:Catalan}

Let us consider a function $f(X)$ in one
variable, where $X$ is an $N\times N$ 
Hermitian matrix. 
One of the main problems of matrix integration theory
is to calculate the \emph{expectation value}
\be\label{expectation}
\la \tr f(X)\ra := 
\frac{
\int_{H_{N\times N}}\tr f(X) e^{V(X)}dX}
{\int_{H_{N\times N}} e^{V(X)}dX},
\ee
where the \emph{potential function}
$V(X)$ is given, such as the Gaussian potential
\be\label{Gaussian}
V(X) = -\half \tr(X^2),
\ee
so that 
$$
C_N=\int_{H_{N\times N}} e^{V(X)}dX
$$
 is
finite. 
The integration measure
 $dX$ is the standard $U(N)$-invariant
 Lebesgue measure 
of the space $H_{N\times N}=\bR^{N^2}$
of Hermitian matrices of size $N$.
When a Gaussian potential  \eqref{Gaussian}
is chosen, $e^V(X)dX$ is 
a probability measure after an appropriate
normalization, and
$\la \tr (X^m)\ra$ is the $m$-th 
\emph{moment}.
If we know all the moments, then we can 
calculate the expectation value of any 
polynomial function $f(X)$. 
Therefore, the problem changes into
calculating a  generating function 
of the moments
\be\label{moment generating}
\frac{1}{C_N}
\sum_{m=0}^\infty
\frac{1}{z^{m+1}}
\int_{H_{N\times N}}
 \tr(X^m)e^{V(X)}dX.
\ee
For a norm bounded matrix $X$, we have
$$
tr \left(\frac{1}{z-X} \right) = 
\sum_{m=0}^\infty
\frac{1}{z^{m+1}}
 \tr(X^m).
 $$
 Therefore, the  \emph{resolvent} of 
 a random matrix $X$,
\be\label{resolvent}
\left<
\tr\left(\frac{1}{z-X}\right)
\right>
=
\frac{1}{C_N}
\int_{H_{N\times N}}
\tr \left(\frac{1}{z-X} \right)e^{V(X)}dX,
\ee
looks the same as 
 \eqref{moment generating}.
 But they are not the same. 
For example, let us consider
the $N=1$ case and
write $X=x\in \bR$. Then
the formula
$$
\frac{1}{\sqrt{2\pi}}\int_{-\infty}^\infty \frac{1}{z-x}\;e^{-\half x^2}
dx
\overset{?}{=} 
\sum_{m=0} ^\infty \frac{(2m-1)!!}{z^{2m+1}}
=
\frac{1}{\sqrt{2\pi}}
\sum_{m=0}^\infty
\frac{1}{z^{m+1}}
\int_{-\infty}^\infty 
x^m\;e^{-\half x^2}
dx
$$
is valid only as the asymptotic expansion 
of the analytic function
$$
\frac{1}{\sqrt{2\pi}}\int_{-\infty}^\infty \frac{1}{z-x}\;e^{-\half x^2}
dx
$$
in $z$
for $Im(z)\ne 0$ and $Re(z)\rar +\infty$.
Still we can see that the information of 
the generating function of the moment
\eqref{moment generating}
can be extracted from the resolvent
\eqref{resolvent} if we apply the
technique of  asymptotic 
expansion of a holomorphic function
near at an essential singularity.
The asymptotic method of matrix integrals leads
to many interesting formulas, such as calculating
the orbifold Euler characteristic $\chi(\cM_{g,n})$
of the moduli space of smooth pointed curves
\cite{HZ}. We refere to \cite{M1998,MP1998} 
for introductory materials of these topics.

More generally, we can consider
\emph{multi-resolvent correlation functions}
\be\label{multi-resolvent}
\left<\prod_{i=1}^n
\tr\left(\frac{1}{z_i-X}\right)
\right>
=
\frac{1}{C_N}
\int_{H_{N\times N}}
\prod_{i=1}^n
\tr \left(\frac{1}{z_i-X} \right)e^{V(X)}dX.
\ee
When we say  ``calculating'' the expectation value
\eqref{multi-resolvent},
we wish to identify
it as a 
\emph{holomorphic} 
function 
in all the parameters, i.e.,
 $(z_1,\dots,z_n)$, the coefficients of 
 the potential $V$, 
and  the matrix size $N$. In particular, 
the analytic dependence on 
the parameter $N$ is an  important
feature we wish to determine.

It is quite an involved problem in analysis,
and we do not attempt to follow this route
in these lecture notes.
The collaborative effort of the random matrix
community has devised a recursive 
method of solving this analysis problem
(see, for example,
\cite{CEO, E2014, EO2007}), which is 
now known as 
the \textbf{topological recursion}.
One thing we can easily expect here is that since
\eqref{multi-resolvent}
is an analytic function in 
$(z_1,\dots,z_n)$,
 there must be an obvious relation between the 
topological recursion  and
algebraic geometry. 

What is amazing is that the exact same 
recursion formula
happens to appear in the context of many
different
enumerative geometry problems, again and again.
Even though the counting problems are different,
the topological recursion always takes the 
same general formalism. Therefore, to understand
the nature of this formalism, it suffices to 
give the simplest non-trivial example. 
This is what we wish to accomplish in this 
section.

\subsection{Counting  graphs on 
a surface}

Let us start with the following problem:

\begin{prob} Find the number of
 distinct cell-decompositions of
 a given closed oriented topological surface
of genus $g$, with  the specified number of
$0$-cells, and  the number of
$1$-cells that are incident to each $0$-cell.
\end{prob}

Denote by $C_g$ a compact,
$2$-dimensional, oriented
 topological manifold of genus $g$ 
without
boundary. Two cell-decompositions of $C_g$ 
are identified
if there is an orientation-preserving
 homeomorphism of $C_g$ onto itself
that brings one to the other. If there is such a map
for the same cell-decomposition, then it is an 
automorphism of the data. The $1$-\emph{skeleton}
of a cell-decomposition, which we denote by $\gam$,
is a \emph{graph} drawn on $C_g$. We call a
$0$-cell a \emph{vertex}, a $1$-cell an \emph{edge},
and a $2$-cell a \emph{face}. The midpoint of
an edge separates the edge into
 two \emph{half-edges} joined together at the midpoint.
The \emph{degree} of a vertex is the number of
half-edges incident to it. 
For the purpose of counting, we label all vertices. 
To be more precise, we give a total ordering to the
set of vertices. Most of the time we simply use
$[n]=\{1,\dots,n\}$ to label the set of $n$ vertices.

The $1$-skeleton $\gam$ is usually called a \emph{ribbon graph}, which is a graph with 
a cyclic order assigned to incident half-edges at 
each vertex. The face-labeled ribbon graphs 
 describe an orbifold cell-decomposition of 
$\cM_{g,n}\times \bR_+^n$. Since we label 
vertices of $\gam$, there is a slight difference
as to what the graph represents. It is the \emph{dual} 
graph of a ribbon graph, and its vertices are
labeled. To emphasize the dual nature, we call
$\gam$ a \textbf{cell graph}.

Most cell graphs do not have any non-trivial
automorphisms. If there is one, then it induces
a cyclic permutation of half-edges at a vertex, 
since we label all vertices. Therefore, if we pick
one of the incident half-edges at each vertex,
 assign an outgoing arrow to it, and require that
an automorphism also fix the arrowed half-edges,
then the graph has no non-trivial automorphisms. 
For a counting problem, no automorphism is 
a desirable situation because the bijective 
counting method 
works better there. Let us call such a graph an 
\textbf{arrowed cell graph}.
Now the refined problem:

\begin{prob}
Let $\vec{\Gam}_{g,n}(\mu_1,\dots,\mu_n)$
denote the set of arrowed cell graphs drawn 
on a closed, connected, 
oriented surface of genus $g$ with 
$n$ labeled vertices of degrees 
$\mu_1,\dots,\mu_n$. Find its cardinality
\be
\label{Cgn}
C_{g,n}(\mu_1,\dots,\mu_n):=
|\vec{\Gam}_{g,n}(\mu_1,\dots,\mu_n)|.
\ee
\end{prob}

When 
$\gam\in \vec{\Gam}_{g,n}(\mu_1,\dots,\mu_n)$,
we say $\gam$ has \emph{type} $(g,n)$. 
Denote by $c_\a(\gam)$ the number of $\a$-cells
of the cell-decomposition
associated with $\gam$. Then we have
$$
2-2g = c_0(\gam)-c_1(\gam)+c_2(\gam), \qquad
c_0(\gam)  =n, \qquad
2 c_1(\gam) = \mu_1+\cdots+\mu_n.
$$
Therefore,  
$\vec{\Gam}_{g,n}(\mu_1,\dots,\mu_n)$ 
is a finite set.

\begin{figure}[htb]
\includegraphics[height=0.8in]{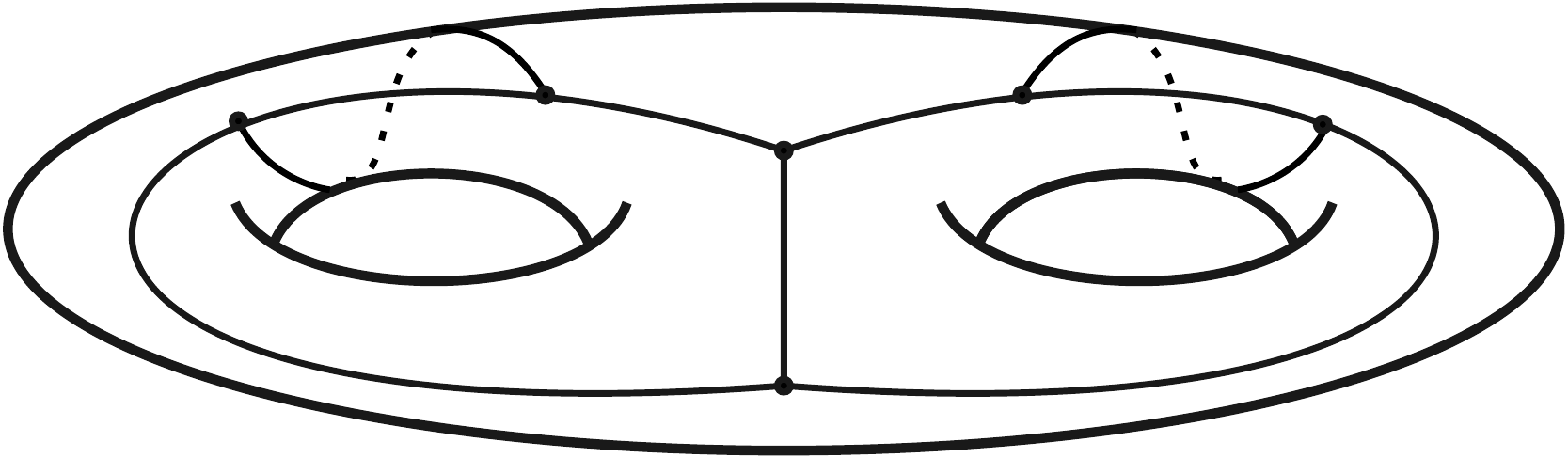}
\caption{A cell graph of type $(2,6)$.}
\label{figcell26}
\end{figure}

\begin{ex}
An arrowed cell graph of type $(0,1)$ is 
a collection of loops drawn on a plane
as in Figure~\ref{figcell01}. If we assign a pair
of parenthesis to each loop, starting from
 the arrowed one
as $($, and go around the unique vertex counter-clock
wise, then we obtain the parentheses pattern $(((\;)))$. 
Therefore, 
\be
\label{Catalan}
C_{0,1}(2m) = C_{m} = \frac{1}{m+1}
\binom{2m}{m}
\ee
is the $m$-th Catalan number. 
Thus it makes sense to call 
$C_{g,n}(\mu_1,\cdots,\mu_n)$  
\textbf{generalized Catalan numbers}.
Note that it is a symmetric
function  in $n$ integer variables.
\end{ex}

\begin{figure}[htb]
\includegraphics[height=0.8in]{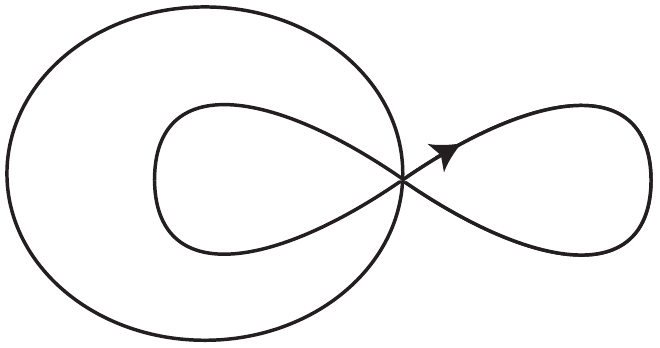}
\caption{A cell graph of type $(0,1)$ with a
vertex of degree $6$.}
\label{figcell01}
\end{figure}

\begin{thm}[Catalan Recursion, \cite{DMSS, WL}]
The generalized Catalan numbers 
satisfy the following equation.
\begin{multline}
\label{Catalan recursion}
C_{g,n}(\mu_1,\dots,\mu_n) 
=\sum_{j=2}^n \mu_j 
C_{g,n-1}(\mu_1+\mu_j-2,\mu_2,\dots,
\widehat{\mu_j},\dots,\mu_n)
\\
+
\sum_{\a+\b = \mu_1-2}
\left[
C_{g-1,n+1}(\a,\b,\mu_2,\cdots,\mu_n)+
\sum_{\substack{g_1+g_2=g\\
I\sqcup J=\{2,\dots,n\}}}
C_{g_1,|I|+1}(\a,\mu_I)C_{g_2,|J|+1}(\b,\mu_J)
\right],
\end{multline}
where $\mu_I=(\mu_i)_{i\in I}$ for 
an index set $I\subset[n]$,
$|I|$ denotes the cardinality of $I$, and
the third sum in the formula is for 
all  partitions of $g$ and 
set partitions of $\{2,\dots,n\}$. 
\end{thm}

\begin{proof}
Let $\gam$ be an arrowed cell graph counted
by the left-hand side of \eqref{Catalan recursion}.
Since all  vertices of $\gam$ are labeled, we write
the vertex set by $\{p_1,\dots,p_n\}$.
We take a look at the half-edge incident to $p_1$ that
carries an arrow. 

\begin{case}
The arrowed half-edge  extends to an edge $E$ that
connects $p_1$ and $p_j$ for some $j>1$.
\end{case}

In this case, we contract the edge and join the two
vertices $p_1$ and $p_j$ together. By this process
we create a new vertex of degree $\mu_1+\mu_j-2$.
To make the counting bijective, we need to be able
to go back from the contracted graph to the original,
provided that we know $\mu_1$ and $\mu_j$.
Thus we place an arrow to the half-edge 
next to $E$ around $p_1$ with respect to the
counter-clockwise cyclic order that comes from 
the orientation of the surface. In this process
we have $\mu_j$ different arrowed graphs 
that produce the same result, because we must
remove the arrow placed around the vertex $p_j$
in the original graph. This gives  
the first line of the right-hand side of 
\eqref{Catalan recursion}. See 
Figure~\ref{fig:case1}.

\begin{figure}[htb]
\includegraphics[width=2.5in]{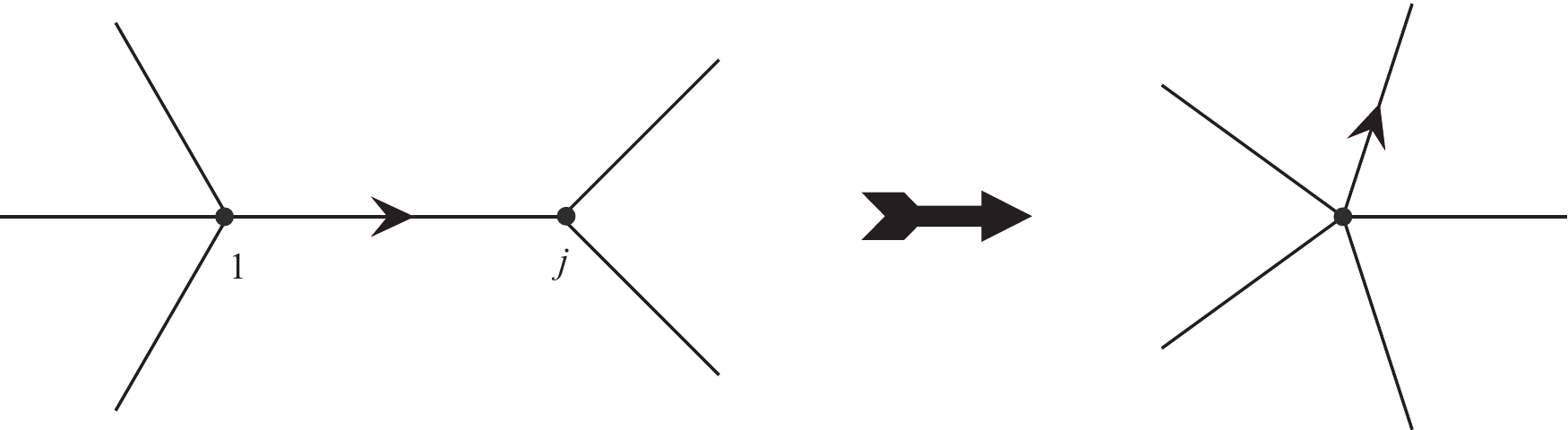}
\caption{The process of contracting the arrowed edge
$E$ that connects vertices $p_1$ and $p_j$, $j>1$.}
\label{fig:case1}
\end{figure}

\begin{case}
The arrowed half-edge at $p_1$ is  a loop
$E$ that goes out from, and comes back to, $p_1$.
\end{case}

The process we apply is again contracting the loop $E$.
The loop $E$ separates all other incident 
half-edges at $p_1$ into 
two groups, one consisting of $\a$ of them placed on
one side of the loop, and the other consisting of 
$\b$ half-edges placed on the other side. It can 
happen that $\a=0$ or $\b=0$. 
Contracting a loop on a surface causes pinching. 
Instead of creating a pinched (i.e., singular) surface, 
we separate the double point into two new vertices
of degrees $\a$ and $\b$. 
Here again we need to remember the place of the 
loop $E$. Thus we put an arrow to the half-edge
next to the loop in each group.
See Figure~\ref{fig:case2}.

\begin{figure}[htb]
\includegraphics[width=2.5in]{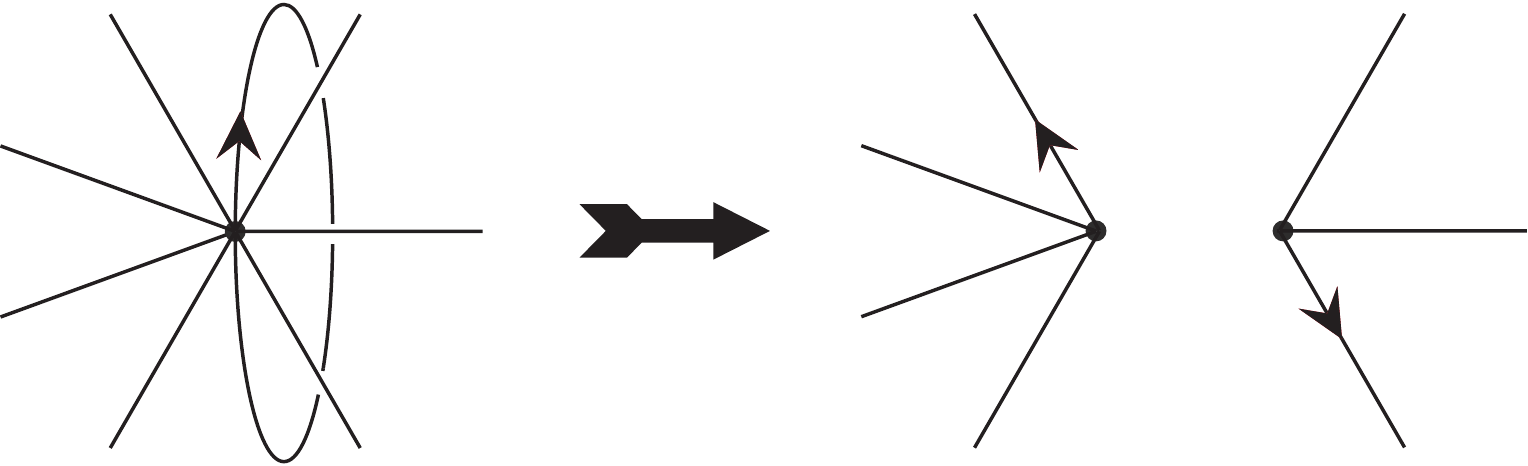}
\caption{The process of contracting the arrowed loop
$E$ that is attached to  $p_1$.}
\label{fig:case2}
\end{figure}

After the pinching and separating the double 
point, the original surface of genus $g$ with 
$n$ vertices $\{p_1,\dots,p_n\}$ may change its
topology. It may have genus $g-1$, or it splits into
two pieces of genus $g_1$ and $g_2$. 
The second line of \eqref{Catalan recursion}
records all such cases. 
Normally we would have a factor  $\half$ 
in front of the second line of the formula. 
We do not have it here because the arrow on the loop
could be in two different directions. Placing the
arrow on the other half-edge of the loop
is equivalent to interchanging $\a$ and $\b$.

This completes the proof.
\end{proof}

\begin{rem}
For $(g,n) = (0,1)$, the above formula reduces to 
\begin{equation}
\label{Cm}
C_{0,1}(\mu_1) =
\sum_{\a+\b=\mu_1-2} C_{0,1}(\a) C_{0,1}(\b).
\end{equation}
Since the degree of the unique vertex 
is always even for type $(0,1)$ graphs,
by defining $C_{0,1}(0)=1$, \eqref{Cm}
gives the Catalan recursion.
 Only for $(g,n)=(0,1)$,
 this irregular case of $\mu_1=0$ 
happens, because a degree $0$ single vertex is
\emph{connected}, and gives a cell-decomposition
of $S^2$. All other cases,
if one of the vertices has degree $0$, then the 
Catalan number $C_{g,n}(\mu_1,\dots,\mu_n)$ 
is simply $0$ because there is no corresponding
\emph{connected}
cell decomposition.
\end{rem}

\begin{rem}
Eqn.~\eqref{Catalan recursion} is a recursion 
with respect to 
$$
2g-2+n +\sum_{i=1} ^n \mu_i.
$$
The values are therefore  completely determined
by the initial value $C_{0,1}(0) = 1$. 
The formula does not give a recursion of
a \emph{function} $C_{g,n}(\mu_1,\dots,\mu_n)$,
because the same type $(g,n)$ appears on the 
right-hand side.
\end{rem}

The classical Catalan recursion \eqref{Cm}
determines all values of $C_{0,1}(2m)$, 
but the closed formula \eqref{Catalan}
requires a different strategy. Let us introduce
a generating function
\be
\label{z}
z = z(x) = \sum_{m=0}^\infty C_{0,1}(2m)
x^{-2m-1}.
\ee
From \eqref{Cm} we have
$$
z^2 = \sum_{m=0}^\infty 
\left(\sum_{a+b=m} C_{0,1}(2a)C_{0,1}(2b)
\right) x^{-2m-2}
=
\sum_{m=0}^\infty C_{0,1}(2m+2)x^{-2m-2}.
$$
Since
$$
xz = \sum_{m=-1}^\infty C_{0,1}(2m+2) x^{-2m-2},
$$
we obtain an equation
 $xz=z^2+1$, or
\be
\label{xz}
x = z + \frac{1}{z}.
\ee
This is the inverse function of 
the complicated-looking generating
function $z(x)$ at the branch
$z\rar 0$ as $x\rar +\infty$!  
Thus the curve \eqref{xz} knows 
everything about the Catalan numbers. 
For example, we can prove the closed 
formula \eqref{Catalan}.
The solution of \eqref{xz} as a quadratic equation
for $z$ that gives the above branch is given by
$$
z(x) = \frac{x-\sqrt{x^2-4}}{2} = \frac{x}{2}
\left(1-\sqrt{1-\frac{4}{x^2}}\right).
$$
The binomial expansion 
of the square root 
$$
\sqrt{1+X} = \sum_{m=0}^\infty
\binom{\half}{m}X^m,
$$
$$
\binom{\half}{m}:=
\frac{\half(\half-1)(\half-2)\cdots(\half-m+1)}{m!}
=(-1)^{m-1}\frac{(2m-3)!!}{2^m m!}
=\frac{(-1)^{m-1}}{4^m}\frac{1}{2m-1}
\binom{2m}{m},
$$
then 
gives the closed formula \eqref{Catalan} for the
Catalan numbers:
\begin{align*}
z(x) &=
\sum_{m=0}^\infty C_{0,1}(2m)
x^{-2m-1}
\\
&=
\frac{x}{2}
\left(1-\sqrt{1-\frac{4}{x^2}}\right)
\\
&=\frac{x}{2}\left(
1-\left(
\sum_{m=0}^\infty \frac{(-1)^{m-1}}{4^m}
\frac{1}{2m-1}\binom{2m}{m}(-1)^m \left(
\frac{4}{x^2}\right)^2
\right)\right)
\\
&=
\half
\sum_{m=1}^\infty \frac{1}{2m-1}\binom{2m}{m}
x^{-2m+1}
\\
&=
\sum_{m=0}^\infty \frac{1}{m+1}\binom{2m}{m}
x^{-2m-1}.
\end{align*}

Another piece of information we obtain from 
\eqref{xz} is the radius of convergence
of the infinite series $z(x)$ of \eqref{z}.
Since
$$
dx =\left(1-\frac{1}{z^2}\right)dz,
$$
the map \eqref{xz} is critical 
(or ramified) at $z=\pm 1$. The critical values
are $x=\pm 2$. On the branch we are considering,
\eqref{z} is the inverse function of \eqref{xz}
for all values of $|x|> 2$. This means 
the series $z(x)$ is absolutely convergent 
on the same domain.

\begin{rem}
In combinatorics, we often consider a generating
function of interesting quantities only as a formal
power series. The idea of topological recursion
tells us that we should consider the
\textbf{Riemann surface} of the 
Catalan number generating function
$z = z(x)$. We then recognize that there is a
global algebraic curve hidden in the scene,
which is the curve of Figure~\ref{fig:z}.
The topological recursion mechanism then tells
us how to calculate all $C_{g,n}(\vec{\mu})$
from this curve alone, known as the 
\textbf{spectral curve}.
\end{rem}

\begin{figure}[htb]
\includegraphics[height=1.4in]{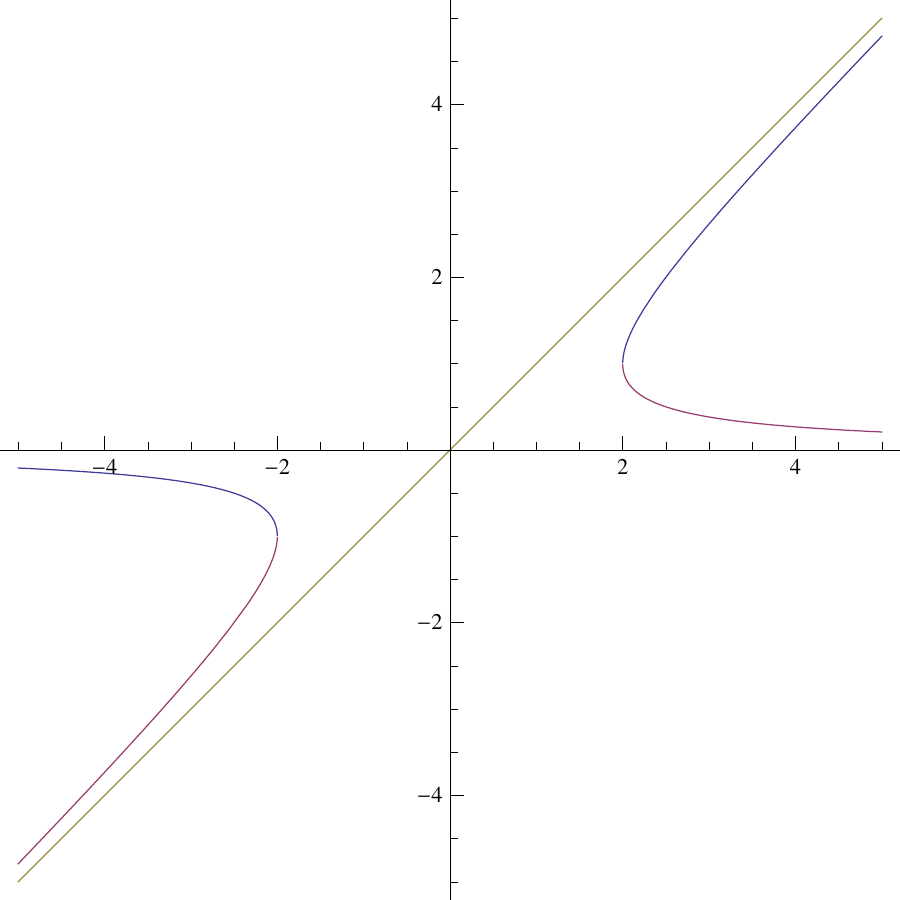}
\caption{The \textbf{Riemsnn surface}
of the Catalan generating function $z=z(x)$.}
\label{fig:z}
\end{figure}

\subsection{The spectral curve of a Higgs
bundle and its desingularization}

To consider the quantization of the curve
\eqref{xz}, we need to place it into a cotangent
bundle. Here again, we use the base curve
$\bP^1$ and the same vector bundle
$E$ of \eqref{E} on it. As a Higgs field, we use
\be
\label{Catalan Higgs}
\phi :=\begin{bmatrix}
-xdx&-(dx)^2\\
1
\end{bmatrix} : E\lrar K_{\bP^1}(4)\tensor E.
\ee
Here, $(dx)^2\in 
H^0\big(\bP^1, K_{\bP^1}^{\tensor 2}(4)\big)$
is the unique (up to a constant factor) 
quadratic differential
with an order $4$ pole at the infinity,
and $xdx\in 
H^0\big(\bP^1, K_{\bP^1}(2)\big)$
is the unique meromorphic differential
with a zero at $x=0$ and a pole of order $3$
at $x=\infty$.
In the affine coordinate $(x,y)$ of the Hirzebruch
surface $\bF^2$ as before, 
the spectral curve $\Sigma$ is given by
\begin{equation}
\label{Hermite-spectral}
\det\big(\eta-\pi^*(\phi)\big) = (y^2 +xy +1)(dx)^2 = 0,
\end{equation}
where $\pi:\bF_2\lrar \bP^1$ is the projection.
Therefore, the generating function $z(x)$ of \eqref{z} 
 gives a parametrization of 
the spectral curve
\be
\label{xy in z}
\begin{cases}
x = z(x)+\frac{1}{z(x)}
\\
y=-z(x).
\end{cases}
\ee
In the other affine coordinate
$(u,w)$ of \eqref{uw}, 
the spectral curve is singular at $(u,w) = (0,0)$:
\begin{equation}
\label{Hermite singularity}
u^4  -uw +w^2 = 0.
\end{equation}

\begin{figure}[htb]
\includegraphics[width=1.5in]{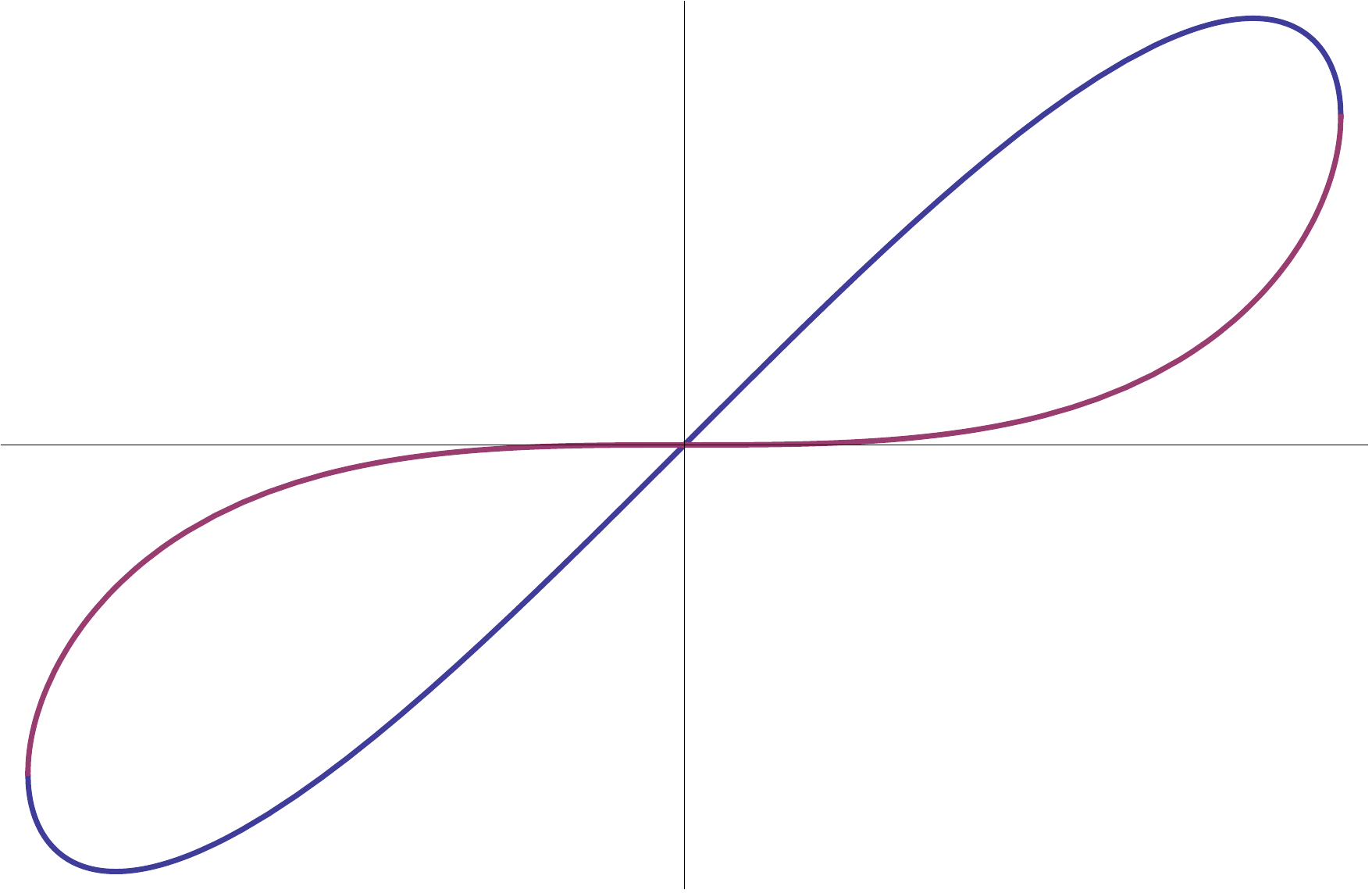}
\caption{The spectral curve $\Sigma$ of
\eqref{Hermite-spectral}. The horizontal 
line is the divisor $w=0$ at infinity, and the vertical
line is the fiber class $u=0$. The spectral curve
intersects with $w=0$ four times. One of the two
curve germ components 
 is given by $w=u$, and
the other by $w=u^3$.}
\label{fig:Hermite spectral}
\end{figure}

Blow up $\bF_2=\overline{T^*\bP^1}$ once
at the nodal singularity $(u,w)=(0,0)$
 of the spectral curve
$\Sigma$,
  and let
 $\widetilde{\Sigma}\lrar \Sigma$ be the proper
 transform of $\Sigma$. 
 \begin{equation}
 \label{blow-up}
\xymatrix{
\widetilde{\Sigma} 
\ar[dd]_{\tilde{\pi}} \ar[rr]^{\tilde{i}}\ar[dr]^{\nu}&&Bl(\overline{T^*\bP^1})
\ar[dr]^{\nu}
\\
&\Sigma \ar[dl]_{\pi}\ar[rr]^{i} &&
\overline{T^*\bP^1}  \ar[dlll]^{\pi}
\\
\bP^1 		}
\end{equation} 
In terms of the coordinate
$w_1$ defined by $w=w_1 u$,  
\eqref{Hermite singularity} becomes
 \begin{equation}
 \label{circle}
 u^2 + \left(w_1-\half\right)^2 = \frac{1}{4},
 \end{equation}
 which is the defining equation for 
 $\widetilde{\Sigma}$.  From this 
 equation we see that  its geometric 
genus is $0$, hence  it is just 
 another $\bP^1$. 
The covering
$\tilde{\pi}: \widetilde{\Sigma}\lrar \bP^1$
is ramified at two points, 
corresponding to the original ramification 
points $(x,y) = (\pm 2,\mp1)$ of 
$\pi: \Sigma\lrar \bP^1$.
The rational parametrization of \eqref{circle} is
given by 
\be
\label{rational}
\begin{cases}
u =\half \cdot \frac{t^2-1}{t^2+1}\\
w_1=\half -\frac{t}{t^2+1},
\end{cases}
\ee
where $t$ is the affine coordinate of 
$\widetilde{\Sigma}$ so that $t=\pm 1$ gives
$(u,w)= (0,0)$. Indeed,
the parameter $t$ is a \textbf{normalization
coordinate} of the spectral curve $\Sigma$:
\begin{equation}
\label{Hermite t}
\begin{cases}
x = 2+\frac{4}{t^2-1}\\
y=-\frac{t+1}{t-1},
\end{cases}
\qquad
\begin{cases}
u = \half \cdot \frac{t^2-1}{t^2+1}\\
w=\frac{1}{4}\cdot \frac{(t-1)^3(t+1)}{(t^2+1)^2}.
\end{cases}
\end{equation}
Although the expression of $x$ and $y$ in terms
of the normalization coordinate is more
complicated than \eqref{xy in z}, 
it is important to note that the 
spectral curve $\widetilde{\Sigma}$ is now
non-singular.

\subsection{The generating function, 
or the Laplace transform}

For all $(g,n)$ except for $(0,1)$ and $(0,2)$,
 let us introduce
the generating function of
\eqref{Cgn}
as follows:
\be
\label{Catalan Fgn}
F_{g,n}^C(x_1,\dots,x_n):=
\sum_{\mu_1\ge 1,\dots,\mu_n\ge 1}
\frac{C_{g,n}(\mu_1,\dots,\mu_n)}
{\mu_1\cdots\mu_n}
x_1 ^{-\mu_1}\cdots x_n^{-\mu_n}.
\ee
If we consider $x_i=e^{w_i}$, then the above
sum is just the discrete Laplace transform of
the  function in $n$ integer variables:
$$
\frac{C_{g,n}(\mu_1,\dots,\mu_n)}
{\mu_1\cdots\mu_n}.
$$
Our immediate goal is to compute the 
Laplace transform as a holomorphic function.
Since the only information we have now
is the generalized Catalan recursion
\eqref{Catalan recursion}, how much can we 
say about this function? Actually, 
the following theorem is proved, all from 
using the recursion alone!

\begin{thm}
\label{thm:FgnC}
Let us consider the generating function
\eqref{Catalan Fgn}
as a function in the normalization coordinates 
$(t_1,\dots,t_n)$ satisfying
$$
x_i = 2+\frac{4}{t_i^2-1}, \qquad i=1, 2, \dots, n,
$$
as identified in \eqref{Hermite t},
and by abuse of notation, we simply write it
as $F_{g,n}^C(t_1,\dots,t_n)$. For the 
range of $(g,n)$ with the stability
condition $2g-2+n>0$,  we have
the following. 
\begin{itemize}
\item The generating function
$F_{g,n}^C(t_1,\dots,t_n)$ is a \textbf{Laurent 
polynomial} in the $t_i$-variables
of the total degree $3(2g-2+n)$. 
\item The reciprocity relation holds:
\be
\label{reciprocity}
F_{g,n}^C(1/t_1,\dots,1/t_n) = 
F_{g,n}^C(t_1,\dots,t_n).
\ee
\item The special values at $t_i=-1$ are given by
\be
\label{initial}
F_{g,n}^C(t_1,\dots,t_n)\big|_{t_i=-1} = 0
\ee
for each $i$.
\item The diagonal value at $t_i=1$ gives
the orbifold Euler characteristic of the
moduli space $\cM_{g,n}$:
\be
\label{Fgn and Euler}
F_{g,n}^C(1,\dots,1) = (-1)^n \rchi(\cM_{g,n}).
\ee
\item The restriction of the Laurent polynomial
$F_{g,n}^C(t_1,\dots,t_n)$
to its highest degree
 terms gives a homogeneous polynomial
 defined by
\be
\label{Fgn highest}
F_{g,n}^{C, \text{ highest}}(t_1,\dots,t_n)
=\frac{(-1)^n}{2^{2g-2+n}}
\sum_{\substack{d_1+\cdots+d_n\\
=3g-3+n}}\la\tau_{d_1}\cdots \tau_{d_n}
\ra_{g,n}
\prod_{i=1}^n |2d_i-1|!!
\left(\frac{t_i}{2}\right)^{2d_i+1}.
\ee
\end{itemize}
\end{thm}

Thus the function $F_{g,n}^C(t_1,\dots,t_n)$
knows the orbifold Euler characteristic
of $\cM_{g,n}$, and all the cotangent
class intersection numbers \eqref{intersection}
for all values of $(g,n)$ in the stable range!
It is also striking that it is actually a Laurent
polynomial, while the definition 
\eqref{Catalan Fgn} is given only as a 
formal Laurent series.
The reciprocity \eqref{reciprocity}
is the reflection  of the invariance
of the spectral curve $\Sigma$ under the
rotation 
$$
(x,y) \longmapsto (-x,-y).
$$

This surprising theorem is a consequence of 
the Laplace transform of the Catalan 
recursion itself. 

\begin{thm}[Differential recursion, \cite{MZhou}]
\label{thm:Fgn recursion}
The Laplace transform $F^C_{g,n}(t_1,\dots,t_n)$ 
satisfies the following
differential recursion equation
for every $(g,n)$ subject to $2g-2+n\ge 2$.
\begin{multline}
\label{FC recursion}
\frac{\partial}{\partial t_1}F^C_{g,n}(t_1, \dots,t_n)
\\
=
-\frac{1}{16}
\sum_{j=2} ^n
\left[\frac{t_j}{t_1^2-t_j^2}
\left(
\frac{(t_1^2-1)^3}{t_1^2}\frac{\partial}{\partial t_1}
F^C_{g,n-1}(t_1,\dots,\widehat{t_j},\dots,t_n)
\right.\right.
\\
-
\left.\left.
\frac{(t_j^2-1)^3}{t_j^2}\frac{\partial}{\partial t_j}
F^C_{g,n-1}(t_2,\dots,t_n)
\right)
\right]
\\
-\frac{1}{16}
\sum_{j=2} ^n
\frac{(t_1^2-1)^2}{t_1^2}\frac{\partial}{\partial t_1}
F^C_{g,n-1}(t_1,\dots,\widehat{t_j},\dots,t_n)
\\
-
\frac{1}{32}\;\frac{(t_1^2-1)^3}{t_1^2}
\left.
\left[
\frac{\partial^2}{\partial u_1\partial u_2}
F^C_{g-1,n+1}(u_1,u_2,t_2, t_3,\dots,t_n)
\right]
\right|_{u_1=u_2=t_1}
\\
-
\frac{1}{32}\;\frac{(t_1^2-1)^3}{t_1^2}
\sum_{\substack{g_1+g_2=g\\
I\sqcup J=\{2,3,\dots,n\}}}
^{\rm{stable}}
\frac{\partial}{\partial t_1}
F^C_{g_1,|I|+1}(t_1,t_I)
\frac{\partial}{\partial t_1}
F^C_{g_2,|J|+1}(t_1,t_J).
\end{multline}
For a subset $I\subset \{1,2,\dots,n\}$, we denote
 $t_I = (t_i)_{i\in I}$.
The ``stable'' summation means
$2g_1+|I|-1>0$ and $2g_2+|J|-1>0$.

The differential recursion uniquely determines
all $F_{g,n}^C(t_1,\dots,t_n)$ by integrating
the right-hand side of \eqref{FC recursion}
from $-1$ to $t_1$ with respect to the variable
$t_1$.
The initial conditions are
\be
\label{F11}
F_{1,1}^C(t) = -\frac{1}{384}
\frac{(t+1)^4}{t^2}
\left(t-4+\frac{1}{t}\right)
\ee
and
\be
\label{F03}
F_{0,3}^C(t_1,t_2,t_3)= -\frac{1}{16}
(t_1+1)(t_2+1)(t_3+1)\left(1+\frac{1}{t_1\;t_2\;t_3}
\right).
\ee
\end{thm}

\begin{rem}
Theorem~\ref{thm:Fgn recursion} is proved
by Mincheng Zhou in his senior thesis
\cite{MZhou}. It is the result of 
a rather difficult  calculation
of the Laplace transform of the Catalan 
recursion \eqref{Catalan recursion}.
\end{rem}

\begin{rem}
Theorem~\ref{thm:FgnC} has never been stated in 
this format before. Its proof
follows from the results of \cite{OM2, DMSS,  MP2012,MS}, based on  induction
 using \eqref{FC recursion}.
The essential point of the discovery 
of Theorem~\ref{thm:FgnC} is the 
use of the normalization coordinate $t$
of \eqref{Hermite t}. The authors almost
accidentally found the coordinate
transformation 
$$
z(x) = \frac{t+1}{t-1}
$$
in \cite{DMSS}. Then in \cite{OM2}, we
noticed that this coordinate was exactly the
normalization coordinate that was naturally obtained
in the blow-up process \eqref{blow-up}.
\end{rem}

The uniqueness of the solution of \eqref{FC recursion}
allows us to identify the solution $F_{g,n}^C$
with the Laplace transform of the number of 
lattice points in $\cM_{g,n}$, as we see later
in this section. Through this identification,
\eqref{initial} and \eqref{Fgn and Euler} become
obvious. 
The asymptotic behavior \eqref{Fgn highest}
follows from the lattice point approximation of
the Euclidean volume of polytopes, and 
Kontsevich's theorem that identifies the 
Euclidean volume of $\cM_{g,n}$ and 
the intersection numbers \eqref{intersection}
on $\Mbar_{g,n}$.

\subsection{The unstable geometries and the 
initial value of the topological recursion}

The actual computation of the 
 Laplace transform equation \eqref{FC recursion}
 from \eqref{Catalan recursion} requires the evaluation
 of the Laplace transforms of $C_{0,1}(\mu_1)$
 and $C_{0,2}(\mu_1,\mu_2)$. It is done as follows.

Since a degree $0$ vertex is allowed for the $(g,n)
=(0,1)$ unstable geometry, we do not have 
the corresponding $F_{0,1}^C(x)$ in 
\eqref{Catalan Fgn}. Since
$$
d_{x_1}\cdots d_{x_n}F_{g,n}^C =
\sum_{\mu_1\ge 1,\dots,\mu_n\ge 1}(-1)^n
C_{g,n}(\mu_1,\dots,\mu_n)x_1^{-\mu_1-1}
\cdots x_n^{-\mu_n-1}dx_1\cdots dx_n,
$$
we could choose
$$
d_x F_{0,1}^C(x) = -\big(z(x)-x^{-1}\big)dx
$$
as a defining equation for $F_{0,1}^C$. 

In the  light of
\eqref{xy in z}, $ydx = -zdx$ is a natural 
global holomorphic $1$-form on $T^*\bP^1$,
called the \emph{tautological} $1$-form.
Its exterior differentiation $d(ydx) = dy\wedge dx$
 defines the canonical holomorphic symplectic 
structure on $T^*\bP^1$.
Since we are interested in the \emph{quantization}
of the Hitchin spectral curve, we need a
symplectic structure here, which is readily available
for our use from $ydx$.

Therefore, it is reasonable for us to 
define the `function' $F_{0,1}$ by 
\be
\label{F01 definition}
dF_{0,1} = ydx, \qquad \text{or}\qquad
F_{0,1} = \int ydx.
\ee
Although this equation does not determine the constant
term, it does not play any role for our purposes. 
Here, we choose the constant of integration to be $0$.
Since the symplectic structure on $T^*\bP^1$
is non-trivial, we need to interpret the equation 
\eqref{F01 definition} being defined on the 
spectral curve, and be prepared that we may not 
obtain any meromorphic function on the spectral curve.
For the Catalan case, we have to use the 
spectral curve coordinate $z$ or $t$
as a variable of $F_{g,n}^C$, instead of $x$.
Since 
$$
-zdx=-zdz+\frac{dz}{z}, 
$$
we conclude
\be
\label{F01C}
\begin{aligned}
F_{0,1}^C(z)&:= -\half z^2+\log z
\\
&=
-\half \left(\frac{t+1}{t-1}\right)^2
+\log\left(\frac{t+1}{t-1}\right).
\end{aligned}
\ee
We encounter a non-algebraic function here, indeed.

For the computation of the Laplace transform
$F_{0,2}^C$, we need
an explicit formula for $C_{0,2}(\mu_1,\mu_2)$.
Luckily, such computation has been 
carried out in
\cite{KP}, fully utilizing 
the technique of the
\emph{dispersionless Toda lattice hierarchy}.
 It is surprising to see how much integrable system
 consideration is involved in
computing such a simple quantity as 
$C_{0,2}(\mu_1,\mu_2)$! 
The result is the following.

\begin{thm}[Calculation of the $2$-point Catalan numbers, \cite{KP}]
\label{KP}
For every $\mu_1,\mu_2>0$, the genus $0$, 
$2$-point Catalan numbers are given by
\be
\label{C02}
C_{0,2}(\mu_1,\mu_2)
=
2 \frac{\lfloor \frac{\mu_1+1}{2}
\rfloor \lfloor \frac{\mu_2+1}{2}
\rfloor}{\mu_1+\mu_2}
\binom{\mu_1}{\lfloor \frac{\mu_1}{2}\rfloor}
\binom{\mu_2}{\lfloor \frac{\mu_2}{2}\rfloor}.
\ee
\end{thm}

We refer to \cite{DMSS} for the derivation 
of $F_{0,2}^C$. 
The result is
\be
\begin{aligned}
\label{F02C}
F_{0,2}^C(z_1,z_2) &= -\log(1-z_1z_2)
\\
&=
-\log\big(-2(t_1+t_2)\big)+\log (t_1-1)+\log(t_2-1).
\end{aligned}
\ee
For the purpose of later use, we note that
\be
\label{dF02C}
\begin{aligned}
d_{t_1}d_{t_2}F_{0,2}^C(t_1,t_2)
&=
\frac{dt_1\cdot dt_2}{(t_1+t_2)^2}
\\
&=
\frac{dt_1\cdot dt_2}{(t_1-t_2)^2}
-(\tilde{\pi}\times \tilde{\pi})^*
\frac{dx_1\cdot dx_2}{(x_1-x_2)^2},
\end{aligned}
\ee
where $\tilde{\pi}:\widetilde{\Sigma}\lrar \bP^1$
is the projection of \eqref{blow-up}, i.e., the
variable transformation \eqref{Hermite t}.

\subsection{Geometry of the topological recursion}

Computation of a multi-resolvent
\eqref{multi-resolvent} is one thing. What we have
in front of us is a different story. 
We wish to compute an asymptotic
expansion of a solution to the differential 
equation that is defined on the base curve $C$
and gives the quantization of the 
Hitchin spectral curve of a meromorphic
Higgs bundle $(E,\phi)$. The expansion has
to be done at the essential singularity of the
solution.

\begin{quest}
Is there a mathematical framework suitable 
for such problems?
\end{quest}

The discovery of \cite{OM1, OM2} 
gives an answer: 
\textbf{Generalize the formalism of
Eynard and Orantin of \cite{EO2007} to the 
geometric situation of meromorphic Higgs
bundles. Then this generalized topological 
recursion computes the asymptotic expansion in 
question.}

We are now ready to present the topological 
recursion, continuing our
investigation of the particular example
of Catalan numbers. The point here is that 
\textbf{the topological recursion is a 
universal formula depending only on geometry}. 
Therefore, 
we can arrive at the general formula from 
any example.
\emph{One example rules them all!}

\begin{thm}[The topological recursion for the 
generalized Catalan numbers, \cite{DMSS}]
\label{thm:TR for Catalan}
Define symmetric $n$-linear
differential forms
on $(\widetilde{\Sigma})^n$ for $2g-2+n>0$
by
\be
\label{WgnC}
W_{g,n}^C(t_1,\dots,t_n):=
d_{t_1}\cdots d_{t_n}F_{g,n}^C(t_1,\dots,t_n),
\ee
and for $(g,n)=(0,2)$ by
\be
\label{W02C}
W_{0,2}^C(t_1,t_2) := \frac{dt_1\cdot dt_2}
{(t_1-t_2)^2}.
\ee
Then these differential forms satisfy the 
following integral recursion equation, called the
\textbf{topological recursion}.
\begin{multline}
\label{Catalan TR}
W_{g,n}^C(t_1,\dots,t_n)
=
-\frac{1}{64} \; 
\frac{1}{2\pi i}\int_\gam
\left(
\frac{1}{t+t_1}+\frac{1}{t-t_1}
\right)
\frac{(t^2-1)^3}{t^2}\cdot \frac{1}{dt}\cdot dt_1
\\
\times
\Bigg[
\sum_{j=2}^n
\bigg(
W_{0,2}^C(t,t_j)W_{g,n-1}^C
(-t,t_2,\dots,\widehat{t_j},
\dots,t_n)
+
W_{0,2}^C(-t,t_j)W_{g,n-1}^C
(t,t_2,\dots,\widehat{t_j},
\dots,t_n)
\bigg)
\\
+
W_{g-1,n+1}^C(t,{-t},t_2,\dots,t_n)
+
\sum^{\text{stable}} _
{\substack{g_1+g_2=g\\I\sqcup J=\{2,3,\dots,n\}}}
W_{g_1,|I|+1}^C(t,t_I) W_{g_2,|J|+1}^C({-t},t_J)
\Bigg].
\end{multline}
The last sum is restricted to the stable 
geometries, i.e., the partitions 
should satisfies
$2g_1-1+|I|>0$ and $2g_2-1+|J|$, as 
in \eqref{FC recursion}. The 
contour integral is taken with respect to 
$t$ on the exactly the same
cycle  defined by Figure~\ref{fig:contourC}
as before,
where $t$ is the normalization coordinate
of \eqref{Hermite t}. Note that
the second and the third lines of \eqref{Catalan TR}
is a quadratic differential in the variable $t$. 
\end{thm}

\begin{rem}
The notation $\frac{1}{dt}$ requires a justification.
We note that two global meromorphic  sections
of the same line bundle is a global meromorphic function.
Here we are taking the ratio of two meromorphic
$1$-forms on the factor  $\widetilde{\Sigma}$
corresponding to the $t$-variable. Thus 
after taking this ratio, the
integrand becomes a meromorphic $1$-form in 
$(-t)$-variable, which is integrated 
along the cycle $\gam$.
\end{rem}

\begin{rem}
The recursion  \eqref{Catalan TR} is 
a genuine induction formula with respect to 
$2g-2+n$. Thus from  $W_{0,2}^C$, we can calculate 
all $W_{g,n}^C$s one by one. This is a big
difference between \eqref{Catalan TR}
and \eqref{Catalan recursion}. The latter
relation contains  terms with $C_{g,n}$ in the 
right-hand side as well, therefore, 
$C_{g,n}$ is not determined as a function by
an induction procedure.
\end{rem}

\begin{rem}
If we apply \eqref{WgnC} to $F_{0,2}^C$ of
\eqref{F02C}, then we obtain \eqref{dF02C},
not \eqref{W02C}. The difference is the pull-back of
the $2$-form $\frac{dx_1\cdot dx_2}{(x_1-x_2)^2}$.
This difference does not affect the recursion formula
\eqref{Catalan TR} for $2g-2+n> 1$. The only
case affected is $(g,n)=(1,1)$. 
The above recursion allows us
to calculate $W_{1,1}^C$ from $W_{0,2}^C$ as
we see below, but  we cannot 
use \eqref{dF02C} in place
of $W_{0,2}^C$ for this case because the specialization $t_1=t,t_2=-t$
does not make sense in 
$d_{t_1}d_{t_2}F_{0,2}(t_1,t_2)$.
\end{rem}

\begin{rem}
In Subsection~\ref{subsect:TR}, we will formulate the 
topological recursion as a universal formalism
for the context of Hitchin spectral curves in 
a coordinate-free manner depending only on 
the geometric setting. There 
we will explain the meaning of 
$W_{0,2}$, and the formula for the 
topological recursion. At this moment, we note that
the origin of the topological recursion is 
the edge-contraction mechanism of the 
Catalan recursion \eqref{Catalan recursion}.
There is a surprising
 relation between the edge-contraction operations
and two dimensional topological quantum field
theories. We refer to \cite{OM3} for more detail.
\end{rem}

\begin{rem}
The integral recursion \eqref{Catalan TR}
and the PDE recursion \eqref{FC recursion}
are equivalent in the range of
 $2g-2+n\ge 2$, since if we know $W_{g,n}^C$
from the integral recursion, then we
can calculate $F_{g,n}^C$ 
by the integration
$$
F_{g,n}^C(t_1,\dots,t_n)
=\int_{-1}^{t_1}\cdots\int_{-1}^{t_n}
W_{g,n}^C.
$$
But the differential recursion does not provide
any mechanism
to calculate $F_{1,1}^C$ and $F_{0,3}^C$.
\end{rem}

To see how the topological recursion works,
let us compute $W_{1,1}^C(t_1)$ from 
\eqref{Catalan TR}.

\begin{align*}
W_{1,1}^C(t_1)
&=
-\frac{1}{64} \; 
\frac{1}{2\pi i}
\left[
\int_\gam
\left(
\frac{1}{t+t_1}+\frac{1}{t-t_1}
\right)
\frac{(t^2-1)^3}{t^2}\cdot \frac{1}{dt}\cdot 
W_{0,2}^C(t, -t)
\right]
dt_1
\\
&=
-\frac{1}{64} \; 
\frac{1}{2\pi i}
\left[
\int_\gam
\left(
\frac{1}{t+t_1}+\frac{1}{t-t_1}
\right)
\frac{(t^2-1)^3}{t^2}\cdot \frac{1}{dt}\cdot 
\left(\frac{-(dt)^2}{(2t)^2}\right)
\right]
dt_1
\\
&=
\frac{1}{256} \; 
\frac{1}{2\pi i}
\left[
\int_\gam
\left(
\frac{1}{t+t_1}+\frac{1}{t-t_1}
\right)
\frac{(t^2-1)^3}{t^2}\cdot \frac{1}{t^2}\cdot dt
\right]dt_1
\\
&=
-\frac{1}{128} \frac{(t_1^2-1)^3}{t_1^4}.
\end{align*}
Here, we changed the contour integral to the
negative of the
residue calculations at $t= \pm t_1$, as indicated in
Figure~\ref{fig:contourC}.
From \eqref{F11}, we find that indeed 
$$
F_{1,1}^C(t_1)=\int_{-1}^{t_1}W_{1,1}^C.
$$
As explained in \cite{DMSS}, we can calculate
$W_{0,3}^C$ from $W_{0,2}^C$  as well,
which recovers the initial condition
\eqref{F03}.

To understand the 
geometry behind the topological recursion, we
need to identify each term of the formula.
First, recall the normalization morphism
$$
\tilde{\pi}:\widetilde{\Sigma}\lrar \bP^1
$$ 
of \eqref{blow-up}. 
We  note that the transformation 
$t\longmapsto -t$ appearing in the 
recursion formula is the 
\textbf{Galois conjugation} of the 
global Galois covering 
$\tilde{\pi}:\widetilde{\Sigma}\lrar \bP^1$.
From \eqref{Hermite t}, we see that
this transformation is induced by the the involution
$y\longmapsto -y$  of
$T^*\bP^1$.
The fixed point set of the Galois conjugation is the
set of ramification points of the covering
$\tilde{\pi}$, and the residue integration 
of \eqref{Catalan TR} is taken around the
two ramification points.

We claim  that
$W_{0,2}^C(t_1,t_2)$ is the 
\emph{Cauchy differentiation kernel} on 
$\widetilde{\Sigma}$.
This  comes from the intrinsic 
geometry of the curve $\widetilde{\Sigma}$.
The Cauchy differentiation kernel on $\bP^1$
is the unique meromorphic symmetric bilinear
differential form 
\be
\label{Cauchy}
c(t_1,t_2) := \frac{dt_1\cdot dt_2}{(t_1-t_2)^2}
\ee
 on
$\bP^1\times \bP^1$
such that
\be
\label{f to df}
df(t_2) = q_* \big( c(t_1,t_2) p^* f(t_1)\big)
\ee
for every rational function $f$ on $\bP^1$.
Here, $p$ and $q$ are the projection maps
 \begin{equation}
 \label{diff kernel P1}
\xymatrix{
&\bP^1 \times \bP^1 
\ar[dl]_{p}\ar[dr]^{q}
\\
\bP^1 &&\bP^1 		}
\end{equation}
to the first and second factors. The push-forward
$q_*$ is defined by the integral
$$
q_* \big( c(t_1,t_2) p^* f(t_1)\big)
=
\frac{1}{2\pi i}\oint
c(t_1,t_2)  f(t_1)
$$
along a small loop in the fiber $q^*(t_2)$
that is centered at its intersection with 
the diagonal of $\Delta\subset\bP^1\times\bP^1$.
In terms of the
other affine coordinate $u_i=1/t_i$
of $\bP^1$, we have
$$
\frac{dt_1\cdot dt_2}{(t_1-t_2)^2}
=
\frac{du_1\cdot du_2}{(u_1-u_2)^2}
=
-\frac{dt_1\cdot du_2}{(t_1u_2-1)^2}
=
-\frac{du_1\cdot dt_2}{(u_1t_2-1)^2}.
$$
Therefore, $c(t_1,t_2)$ is a globally defined 
bilinear meromorphic form
on $\bP^1\times \bP^1$ with poles only along 
the diagonal $\Delta$.

\begin{rem}
\label{rem:omega}
The Cauchy integration formula
$$
f(w) =\frac{1}{2\pi i}\oint_{|z-w|<\epsilon}
f(z)\frac{dz}{z-w}
$$
 is the most
useful formula in complex analysis. 
The Cauchy 
integration kernel $dz/(z-w)$ is a globally defined
meromorphic $1$-form on $\bC$ that has only one
pole, and its order is $1$.
It has to be noted that on a compact Riemann surface
$\Sigma$,
we do not have such a form.
The best we can so is the meromorphic $1$-form
 $\omega_\Sigma^{a-b}(t)$ with the following 
 properties:
 \begin{itemize}
 \item It is a globally defined meromorphic $1$-form
 with a pole of order $1$ and residue $1$ at 
 $a$ and a pole of order $1$ and residue 
 $-1$ at $b$ for some pair $(a,b)$ 
 of distinct points of $\Sigma$.
 \item It is holomorphic in $t$ except for $t=a$
 and $t=b$.
 \end{itemize}
 Since we can add any holomorphic $1$-form
 to $\omega_\Sigma^{a-b}(t)$ without violating
 the characteristic properties, the ambiguity
 of this form is exactly $H^0(\Sigma,K_\Sigma)$.
 Therefore, it is unique on  $\Sigma=\bP^1$, and
 is given by
 \be\label{omega P1}
 \omega_{\bP^1}^{a-b}(t)
 =
 \left(
 \frac{1}{t-a}-\frac{1}{t-b}
 \right)dt.
 \ee
 Note that
 $$
 d_{t_2}\omega_{\bP^1}^{t_2-b}(t_1)
 =c(t_1,t_2).
 $$
 for any $b$.
 \end{rem}

Let us go back to the topological recursion 
for the Catalan numbers \eqref{Catalan TR}.
Since all $W_{g,n}^C$ are determined by
$W_{0,2}^C$ using the recursion, the remaining
quantity we need to identify  is the 
\emph{integration kernel}.
Recall the tautological $1$-form
$$
\eta_{\bP^1}:= ydx
$$
on the tangent bundle $T^*\bP^1$. We can 
see from \eqref{blow-up} that its pull-back
to the normalized 
spectral curve $\widetilde{\Sigma}$ is given by
\be
\label{eta P1}
\begin{aligned}
\eta(t) &= -\left(\frac{t+1}{t-1}\right)
d\left(\frac{4}{t^2-1}\right)
\\
&=\frac{8t}{(t+1)(t-1)^3}\; dt.
\end{aligned}
\ee
The factor 
$$
-\frac{1}{64} 
\left(
\frac{1}{t+t_1}+\frac{1}{t-t_1}
\right)
\frac{(t^2-1)^3}{t^2}\cdot \frac{1}{dt}\cdot dt_1
$$
in \eqref{Catalan TR} is called 
the integration kernel, which is equal to 
\be
\label{kernel P1}
\half\; \frac{\omega_{\bP^1}^{(-t)-t}(t_1)}
{\eta(-t)-\eta(t)}.
\ee
The integration kernel appears in this form
in more general situations.

\subsection{The quantum curve for 
Catalan numbers}

There is yet another important role  the 
differential recursion 
\eqref{FC recursion} plays. This is the derivation of 
the quantum curve equation for the Catalan case. 
In terms of the coordinates $(x,y)$ of 
\eqref{Hermite-spectral}, the spectral curve
is $y^2+xy+1 = 0$. The generating function 
$z(x)$ of \eqref{z}, the normalization coordinate
$t$ of \eqref{Hermite t}, and the base curve 
coordinate $x$ on $\bP^1$ are related by 
\be
\label{t in x}
t =t(x):= \frac{z(x)+1}{z(x)-1},
\ee
which we consider as a function of $x$ that gives
the branch of $\widetilde{\Sigma}$ determined
by
$$
\lim_{x\rar +\infty} t(x) = -1.
$$

\begin{thm}[Quantum curve for generalized
Catalan numbers, \cite{MS}]
\label{thm:QC Catalan}
Define
\be
\label{Psi Catalan}
\Psi(t,\hbar):=
\exp\left(\sum_{g=0}^\infty \sum_{n=1}^\infty
\frac{1}{n!} \hbar^{2g-2+n}F_{g,n}^C(t,t,\dots,t)
\right),
\ee
incorporating \eqref{FC recursion}, \eqref{F01C}, 
and \eqref{F02C}.
Then as a function in $x$ through \eqref{t in x},
$\Psi\big(t(x),\hbar\big)$ satisfies the
following differential equation
\be
\label{QC Catalan}
\left[
\left(
\hbar\frac{d}{dx} 
\right)^2
+ x\left(
\hbar\frac{d}{dx} 
\right)
+1
\right]
\Psi\big(t(x),\hbar\big) = 0.
\ee
The  semi-classical limit of \eqref{QC Catalan}
using $S_0(x)=F_{0,1}^C\big(t(x)\big)$
coincides with the spectral curve 
$$
y^2+xy + 1 = 0
$$
of \eqref{Hermite-spectral}.
\end{thm}

\begin{rem}
Since $t=t(x)$ is a complicated function,
it is surprising  to see that 
$\Psi(t,\hbar)$ satisfies such a simple equation as
\eqref{QC Catalan}.
\end{rem}

\begin{rem}
The definition \eqref{Psi Catalan} and
the meaning of \eqref{QC Catalan} is the
same as in the situation we have
 explained in the Introduction. 
First, we define
\be
\label{Sm Catalan}
S_m(x):=\sum_{2g-2+n=m-1}\frac{1}{n!}
F_{g,n}^C\big(t(x),t(x),\dots,t(x)\big)
\ee
for every $m\ge 0$. We then 
interpret \eqref{QC Catalan}
as
\be
\label{Catalan conjugate}
\left[
e^{-\frac{1}{\hbar}S_0(x)}
\left(
\left(
\hbar\frac{d}{dx} 
\right)^2
+ x\left(
\hbar\frac{d}{dx} 
\right)
+1
\right)
e^{\frac{1}{\hbar}S_0(x)}
\right]
\exp\left(\sum_{m=1}^\infty \hbar^{m-1}S_m(x)
\right) = 0.
\ee
\end{rem}

To prove Theorem~\ref{thm:QC Catalan},
let us recall a lemma from \cite{MS}:

\begin{lem}
\label{lem:principal specialization}
Consider an open coordinate neighborhood
$U\subset \Sigma$ of a projective algebraic curve
$\Sigma$ with a coordinate $t$.
Let $f(t_1,\dots,t_n)$ be a 
meromorphic symmetric function
in $n$ variables defined on $\Sigma^n$.
Then on the coordinate neighborhood $U^n$, we have
\begin{equation}
\begin{aligned}
\label{dfdt}
\frac{d}{dt}f(t,t,\dots,t) 
&= n
\left.
\left[
\frac{\partial}{\partial u}f(u,t,\dots,t)
\right]
\right|_{u=t};
\\
\frac{d^2}{dt^2}f(t,t,\dots,t) 
&=
n\left.\left[
\frac{\partial^2}{\partial u ^2} 
f(u,t,\dots,t)
\right]\right|_{u=t}
\\
&\qquad
+
n(n-1)
\left.\left[
\frac{\partial^2}{\partial u_1 \partial u_2} 
f(u_1,u_2,t,\dots,t)
\right]\right|_{u_1=u_2=t}.
\end{aligned}
\end{equation}
For a meromorphic function in one variable 
$f(t)$ on  $\Sigma$,
we have
\begin{equation}
\label{lhopital}
\lim_{t_2\rar t_1}
\left[
\omega_\Sigma^{t_2-b}(t_1)\big(
f(t_1)-f(t_2)
\big)
\right]
=d_1f(t_1),
\end{equation}
where $\omega_\Sigma^{t_2-b}(t_1)$ is the $1$-form
of  Remark~\ref{rem:omega}.
\end{lem}

\begin{proof}[Proof of Theorem~\ref{thm:QC Catalan}]
Differential forms can be restricted to a subvariety,
but partial differential equations cannot be
restricted to a subvariety in general. Therefore,
it is non-trivial that \eqref{FC recursion} has
any meaningful restriction 
to the diagonal of $(\bP^1)^n$.
Our strategy is the following.
First, we expand \eqref{Catalan conjugate}
with respect to powers of $\hbar$, and derive
an infinite system of ordinary differential 
equations for a finite collection of $S_m(x)$'s.
(This method is known as the WKB analysis.)
We then prove that these ODEs are exactly
the \textbf{principal specialization} of
\eqref{FC recursion}, using \eqref{dfdt}.

Let 
\be
\label{Free}
F(x,\hbar):=
\sum_{m=0}^\infty
\hbar^{m-1} S_m(x).
\ee
Unlike the ill-defined  expression \eqref{Psi Catalan},
$F(x,\hbar)$ is just a generating function of
$S_m(x)$'s.
Then \eqref{QC Catalan}, interpreted as
\eqref{Catalan conjugate}, is equivalent to 
\begin{equation}
\label{F diff}
\hbar^2 \frac{d^2}{dx^2}F + 
\hbar^2\frac{dF}{dx}
\frac{dF}{dx} +x \hbar \frac{dF}{dx} +1=0.
\end{equation}
The $\hbar$-expansion of \eqref{F diff} 
gives
\begin{align}
\label{sclC}
&\hbar^0{\text{-terms}}:
\quad (S_0'(x))^2 +xS_0'(x) + 1=0,
\\
\label{consistencyC}
&\hbar^1{\text{-terms}}:\quad
2S_0'(x)S_1'(x) + S_0''(x)+xS_1'(x)=0,
\\
\label{h m+1C}
&\hbar^{m+1}{\text{-terms}}:\quad
S_m''(x) +\sum_{a+b=m+1}
S_a'(x)S_b'(x)+xS_{m+1}'(x)=0, \quad m\ge 1,
\end{align}
where $'$ denotes the $x$-derivative.
The WKB method is to solve these equations
iteratively and find $S_m(x)$ for all $m\ge 0$. Here,
\eqref{sclC} is the
\textbf{semi-classical limit}
of \eqref{QC Catalan}, and 
\eqref{consistencyC} is the 
\emph{consistency condition} we need for solving
the WKB expansion.
Since  $dS_0(x)=ydx$, we have $y=S_0'(x)$. Thus 
\eqref{sclC} follows from  the 
spectral curve  equation
\eqref{Hermite-spectral} and the definition of
$F_{0,1}^C(t)$ of \eqref{F01C}.

Recalling $x=z+1/z$, we obtain
$$
\frac{d}{dx} = \frac{z^2}{z^2-1}\frac{d}{dz}.
$$
In $z$-coordinate, $F_{0,2}^C(z,z) = -\log(1-z^2)$,
which follows from \eqref{F02C}.
Therefore,
$$
S'_1(x) = -\half \frac{z^2}{z^2-1}\frac{d}{dz}
\log(1-z^2)
=
-\frac{z^3}{(z^2-1)^2}.
$$
We can then calculate
\begin{align*}
2S_0'(x)S_1'(x) + S_0''(x)+xS_1'(x)
&=
2z \frac{z^3}{(z^2-1)^2}-\frac{z^2}{z^2-1}-
\left(z+\frac{1}{z}\right)\frac{z^3}{(z^2-1)^2}
\\
&=\frac{1}{(z^2-1)^2}\left(2z^4 -z^2(z^2-1)
-z^4-z^2
\right)
\\
&=0.
\end{align*}
Therefore, \eqref{consistencyC} holds.
We refer to \cite{OM2} for the proof of
\eqref{h m+1C} from \eqref{FC recursion}
and \eqref{lem:principal specialization}.
\end{proof}

\begin{rem}
The  solution $\Psi\big(t(x),\hbar\big)$ is 
a formal section of $K_{\bP^1}^{-\half}$, as before. 
Note that the quantum curve \eqref{QC Catalan}
has an irregular singularity at $x=\infty$, 
and hence its solution has an essential 
singularity at  infinity. The 
expression \eqref{Psi Catalan} gives 
the asymptotic expansion of a solution 
around 
$$
t\rar -1 \Longleftrightarrow x\rar\infty.
$$
The expansion in $1/x$ for
$\hbar>0$ is given in
\eqref{Catalan Psi expansion} below, using
the Tricomi confluent hypergeometric function.
\end{rem}

\subsection{Counting lattice points  on the
moduli space $\cM_{g,n}$}

The topological recursion 
\eqref{Catalan TR} is a consequence of
\eqref{FC recursion}, and the PDE recursion 
\eqref{FC recursion} is essentially the Laplace
transform of the combinatorial formula
\eqref{Catalan recursion}. 
Then from where does the relation to 
the intersection numbers 
\eqref{Fgn highest} arise?

To see this relation, we need to consider the
dual of the cell graphs. They are commonly 
known as Grothendieck's dessins d'enfants
(see for example, \cite{SL}), 
or ribbon graphs, as mentioned 
earlier. 
Recall that a ribbon graph has unlabeled 
vertices and edges, but faces are labeled. 
A \emph{metric} ribbon graph is a ribbon graph with a
positive real number (the length) assigned to each edge.
For a given ribbon graph $\Gamma$ with $e=e(\Gamma)$
edges, the space of metric
ribbon graphs is $\bR_+ ^{e(\Gamma)}/\Aut (\Gamma)$,
where the
automorphism group acts by permutations of edges
(see \cite[Section~1]{MP1998}).
We restrict ourselves to the case that
$\Aut (\Gamma)$ fixes each $2$-cell of the cell-decomposition.
We also require that every vertex of a ribbon graph has degree
(i.e., valence) $3$ or more.
Using the canonical holomorphic coordinate systems on
a topological surface of \cite[Section~4]{MP1998}
 and the Strebel differentials \cite{Strebel}, we have
an isomorphism of topological orbifolds 
\begin{equation}
\label{M=RG}
{\cM}_{g,n}\times \bR_+ ^n \isom RG_{g,n}.
\end{equation}
Here
$$
RG_{g,n} = \coprod_{\substack{\Gamma {\text{ ribbon graph}}\\
{\text{of type }} (g,n)}}
\frac{\bR_+ ^{e(\Gamma)}}{\Aut (\Gamma)}
$$
is the orbifold consisting of metric ribbon graphs of a given
topological type $(g,n)$.
The gluing of orbi-cells
 is done by making
the length of a non-loop edge tend to $0$. The space
 $RG_{g,n}$ is a smooth orbifold
(see \cite[Section~3]{MP1998} and \cite{STT}).
We denote by $\pi:RG_{g,n}\longrightarrow
\bR_+ ^n$ the natural projection via \eqref{M=RG}, which
is the assignment of the collection of perimeter lengths
of each boundary to a given metric ribbon graph.

Consider a ribbon graph $\Gamma$
whose faces are labeled
by $[n]=\{1,2\dots,n\}$. For the moment let us give a label to
each edge of $\Gamma$ by an index set $[e] = \{1,2,\dots,e\}$.
The edge-face incidence matrix  is then defined by
\begin{equation*}
\begin{aligned}
A_\Gamma &= \big[
a_{i\eta}\big]_{i\in [n],\;\eta\in [e]};\\
a_{i\eta} &= \text{ the number of times edge $\eta$ appears in
face $i$}.
\end{aligned}
\end{equation*}
Thus $a_{i\eta} = 0, 1,$ or $2$, and the sum of the
entries in each column is
always $2$. The $\Gamma$ contribution of the space
$\pi^{-1}(p_1,\dots,p_n) = RG_{g,n}(\bp)$
 of metric ribbon graphs with
a prescribed perimeter $\mathbf{p}=(p_1,\dots,p_n)$ is the orbifold
polytope
$$
P_\Gamma (\mathbf{p})/\Aut(\Gamma),\qquad
P_\Gamma (\mathbf{p})= \{\mathbf{x}\in \bR_+ ^e\;|\;
A_\Gamma \mathbf{x} = \mathbf{p}\},
$$
where $\mathbf{x}=(\ell_1,\dots,\ell_e)$ is the collection of
edge lengths of the metric ribbon graph $\Gamma$. We have
\begin{equation*}
\sum_{i\in [n]} p_i= \sum_{i\in [n]}
\sum_{\eta\in [e]}a_{i\eta}\ell_\eta =
2\sum_{\eta\in [e]}
\ell_\eta.
\end{equation*}

We recall the topological recursion for
the number of
 metric ribbon graphs $RG_{g,n} ^{\bZ_+}$
whose edges have
integer lengths, following \cite{CMS}.
We call such a ribbon graph
an \emph{integral ribbon graph}.
We can interpret an integral ribbon graph as
Grothendieck's \emph{dessin d'enfant} by considering
an edge of integer length as a chain of edges of length one
connected by bivalent vertices, and reinterpreting
 the notion of
$\Aut(\Gamma)$ suitably.
Since we do not go into the number theoretic aspects of
dessins, we stick to the more geometric
notion of integral ribbon graphs.

\begin{Def}
\label{def:N}
The weighted number $\big| RG_{g,n} ^{\bZ_+}(\bp)\big|$
of integral ribbon graphs with
prescribed perimeter lengths
$\bp\in\bZ_+ ^n$ is defined by
\begin{equation}
\label{Ngn}
N_{g,n}(\bp) =
\big| RG_{g,n} ^{\bZ_+}(\bp)\big|
=\sum_{\substack{\Gamma {\text{ ribbon graph}}\\
{\text{of type }} (g,n)}}
\frac{\big|\{\bx\in \bZ_+ ^{e(\Gamma)}\;|\;A_\Gamma \bx = \bp\}
\big|}{|\Aut(\Gamma)|}.
\end{equation}
\end{Def}

Since the finite set
$\{\bx\in \bZ_+ ^{e(\Gamma)}\;|\;A_\Gamma \bx = \bp\}$
is a collection of lattice points in the polytope $P_\Gamma(\bp)$
with respect to the canonical integral structure $\bZ\subset
\bR$ of the real numbers, $N_{g,n}(\bp)$ can be thought of
counting the
number of \emph{lattice points}
in $RG_{g,n}(\bp)$ with a weight factor
$1/|\Aut(\Gamma)|$ for each ribbon graph.
The function $N_{g,n}(\bp)$ is a symmetric function in
$\bp = (p_1,\dots,p_n)$
because the summation runs over all ribbon graphs of topological
type $(g,n)$.

\begin{rem}
Since the integral vector $\bx$ is restricted to
take strictly positive values,
we would have $N_{g,n}(\bp) = 0$ if we were to
substitute $\bp=0$.
This normalization is natural from the point of view of
lattice point counting and Grothendieck's
\emph{dessins d'enphants}. However, we do not
make such a substitution in these lectures because we
consider $\bp$ as a strictly positive integer vector.
This situation is similar to Hurwitz theory
 \cite{EMS, MZ}, where  a partition $\mu$ is
 a strictly positive integer vector that plays the role of our
 $\bp$. We note that  a different
 assignment of values was suggested in \cite{N1, N2}.
\end{rem}

For brevity of notation, we denote by $p_I = (p_i)_{i\in I}$
for a subset $I\in [n]=\{1,2\dots,n\}$. The cardinality of $I$ is
denoted by $|I|$. The following topological recursion
formula was proved in \cite{CMS} using the idea of
ciliation of a ribbon graph.

\begin{thm}[\cite{CMS}]
\label{thm:integralrecursion}
The number of integral ribbon graphs
with prescribed boundary lengthes satisfies the
topological recursion formula
\begin{multline}
\label{integralrecursion}
p_1  N_{g,n}(p_{[n]})
=
\half
\sum_{j=2} ^n
\Bigg[
\sum_{q=0} ^{p_1+p_j}
q(p_1+p_j-q)  N_{g,n-1}(q,p_{[n]\setminus\{1,j\}})
\\
+
H(p_1-p_j)\sum_{q=0} ^{p_1-p_j}
q(p_1-p_j-q)N_{g,n-1}(q,p_{[n]\setminus\{1,j\}})
\\
-
H(p_j-p_1)\sum_{q=0} ^{p_j-p_1}
q(p_j-p_1-q)
N_{g,n-1}(q,p_{[n]\setminus\{1,j\}})
\Bigg]
\\
+\half \sum_{0\le q_1+q_2\le p_1}q_1q_2(p_1-q_1-q_2)
\Bigg[
N_{g-1,n+1}(q_1,q_2,p_{[n]\setminus\{1\}})
\\
+\sum_{\substack{g_1+g_2=g\\
I\sqcup J=[n]\setminus\{1\}}} ^{\rm{stable}}
N_{g_1,|I|+1}(q_1,p_I)
N_{g_2,|J|+1}(q_2,p_J)\Bigg].
\end{multline}
Here
$$
H(x) = \begin{cases}
1 \qquad x>0\\
0 \qquad x\le 0
\end{cases}
$$
is the Heaviside function,
and the last sum is taken for all partitions
$g=g_1+g_2$ and $I\sqcup J= \{2,3,\dots,n\}$
 subject to the stability conditions
$2g_1-1+{I}>0$ and $2g_2-1+|J|>0$.
\end{thm}

For a fixed $(g,n)$ in the stable range, {\it i.e.}, $2g-2+n>0$, we
choose $n$ variables $t_1,t_2,\dots,t_n$, and
define the function
\be
\label{zt}
z(t_i,t_j) = \frac{(t_i+1)(t_j+1)}{2(t_i+t_j)}.
\ee
An edge $\eta$ of a ribbon graph $\Gamma$ bounds two
faces, say $i_\eta$ and $j_\eta$. These two faces may be
actually the same. Now we define the
\emph{Poincar\'e polynomial} 
\cite{MP2012}
of $RG_{g,n}$ in the $z$-variables
by
\begin{equation}
\label{Poincare}
F_{g,n}^P(t_1,\dots,t_n) =
\sum_{\substack{\Gamma {\text{ ribbon graph}}\\
{\text{of type }} (g,n)}}
\frac{(-1)^{e(\Gamma)}}{|\Aut(\Gamma)|}
\prod_{\substack{\eta \text{ edge}\\
\text{of }\Gamma}}
z\!\left(t_{i_\eta},t_{j_\eta}
\right),
\end{equation}
which is a polynomial in $z(t_i,t_j)$ but actually a
symmetric rational function in $t_1,\dots,t_n$.

Let us consider the \emph{Laplace transform}
 \begin{equation}
 \label{Lgn}
 L_{g,n}(w_1,\dots,w_n) \overset{\rm{def}}{=}
 \sum_{\bp\in\bZ_{+} ^n} N_{g,n}(\bp) e^{-\la \bp,w\ra}
 \end{equation}
 of the number of
integral ribbon graphs $N_{g,n}(\bp)$,
where $\la \bp,w\ra=p_1w_1+\cdots+p_nw_n$, and
the summation is taken over all integer vectors $\bp\in\bZ_+^n$
of strictly positive entries.
We shall prove that after the coordinate change
of \cite{CMS}
from the $w$-coordinates to the $t$-coordinates
defined by
\begin{equation}
\label{wt}
e^{-w_j} = \frac{t_j+1}{t_j-1}, \qquad j=1,2,\dots,n,
\end{equation}
 the Laplace transform
$L_{g,n}(w_{[n]})$ becomes the Poincar\'e
polynomial
\begin{equation*}
F_{g,n}^P(t_1,\dots,t_n) =
L_{g,n}\big(w_1(t),\dots,w_n(t)\big).
\end{equation*}

The Laplace transform $L_{g,n}(w_N)$ can be evaluated
using the definition
of the number of integral ribbon graphs
\eqref{Ngn}.
Let $a_\eta$ be the $\eta$-th column of the
incidence matrix $A_\Gamma$ so that
\begin{equation*}
A_\Gamma = \big[a_1\big|a_2\big|\cdots\big|a_{e(\Gamma)}\big].
\end{equation*}
Then
\begin{multline}
\label{LinA}
L_{g,n}(w_{[n]}) =
\sum_{\bp\in\bZ_+^n}
N_{g,n}(\bp) e^{-\la \bp,w\ra}
\\
=
\sum_{\substack{\Gamma {\text{ ribbon graph}}\\
{\text{of type }} (g,n)}}
\sum_{\bp\in\bZ_+^n}
\frac{1}{|\Aut(\Gamma)|}
\big|\{\bx\in \bZ_+ ^{e(\Gamma)}\;|\;A_\Gamma \bx = \bp\}
\big|
 e^{-\la \bp,w\ra}
 \\
 =
 \sum_{\substack{\Gamma {\text{ ribbon graph}}\\
{\text{of type }} (g,n)}}
\frac{1}{|\Aut(\Gamma)|}
\sum_{\bx\in\bZ_+^{e(\Gamma)}}
e^{-\la A_\Gamma \bx,w\ra}
\\
 =
 \sum_{\substack{\Gamma {\text{ ribbon graph}}\\
{\text{of type }} (g,n)}}
\frac{1}{|\Aut(\Gamma)|}
\prod_{\substack{\eta \text{ edge}\\
\text{of }\Gamma}}
\sum_{\ell_\eta=1}^\infty
e^{-\la a_\eta,w\ra\ell_\eta}
\\
 =
 \sum_{\substack{\Gamma {\text{ ribbon graph}}\\
{\text{of type }} (g,n)}}
\frac{1}{|\Aut(\Gamma)|}
\prod_{\substack{\eta \text{ edge}\\
\text{of }\Gamma}}
\frac{
e^{-\la a_\eta,w\ra}}
{1-e^{-\la a_\eta,w\ra}}.
\end{multline}
Every edge $\eta$ bounds two faces, which we call face
$i_\eta ^+$ and face $i_\eta ^-$. When $a_{i\eta}=2$,
these faces are the same.
We then calculate
\begin{equation}
\label{aw}
\frac{
e^{-\la a_\eta,w\ra}}
{1-e^{-\la a_\eta,w\ra}}
=
-z\!\left(t_{i_\eta ^+},t_{i_\eta ^-}
\right).
\end{equation}
This follows from (\ref{wt}) and
\begin{align*}
\frac{e^{-(w_i+w_j)}}{1-e^{-(w_i+w_j)}}
&=
\frac{\frac{(t_i+1)(t_j+1)}{(t_i-1)(t_j-1)}}
{1-\frac{(t_i+1)(t_j+1)}{(t_i-1)(t_j-1)}}
=
-\frac{(t_i+1)(t_j+1)}{2(t_i+t_j)}
=-z(t_i,t_j),
\\
\frac{e^{-2w_i}}{1-e^{-2w_i}}
&=
-\frac{(t_i+1)^2}{4t_i}
=-z(t_i,t_i).
\end{align*}
Note that since $z(t_i,t_j)$ is a symmetric
function, which face is
named $i_\eta ^+$ or $i_\eta ^-$ does not matter.
From \eqref{LinA} and \eqref{aw}, we have
established

\begin{thm}[The Poincar\'e polynomials and
the Laplace transform, \cite{MP2012}]
\label{thm:Finz}
The Laplace transform $L_{g,n}(w_{[n]})$ in terms
of the $t$-coordinates \eqref{wt}
 is the Poincar\'e polynomial
\begin{equation}
\label{Finz}
F_{g,n}^P(t_{[n]}) =
 \sum_{\substack{\Gamma {\text{ ribbon graph}}\\
{\text{of type }} (g,n)}}
\frac{(-1)^{e(\Gamma)}}{|\Aut(\Gamma)|}
\prod_{\substack{\eta \text{ edge}\\
\text{of }\Gamma}}
z\!\left(t_{i_\eta ^+},t_{i_\eta ^-}
\right).
\end{equation}
\end{thm}

\begin{cor}
\label{cor:F1}
The evaluation of $F_{g,n}^P(t_{[n]})$ at
$t_1=\cdots=t_n=1$ gives the
Euler characteristic of $RG_{g,n}$
\begin{equation}
\label{euler}
F_{g,n}^P(1,1\dots,1) = \chi\left(RG_{g,n}\right)
=(-1)^n\chi\left(\cM_{g,n}\right).
\end{equation}
Furthermore, if we evaluate at $t_j=-1$ for any $j$, then we have
\begin{equation}
\label{Fj}
F_{g,n}^P(t_1,t_2\dots,t_n)\big|_{t_j=-1} = 0
\end{equation}
as a function in the rest of the variables $t_{[n]\setminus\{j\}}$.
\end{cor}

\begin{proof}
The Euler characteristic calculation immediately follows
from $z(1,1) = 1$.

Consider a ribbon graph $\Gamma$ of type $(g,n)$. Its $j$-th
face
 has at least one edge on its boundary. Therefore,
$$
\prod_{\eta \text{ edge of }\Gamma}
z\!\left(t_{i_\eta ^+},t_{i_\eta ^-}
\right)
$$
has a factor $(t_j+1)$ by \eqref{zt}. It holds for every ribbon
graph $\Gamma$ in the summation of \eqref{Finz}.
Therefore, \eqref{Fj} follows.
\end{proof}

The following theorem is established in 
\cite{MP2012}. Its proof is the computation 
of the Laplace transform of the 
lattice point recursion
\eqref{integralrecursion}.

\begin{thm}[Differential recursion 
for the Poincar\'e polynomials, \cite{MP2012}]
The Poincar\'e polynomials $F_{g,n}^P(t_1,\dots,t_n)$
satisfy exactly the same differential recursion 
\eqref{FC recursion} with the same
initial values of $F_{1,1}^P$ and $F_{0,3}^P$.
\end{thm}

Since the recursion uniquely determines all
$F_{g,n}$'s for $2g-2+n>0$, we have 
the following:

\begin{cor} For every $(g,n)$ in the 
stable range $2g-2+n>0$, the two 
functions are identical:
\be
\label{C=P}
F_{g,n}^C(t_1,\dots,t_n) = 
F_{g,n}^P(t_1,\dots,t_n).
\ee
\end{cor}

Because of this identification, we see that
\eqref{euler} implies \eqref{Fgn and Euler}, 
and \eqref{initial} follows from \eqref{Fj}.
We can also see how \eqref{Fgn highest} 
holds.

The limit $t_i\rar \infty$ corresponds to $x_i\rar 2$
through the normalization coordinate of
\eqref{Hermite t}, and $x_i=2$ corresponds to a
branch point of 
$\tilde{\pi}:\widetilde{\Sigma}\lrar 
\bP^1$ of \eqref{blow-up}.
The defining equation 
\eqref{Catalan Fgn} of $F_{g,n}^C$ 
does not tell us what we obtain by 
taking the limit $x_i\rar 2$. 
The geometric meaning of the 
$t_i\rar \infty$ limit becomes crystal clear
by the equality \eqref{C=P}. Let us take a look
at the definition of the Poincar\'e polynomial 
\eqref{Poincare}. 
The fact that it is a Laurent polynomial follows
from the recursion \eqref{FC recursion} by induction.
If $|t_i|>|t_j|>\!>1$, then 
$$
z(t_i,t_j) =\frac{(t_i+1)(t_j+1)}{2(t_i+t_j)}
\sim \half t_j.
$$
Therefore, the highest degree part of 
$F_{g,n}^P$ comes from the graphs of 
type $(g,n)$ with the largest number of
edges. Denoting the number of vertices of
a ribbon graph $\Gam$ by $v(\Gam)$, 
we have
$$
2-2g-n = v(\Gam) -e(\Gam).
$$
To maximize $e(\Gam)$, we need to maximize
$v(\Gam)$, which is achieved by 
taking a trivalent graph (since we do not allow
degree $1$ and $2$ vertices in our ribbon graph).
By counting the number of half-edges of a trivalent
graph, we obtain $2e(\Gam) = 3v(\Gam)$.
Hence we have 
\be
\label{e Gam}
e(\Gam) = 6g-6+3n.
\ee
This is the degree of $F_{g,n}^P$, which 
agrees with the degree of \eqref{Fgn highest}, 
and also consistent with the dimension 
of \eqref{M=RG}.

Fix a point $\bp\in \bZ_+^n$, and scale it 
by a large integer $\lam>\!>1$. Then from
\eqref{Ngn} we see that the number 
 of lattice points
in the polytope $P_\Gam(\lam\bp)/\Aut(\Gam)$
that is counted as a part of $N_{g,n}(\lam \bp)$
is the same as the number of scaled 
lattice points $\bx\in \frac{1}{\lam}\bZ_+^n$
in $P_\Gam(\bp)/\Aut(\Gam)$. As $\lam\rar \infty$,
the number of lattice points can be approximated
by the Euclidean volume of the polytope
(cf.~theory of Ehrhart polynomials).

For a fixed $(w_1,\dots,w_n)$ with $Re(w_j)>0$,
the contribution from large $\bp$'s in the
 Laplace transform $L_{g,n}(w_1,\dots,w_n)$
 of \eqref{Lgn} is small.
 The asymptotic behavior of $L_{g,n}$ as
 $w\rar 0$ picks up the large perimeter
 contribution of $N_{g,n}(\bp)$, or the 
 counting of the lattice points of smaller and smaller
 mesh size. Since 
 $$
 t_j\rar \infty \Longleftrightarrow w_j\rar 0,
 $$
 the large $t_j$ behavior of $F_{g,n}^P$, 
 which is a homogeneous polynomial
 of degree $6g-6+3n$, reflects 
 the information of the volume of
 $\cM_{g,n}$ in its coefficients. 
 From Kontsevich \cite{K1992}, 
 we learn that the volume is exactly the 
 intersection number appearing
 in \eqref{Fgn highest}.

The topological recursion for the Airy case
\eqref{Airy TR} is the $t\rar \infty$ limit
of the Catalan topological recursion 
\eqref{Catalan TR}, as we see from the limit
$$
\frac{(t^2-1)^3}{t^2} \lrar t^4.
$$
Since the integrand of \eqref{Airy TR} has no
poles at $t=0$, the  small circle of the contour 
$\gam$ does not contribute any residue. Thus we
have derived the Airy topological recursion from 
the Catalan topological recursion.

\section{Quantization of spectral curves}
\label{sect:QC}

Quantum curves assemble information of 
quantum invariants in a compact manner. 
The global nature of  quantum curves is
not well understood at this moment of
writing. 
In this section, we focus on explaining 
the relation between the 
\textbf{PDE version} of topological recursion
discovered in \cite{OM1,OM2},
and the \emph{local}
expression of 
quantum curves,  suggested for example, in 
\cite{GS}. 
We give  the 
precise definition of the 
PDE  topological recursion
in a geometric and coordinate-free manner. 
The discovery of \cite{OM1,OM2} is 
to connect the ideas from  topological recursion 
with the Higgs bundle theory for the first time.
Global definition of the quantum curves
is being established
by the authors, based on a recent work 
\cite{DFKMMN}, and will be reported elsewhere.

As we have seen in the previous sections, for
the examples of topological recursion such as  
the Catalan numbers and the Airy function case,
the integral 
topological recursion is always a consequence
of a corresponding recursion of free energies
$F_{g,n}$ in the form of partial differential 
equations. Although we do not discuss them in
these lectures, the situation is also true for
the case of various Hurwitz numbers
\cite{BHLM, EMS, MZ}. 
Since quantum curve is a differential equation,
it is more natural to expect that the PDE recursion
is directly related to quantum curves than the
integral topological recursion. 

This consideration motivates the authors'
discovery of PDE topological recursion 
\cite{OM1,OM2}. We find that the most
straightforward quantization of Hitchin spectral curves
is obtained from the PDE recursion. Here, it has
to be remarked that if one uses the
integral topological recursion for Hitchin
spectral curves, that is also introduced in 
\cite{OM1,OM2}, then
the quantization process produces 
a differential equation whose coefficients depend
on all powers of $\hbar$, and thus the result 
is totally different from what we achieve.
This shows that the integral topological 
recursion, which is closer to the original
idea of \cite{EO2007}, and the PDE topological
recursion of \cite{OM1,OM2} are inequivalent
for the case of Hitchin spectral curves of genus
greater than $0$.

From a geometric point of view, our quantization 
is a natural notion. Therefore, we believe the
introduction of PDE topological recursion is 
crucial for building a theory of quantum curves.
It is also consistent from a physics point of view.
Teschner \cite{Teschner} relates quantization 
of Hitchin moduli spaces with the quantization 
of Hitchin spectral curves in the way we do here.

\subsection{Geometry of non-singular
Hitchin spectral curves
of rank $2$}
\label{subsect:non-singular}

We wish to  transplant the idea of 
\cite{EO2007} to Hitchin spectral curves.
Our first task is 
to determine the  differential form $W_{0,2}$ 
that gives the initial value 
of the topological recursion. It can reflect
many aspects of the Hitchin spectral curve.
In these lectures, we choose 
 the Cauchy differentiation kernel as $W_{0,2}$,
which depends only on the intrinsic geometry 
of the Hitchin spectral curve, 
for our purpose. Of course one can make other 
choices for different purposes, such as the kernel
associated with a connection of a spectral line 
bundle on the Hitchin spectral curve.

Even for the Cauchy differentiation kernel,
there is no canonical choice. It depends on
 a symplectic basis for 
the homology of the spectral curve.
We give one particular choice 
in this subsection.

In this subsection, we consider  a
smooth projective curve $C$ of  genus $g(C)\ge 2$
defined over $\bC$.
As before, $K_C$ is the canonical bundle 
on $C$. The cotangent bundle 
$\pi:T^*C\lrar C$ is the total space of $K_C$,
and there is the tautological section
\be
\label{eta}
\eta\in H^0(T^*C,\pi^*K_C),
\ee
which is a globally defined holomorphic $1$-form
on $T^*C$. 
\begin{equation}
 \label{tautological section}
\xymatrix{
&\pi^*(K_C)\ar[dl] 
\\
 T^*C \ar[dr]^\pi
&& K_C \ar[dl]
\\
&C 		}
\end{equation} 
Choose a \emph{generic} quadratic differential
$s\in H^0\big(C,K_C^{\tensor 2}\big)$,
so that the spectral curve
\begin{equation}
 \label{Sigma_s}
\Sigma_s\subset T^*C
\end{equation} 
that is defined by 
 \begin{equation}
 \label{r=2}
 \eta^{\tensor 2} + \pi^*s = 0
 \end{equation}
is non-singular. 
 Our
 spectral curve $\Sigma=\Sigma_s$ 
 is a double sheeted ramified
 covering of $C$ defined by 
  \eqref{r=2}.
The genus of the spectral curve is
$g(\Sigma) = 4g(C)-3$. 
This is because a generic $s$ has 
$\deg(K_C^{\tensor 2})=4g(C)-4$
simple zeros, which correspond to branch
points of the covering $\pi:\Sigma\lrar C$. 
Thus the genus is calculated 
by the Riemann Hurwitz formula
$$
2\left(2-2g(C)-\big(4g(C)-4\big)\right)
= 2-2g(\Sigma) - \big(4g(C)-4\big).
$$
 The cotangent bundle $T^*C$ has a natural
 involution
 \begin{equation}
 \label{eq:sigma}
 \sigma:T^*C\supset T_x^*C\owns (x,y)\longmapsto
 (x,-y)\in T_x^*C\subset T^*C,
 \end{equation}
which preserves the spectral curve $\Sigma$.
The action of this involution 
is the  deck-transformation.
  Indeed,  the  covering $\pi:\Sigma\lrar C$
  is a \textbf{Galois covering} with the
  Galois group  $\bZ/2\bZ = \la \sigma\ra$.

 Let $R\subset \Sigma$ denote the ramification divisor
 of this covering. Because $\Sigma$
 is non-singular, $R$ is supported at
 $4g(C)-4$ distinct points that are determined by
 $s=0$ on $C$. We consider $C$ as the $0$-section
 of $T^*C$. 
  Thus both $C$ and $\Sigma$ 
 are  divisors of $T^*C$. Hence $R$ is 
 also defined
  as $C\cdot \Sigma$ in $T^*C$
  supported on $C\cap \Sigma$. 
 Note that $\eta$ vanishes
 only along $C\subset T^*C$. As a holomorphic
 $1$-form on $\Sigma$, $\iota^*\eta$ has 
 $2g(\Sigma)-2 = 8g(C)-8$ zeros on 
 $\Sigma$. Thus it
 has a degree $2$ zero at each point of $\supp(R)$.
 
 As explained in the earlier sections using examples,
 we wish to choose a differential form $W_{0,2}$
 for our geometric situation. For this purpose, let us
recall  the differential form 
$\omega_\Sigma^{a-b}$ of
 Remark~\ref{rem:omega}. Since we can add 
 any holomorphic $1$-form to 
 $\omega_\Sigma^{a-b}$, we need to impose
 $g(\Sigma)$ independent conditions to 
 make it unique. If we have a principal
 polarization of the period matrix for $\Sigma$,
 then one obvious choice would be to impose 
 \be
 \label{omega normalization}
 \oint_{a_j}\omega_\Sigma^{a-b} =0
 \ee
 for  all ``$A$-cycles'' $a_j$, $j=1,2,\dots,g(\Sigma)$,
 following Riemann himself.
 The reason for this canonical choice
 is that we can make it for a family of 
 spectral curves 
 $$
 \{\Sigma_s\}_{ s\in U},
 $$
 where  
 $U\subset H^0\big(
 C,K_C^{\tensor 2}\big)$ is a
 contractible  open neighborhood 
 of $s$.
 Since the base curve $C$ is fixed, we can  
 choose and fix a symplectic basis for 
 $H_1(C,\bZ)$:
  \be
  \label{symplectic basis on C}
  \la A_1,\dots,A_g;B_1,\dots,B_g\ra  =
  H_1(C,\bZ).
  \ee
  From this choice, we construct a canonical 
  symplectic basis
  for $H^1(\Sigma,\bZ)$ as follows.
 Let us label points of 
 $R=\{p_1,p_2,\dots,p_{4g-4}\}$, where 
 $g=g(C)$. We can 
 connect $p_{2i}$ and $p_{2i+1}$, $i=1, \dots,
 2g-3$, with a simple
 path on $\Sigma$ 
 that is mutually non-intersecting so that
 $\pi^*(\overline{p_{2i}p_{2i+1}})$,
 $i=1, \dots,
 2g-3$, 
 form a part of the basis for $H_1(\Sigma,\bZ)$. 
 We denote these cycles by $\a_1,\dots,\a_{2g-3}$.
 Since $\pi$ is locally homeomorphic away from 
 $R$, we have $g$ cycles $a_1,\dots,a_g$ on
 $\Sigma_s$ so that $\pi_*(a_j) = A_j$ for
 $j=1,\dots,g$, where $A_j$'s are 
 the $A$-cycles of $C$
 chosen as \eqref{symplectic basis on C}. 
 We define the $A$-cycles of $\Sigma$ to be the
 set 
 \begin{equation}
 \label{A on Sigma}
 \{a_1,\dots,a_g,\sigma_*(a_1),\dots,\sigma_*(a_g),
 \a_1,\dots,\a_{2g-3}\}\subset H_1(\Sigma,\bZ),
\end{equation}
where $\sigma$ is the Galois conjugation
(see Figure~\ref{fig:Sigma}).
Clearly, this set can be extended into a symplectic
basis for $H_1(\Sigma,\bZ)$. 
This choice of the symplectic basis 
trivializes the homology bundle
$$
\big\{H_1(\Sigma_s,\bZ)\big\}_{s\in U}
\lrar U\subset H^0(
 C,K_C^{\tensor 2})
 $$
 locally on the contractible neighborhood $U$.

\begin{figure}[htb]
\includegraphics[height=1.7in]{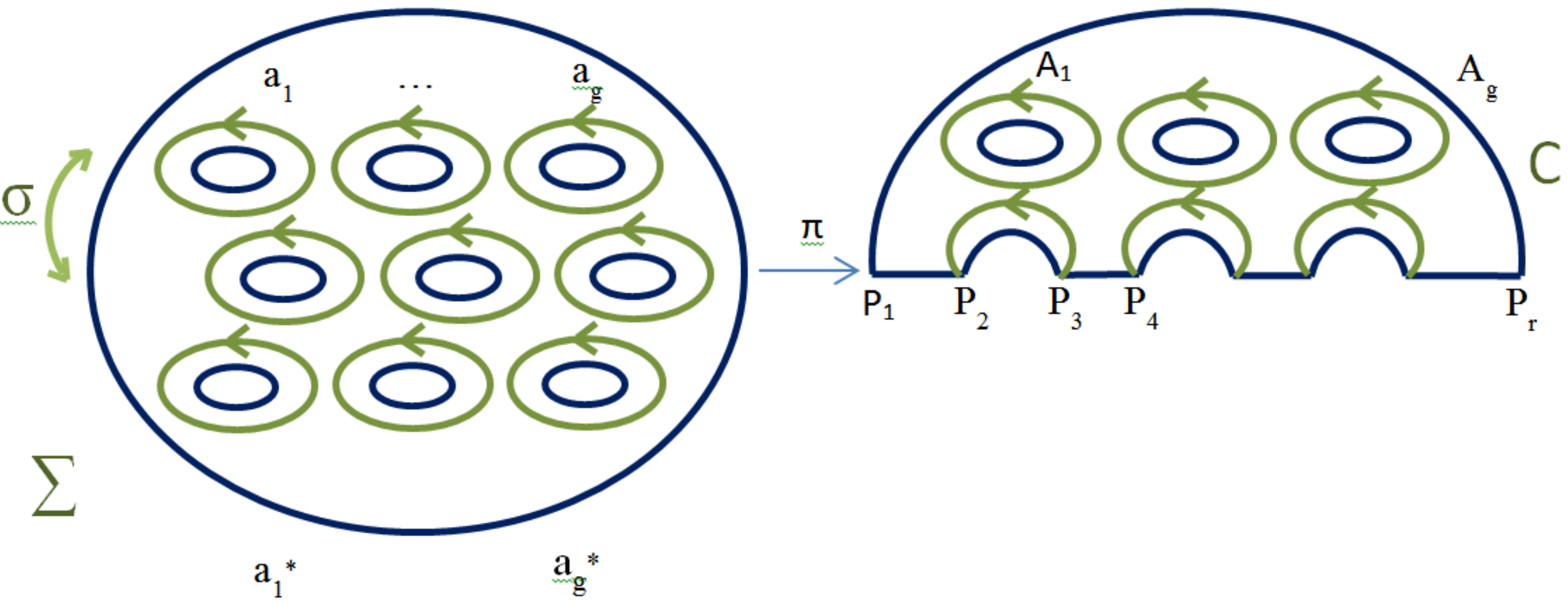}
\caption{The choice of a symplectic basis for
$H_1(\Sigma,\bZ)$ from that of 
$H_1(C,\bZ)$.}
\label{fig:Sigma}
\end{figure}

\subsection{The generalized 
topological recursion for the Hitchin
spectral curves}
\label{subsect:TR}

The topological recursion of \cite{BKMP1, EO2007},
that  has initiated the explosive developments
on this subject in 
recent years, is restricted to non-singular 
spectral curves
in $\bC^2$ or $(\bC^*)^2$. A systematic
generalization of the formalism to the 
case of Hitchin spectral curves is given 
for the first time in \cite{OM1, OM2}.
First, we gave the definition for non-singular
Hitchin spectral curves associated with 
holomorphic Higgs bundles in \cite{OM1}.
We then extended the consideration to 
meromorphic Higgs bundles and singular spectral 
curves in \cite{OM2}. In each case, however, the
actual evaluation of the generalized topological 
recursion for the purpose of quantization of 
the Hitchin spectral curves is limited to 
Higgs bundles of rank $2$. This is due to 
many technical difficulties, and at this point
we still do not have a better understanding 
of the theory in its full generality.

The purpose of this subsection is thus 
to present the theory in the way
 we know as of now, with the   scope limited to
  what seems to work.
 Many aspects of the story can be immediately
 generalized. Mainly for the sake of simplicity 
 of presentation, we concentrate on the case of
 rank $2$ 
 Higgs bundles.
 
 The geometric setup is the following. 
We have a smooth projective
curve $C$ defined over
$\bC$ of an arbitrary genus, and  a meromorphic
Higgs bundle $(E,\phi)$. Here, the vector bundle $E$
 of rank $2$   is a special one of the form
\be
\label{E general}
E = K_C^{\half}\dsum K_C^{-\half}.
\ee
As a meromorphic
Higgs field, we use
\be
\label{phi}
\phi =
\begin{bmatrix}
-s_1&s_2\\
1&
\end{bmatrix}: E\lrar K_C(D)\tensor E
\ee
with poles at an effective divisor $D$ of $C$,
where
\begin{align*}
s_1&\in H^0\big(C,K_C(D)\big),\\
s_2&\in H^0\big(C,K_C^{\tensor 2}(D)\big).
\end{align*}
Although $\phi$ involves a quadratic differential 
in its coefficient, 
since
\begin{align*}
\begin{bmatrix}
-s_1&s_2\\
1&
\end{bmatrix}:
K_C^{\half}\dsum K_C^{-\half}
\lrar
\left(K_C^{\frac{3}{2}}
\dsum K_C^\half\right) \tensor \cO_C(D)
\lrar K_C(D)\tensor \left(
K_C^{\half}\dsum K_C^{-\half}
\right),
\end{align*}
we have
$
\phi\in H^0\big(C,K_C(D)\tensor \End(E)\big).
$

\begin{rem}
We use a particular section $(E,\phi)$
 of the \emph{Hitchin
fibration} given by the form of \eqref{phi}.
This section is often called a \emph{Hitchin
section}. 
This choice is suitable
for the WKB analysis
explained below. The result of our 
quantization through WKB constructs a 
\textbf{formal} section of $K_C^{-\half}$,
and the relation to an $\hbar$-connection on $C$
makes our choice necessary.
The theory being developed as of now
(relation to opers 
\cite{DFKMMN}) also requires the choice 
of a Hitchin section.
\end{rem}

Denote by 
\be
\label{compact T*C}
\overline{T^*C}:= \bP(K_C\dsum \cO_C)
\overset{\pi}{\lrar} C
\ee
the compactified cotangent bundle of $C$
(see \cite{BNR, KS2013}),
which is a ruled surface on the base $C$.
The Hitchin spectral curve 
\begin{equation}
\label{spectral}
\xymatrix{
\Sigma \ar[dr]_{\pi}\ar[r]^{i} 
&\overline{T^*C}\ar[d]^{\pi}
\\
&C		}
\end{equation} 
 for a meromorphic
Higgs bundle is defined as the divisor 
of zeros on $\overline{T^*C}$
of the characteristic polynomial of
$\phi$:
\be
\label{char poly}
\Sigma = \left(\det(\eta - \pi^*\phi)\right)_0,
\ee
where $\eta \in H^0(T^*C, \pi^*K_C)$ is the 
tautological $1$-form on $T^*C$ extended as
a meromorphic $1$-form on the
compactification $\overline{T^*C}$.

\begin{Def}[Integral topological 
recursion for a degree $2$ 
covering]
\label{def:general TR}
Let  $\tilde{\pi}:\widetilde{\Sigma}\lrar C$ be a
degree $2$ covering of $C$
by a non-singular
curve $\widetilde{\Sigma}$. We denote by 
$R$ the
ramification divisor of $\tilde{\pi}$. In this
case the covering
$\tilde{\pi}$ is a Galois covering with
the Galois group 
$\bZ/2\bZ = \la \tilde{\sigma} \ra$,
and $R$ is the fixed-point divisor of the involution
$\tilde{\sigma}$.
The \textbf{integral topological recursion}
is an inductive 
mechanism of constructing meromorphic
differential forms $W_{g,n}$ on the Hilbert scheme 
$\widetilde{\Sigma}^{[n]}$
of  $n$-points on $\widetilde{\Sigma}$
for all $g\ge 0$ and 
$n\ge 1$ in the \emph{stable range}
$2g-2+n>0$, from  given initial data $W_{0,1}$
and $W_{0,2}$. 

\begin{itemize}
\item $W_{0,1}$ is a meromorphic $1$-form
on $\widetilde{\Sigma}$.

\item $W_{0,2}$ is defined to be
\begin{equation}
\label{W02}
W_{0,2}(z_1,z_2) = d_1d_2 
\log E_{\widetilde{\Sigma}}(z_1,z_2),
\end{equation}
where $E_{\widetilde{\Sigma}}(z_1,z_2)$ is 
the normalized Riemann prime
form on 
$\widetilde{\Sigma}\times \widetilde{\Sigma}$
(see \cite[Section~2]{OM1}).
\end{itemize}
Let $\omega^{a-b}(z)$ be a normalized
Cauchy kernel on $\widetilde{\Sigma}$
of Remark~\ref{rem:omega},
which has simple poles at $z=a$ of residue $1$ 
and at $z=b$ of residue $-1$. 
Then 
$$
d_1 \omega^{z_1-b}(z_2) = W_{0,2}(z_1,z_2).
$$
Define
\begin{equation}
\label{Omega}
\Omega := \tilde{\sigma}^*W_{0,1}-W_{0,1}.
\end{equation}
Then $\tilde{\sigma}^*\Omega = -\Omega$, hence
$\supp (R)\subset \supp(\Omega)$,
where $\supp(\Omega)$ denotes the support of both
zero and pole divisors of $\Omega$.
The inductive formula of the topological 
recursion is then given by the following:
\begin{multline}
\label{integral TR}
W_{g,n}(z_1,\dots,z_n)
=\half \frac{1}{2\pi \sqrt{-1}} 
\sum_{p\in \supp(\Omega)}
\oint_{\gam_p}\omega^{\tilde{z}-z}(z_1) 
\\
\times
\frac{1}{\Omega(z)}
\left[W_{g-1,n+1}(z,\tilde{z}, z_2,\dots,z_n)
+\sum_{\substack{g_1+g_2=g\\
I\sqcup J=\{2,\dots,n\}}}^{\rm{No }\; (0,1)}
W_{g_1,|I|+1}(z,z_I)W_{g_2,|J|+1}(\tilde{z},z_J)
\right].
\end{multline}
Here, 
\begin{itemize}
\item $\gam_p$ is a positively oriented small 
loop around a  point $p\in \supp (\Omega)$;
\item the integration is taken with respect to  
 $z\in \gam_p$ for each $p\in \supp (\Omega)$;
\item $\tilde{z} = \tilde{\sigma}(z)$ is the Galois
conjugate of $z\in \widetilde{\Sigma}$;
\item the ratio of two global meromorphic $1$-forms
on the same curve makes sense as a global 
meromorphic function. The operation 
$1/\Omega$ applied to a meromorphic
$1$-form produces this ratio;
\item ``No $(0,1)$'' means that
$g_1=0$ and $I=\emptyset$, or $g_2=0$ and
$J=\emptyset$, are excluded in the summation;
\item the sum runs over all partitions of $g$ and
set partitions of $\{2,\dots,n\}$, other than those 
containing the $(0,1)$ geometry; 
\item $|I|$ is the cardinality of the subset 
$I\subset \{2,\dots,n\}$; and
\item $z_I=(z_i)_{i\in I}$.
\end{itemize}
\end{Def}

The main idea of \cite{OM2} for dealing
with singular spectral curve is the 
following.

\begin{itemize}
\item The integral topological recursion  of 
\cite{OM1,EO2007} is extended to the
 curve $\Sigma$ of \eqref{char poly},
  as \eqref{integral TR}. 
For this purpose,
we blow up $\overline{T^*C}$
several times as in \eqref{blow-up general}
below to construct the normalization
$\widetilde{\Sigma}$.
The blown-up $Bl(\overline{T^*C})$ is
 the minimal 
resolution of the support $ \Sigma \cup C_\infty$
of the
\emph{total} divisor
 \be
 \label{total}
  \Sigma - 
2C_\infty = \left(\det(\eta - \pi^*\phi)\right)_0
-\left(\det(\eta - \pi^*\phi)\right)_\infty
 \ee
 of the characteristic polynomial, where
\begin{equation}
\label{C-infinity}
C_\infty := \bP(K_C\dsum\{0\}) = 
\overline{T^*C}\setminus T^*C
\end{equation}
is the divisor at infinity. Therefore,
in  $Bl(\overline{T^*C})$, 
the proper transform $\widetilde{\Sigma}$
of $\Sigma$
is smooth and does not intersect
with
the proper transform of $C_\infty$.
\begin{equation}
 \label{blow-up general}
\xymatrix{
\widetilde{\Sigma} 
\ar[dd]_{\tilde{\pi}} \ar[rr]^{\tilde{i}}\ar[dr]^{\nu}&&Bl(\overline{T^*C})
\ar[dr]^{\nu}
\\
&\Sigma \ar[dl]_{\pi}\ar[rr]^{i} &&
\overline{T^*C}  \ar[dlll]^{\pi}
\\
C 		}
\end{equation} 

\item The genus of the normalization 
$\widetilde{\Sigma}$ is given by
$$
g(\widetilde{\Sigma}) = 2g(C)-1+\half \delta,
$$
where $\delta$ is the sum of the  number of cusp singularities
of $\Sigma$ 
and the ramification points 
of $\pi:\Sigma\lrar C$.

\item  The generalized topological recursion 
Definition~\ref{def:general TR}
  requires 
a globally defined meromorphic $1$-form $W_{0,1}$
on $\widetilde{\Sigma}$ and 
a symmetric meromorphic $2$-form $W_{0,2}$
on the product $\widetilde{\Sigma}
\times \widetilde{\Sigma}$ as the initial data.
We choose
\begin{equation}
\begin{cases}
\label{W0102}
W_{0,1} = \tilde{i}^*\nu^*\eta\\
W_{0,2} = d_1d_2 \log E_{\widetilde{\Sigma}},
\end{cases}
\end{equation}
where
$E_{\widetilde{\Sigma}}$ is a normalized 
Riemann prime form on 
${\widetilde{\Sigma}}$. The form $W_{0,2}$
depends only on the intrinsic geometry of 
the smooth curve $\widetilde{\Sigma}$. The
geometry of \eqref{blow-up general}
is  encoded in $W_{0,1}$. 
The integral topological recursion
 produces a symmetric meromorphic $n$-linear
differential form $W_{g,n}(z_1,\dots,z_n)$
on $\widetilde{\Sigma}$ for every $(g,n)$
subject to $2g-2+n>0$ from the initial data 
\eqref{W0102}.
\end{itemize}

The key discovery of \cite[(4.7)]{OM1} is that
we should use a partial differential equation 
version of the topological recursion, instead of
\eqref{integral TR}, to construct a quantization
of $\Sigma$.

\subsection{Quantization of Hitchin spectral curves}
\label{sect:quantum}

The passage from the geometry
of Hitchin spectral curve $\Sigma$ of \eqref{char poly}
 to the quantum curve
 \be
\label{Sch}
\left(\left(\hbar\frac{d}{dx}\right)^2
-\tr \, \phi(x) \left(\hbar\frac{d}{dx}\right)
+ \det \phi(x))\right)
\Psi(x,\hbar)=0
\ee
 is a system of PDE recursion 
replacing the
integration formula
 \eqref{integral TR}. 

\begin{Def}[Free energies]
\label{def:free energy}
The \textbf{free energy} of type $(g,n)$ is
a function $F_{g,n}(z_1,\dots,z_n)$
defined on the universal covering
$\cU^n $ of $\widetilde{\Sigma}^n$ such that
$$
d_1\cdots d_n F_{g,n} = W_{g,n}.
$$
\end{Def}

\begin{rem}
The free energies may contain 
logarithmic singularities, since it is an integral
of a meromorphic function. 
For example, $F_{0,2}$ is the Riemann prime
form itself considered as a function on 
$\cU^2$, which has logarithmic singularities 
along the diagonal \cite[Section~2]{OM1}.
\end{rem}

\begin{Def}[Differential  recursion for
a degree $2$ covering]
The \textbf{differential  
recursion} is the
following partial differential equation
for all $(g,n)$ subject to $2g-2+n\ge 2$:
\begin{multline}
\label{differential TR}
d_1 F_{g,n}(z_1,\dots,z_n)
\\
= \sum_{j=2}^n 
\left[
\frac{\omega^{z_j-\sigma(z_j)}
(z_1)}{\Omega(z_1)}\cdot 
d_1F_{g,n-1}\big(z_{[\hat{j}]}\big)
-
\frac{\omega^{z_j-\sigma(z_j)}
(z_1)}{\Omega(z_j)}\cdot 
d_jF_{g,n-1}\big(z_{[\hat{1}]}\big)
\right]
\\
+\frac{1}{\Omega(z_1)}
d_{u_1}d_{u_2}
\left.
\left[F_{g-1,n+1}
\big(u_1,u_2,z_{[\hat{1}]}\big)
+\sum_{\substack{g_1+g_2=g\\
I\sqcup J=[\hat{1}]}}^{\text{stable}}
F_{g_1,|I|+1}(u_1,z_I)F_{g_2,|J|+1}(u_2,z_J)
\right]
\right|_{\substack{u_1=z_1\\u_2=z_1}}.
\end{multline}
Here, $1/\Omega$ is again the ratio operation,
and the  index subset $[\hat{j}]$ denotes
the exclusion of $j\in \{1,2,\dots,n\}$.
The Cauchy integration kernel $\omega^{a-b}(z)$
on the spectral curve $\widetilde{\Sigma}$ is
normalized  differently than the $A$-cycle
normalization we did earlier. This time we 
impose that 
\be
\label{normalization}
\lim_{a\rar b} \omega^{a-b}(z) = 0.
\ee
\end{Def}

\begin{rem}
As pointed out in \cite[Remark 4.8]{OM1}, 
\eqref{differential TR} is a globally defined
coordinate-free 
equation, written in terms of exterior differentiations
and the ratio operation,
on  $\widetilde{\Sigma}$.
\end{rem}

\begin{thm}[The relation between 
the differential recursion and the  integral
recursion, \cite{OM2}]
Suppose that $F_{g,n}$ 
for $2g-2+n>0$ are  globally meromorphic
on $\widetilde{\Sigma}^{[n]}$ with poles 
located only along the divisor of 
$\widetilde{\Sigma}^{[n]}$ when one of the factors 
lies in the  zeros of 
$\Omega$.  Define
$W_{g,n} := d_1\cdots d_n F_{g,n}$
for $2g-2+n>0$, and 
use \eqref{W02} and \eqref{Omega}
for $(g,n)$ in the unstable range. If 
$F_{g,n}$s are symmetric and satisfy the differential  
recursion \eqref{differential TR}, and if
$W_{1,1}$ and $W_{0,3}$ satisfy
the initial equations
of the integral topological recursion
\eqref{integral TR}, then 
$W_{g,n}$s for all valued of 
$(g,n)$ satisfy the integral topological 
recursion. 
\end{thm}

\begin{rem}
The assumption of the theorem holds
for $g(C)=0$,  and therefore, for all
the examples we discuss in these lectures. But
we are not establishing the 
general \textbf{equivalence}
of the integral topological recursion
\eqref{integral TR} and 
the PDE recursion 
\eqref{differential TR}. They are  never
equivalent when $g(\widetilde{\Sigma})>0$. 
Actually, what is assumed in the
above theorem is that $W_{g,n}$ is exact,
i.e., integrable by definition. In particular, 
this implies that $W_{g,n}$ has $0$-period for
every topological cycle. This does not happen
if we start with \eqref{integral TR} in general.
Therefore, the above theorem serves  only as
a heuristic motivation for our discovery of
\eqref{differential TR} in \cite{OM1,OM2}.
\end{rem}

\begin{proof}
Although the context of the statement is slightly
different, the proof is essentially the same as that
of
\cite[Theorem~4.7]{OM1}. The crucial assumption
we have made is that $\tilde{\pi}:\widetilde{\Sigma}
\lrar C$ is a Galois covering. Therefore,
the Galois conjugation $\tilde{\sigma}:
\widetilde{\Sigma}\lrar \widetilde{\Sigma}$ is a
globally defined holomorphic mapping.
To calculate the residues 
in the integral recursion \eqref{integral TR},
we need the global analysis of 
$$
\omega^{\tilde{z}-z}(z_1)\in
H^0\big(\widetilde{\Sigma}\times \widetilde{\Sigma},
q_2^*K_{\widetilde{\Sigma}}
\boxtimes q_2^*\cO_{\widetilde{\Sigma}}
(\tilde{z}+z)\big),
$$
where $q_1$ and $q_2$ are projections
\begin{equation}
\label{omega analysis}
\xymatrix{
&\widetilde{\Sigma}\times \widetilde{\Sigma}
\ar[dl]_{q_1}\ar[dr]^{q_2}
\\
z\in \widetilde{\Sigma}&&
 \widetilde{\Sigma}\owns z_1		.}
\end{equation} 
The residue integration is done at each 
point  $p\in \supp(\Omega)$.  
The poles of 
the integrand of \eqref{integral TR}
that are enclosed in the 
union 
$$
\gam = \bigcup_{p\in \supp(\Omega)}\gam_p
$$
of the contours on the complement of 
$\supp(\Omega)$
 are located at 
\begin{enumerate}
\item $z=z_1,z=\tilde{z_1}$ from 
$\omega^{\tilde{z}-z}(z_1)$; and
\item $z=z_j,z=\tilde{z_j}$, $j=2,\dots,n$,
from $W_{0,2}(z,z_j)$ and $W_{0,2}(\tilde{z},z_j)$
that appear in the second line of 
\eqref{integral TR}.
\end{enumerate}
The integrand has other poles at 
$\supp(\Omega)$ that includes the ramification 
divisor $R$, but they are not enclosed in $\gam$.
The local behavior of $\omega^{\tilde{z}-z}(z_1)$
at $z=z_1,z=\tilde{z_1}$ is well understood, 
and residues of the integrand of
\eqref{integral TR} are simply the evaluation 
of the differential form at $z=z_1,z=\tilde{z_1}$.
The double poles coming from 
$W_{0,2}(z,z_j)$ and $W_{0,2}(\tilde{z},z_j)$
contribute as differentiation of the factor it is 
multiplied to. Adding all contributions
from the poles, we obtain
\eqref{differential TR}.
\end{proof}

Now let us consider a spectral curve 
$\Sigma\subset 
\overline{T^*C}$
 of \eqref{spectral}
defined by a pair of meromorphic sections
$s_1=-\tr\phi$ of $K_C$ and 
$s_2=\det\phi$ of $K_C^{\tensor 2}$.
Let
$\widetilde{\Sigma}$ be the desingularization of 
$\Sigma$ in \eqref{blow-up}.
We  apply the topological recursion
 \eqref{integral TR}
to the covering $\tilde{\pi}:\widetilde{\Sigma}\lrar C$.
 The geometry of the 
spectral curve $\Sigma$ provides us with a
canonical choice of the initial differential forms
\eqref{W0102}. 
At this point we pay  special attention to the fact that 
the topological recursions \eqref{integral TR}
and \eqref{differential TR} are both defined 
on the spectral curve 
$\widetilde{\Sigma}$, while we wish to 
construct a differential equation on $C$. 
Since the free energies are defined on the
universal covering of $\widetilde{\Sigma}$,
we need to have a mechanism to relate a
coordinate on the desingularized spectral curve
and that of the base curve $C$.

To analyze the singularity structure of
$\Sigma$, let us consider the discriminant of the 
defining equation \eqref{r=2} of the
spectral curve.

 \begin{Def}[Discriminant divisor]
The \textbf{discriminant divisor} of
the spectral curve 
\be
\label{spectral general}
\eta^{\tensor 2}+\pi^*s_1\eta +\pi^* s_2=0
\ee
 is a
 divisor on $C$ defined by
\begin{equation}
\label{discriminant}
\Delta:=\left(\frac{1}{4}s_1^2 -s_2
\right) = \Delta_0-\Delta_\infty.
\end{equation}
Here,
\be
\label{Delta 0}
\Delta_0= \sum_{i=1}^m m_iq_i , \quad m_i>0,
\quad q_i\in C,
\ee
is the divisor of zeros, and 
\be
\label{Delta infinity}
\Delta_\infty= \sum_{j=1}^n n_jp_j, \quad n_j>0,
\quad p_j\in C,
\ee
is the divisor of $\infty$.
\end{Def}

Since 
$
\frac{1}{4}s_1^2 -s_2
$
is a meromorphic section of $K_C^{\tensor 2}$,
we have
\begin{equation}
\label{deg Delta}
\deg \Delta = 
\sum_{i=1}^m m_i - \sum_{j=1}^n  n_j  = 4g-4.
\end{equation}
 
 \begin{thm}[Geometric genus formula, \cite{OM2}]
\label{thm:geometric genus formula}
Let us define an invariant of the discriminant 
divisor by
\begin{equation}
\label{delta}
\delta=|\{i\;|\; m_i \equiv 1 \mod 2\}|+
|\{j\;|\; n_j \equiv 1 \mod 2\}|.
\end{equation}
Then    
the geometric genus of the 
spectral curve $\Sigma$ 
of \eqref{spectral general}
is given by
\begin{equation}
\label{pg}
g(\widetilde{\Sigma})
:= p_g(\Sigma) = 2g-1+\half \delta.
\end{equation}
We note that   \eqref{deg Delta}
 implies  $\delta \equiv 0\mod 2$.
\end{thm}

Take an arbitrary point $p\in C\setminus 
\supp(\Delta)$, and a local coordinate $x$ around 
$p$. 
By choosing a small 
disc $V$ 
around $p$, we can make the inverse image of 
$\tilde{\pi}:\widetilde{\Sigma}\lrar C$ 
consist of two isomorphic discs. Since $V$ is
away from the critical values of $\tilde{\pi}$,
the inverse image consists of two discs in 
the original spectral curve $\Sigma$. 
Note that we choose an eigenvalue $\a$ of
$\phi$ on $V$ in our main construction. 
We are thus specifying one of the inverse image discs
here. Let us name the disc $V_\a$ that
corresponds to 
$\a$.

At this point apply the WKB analysis to the differential
equation \eqref{Sch} with 
 the WKB expansion of the solution
\begin{equation}
\label{WKBa}
\Psi^\a (x,\hbar) = \exp\left(\sum_{m=0}^\infty
\hbar^{m-1} S_m\big(x(z)\big)\right) = 
\exp F^\a(x,\hbar),
\end{equation}
where we choose a coordinate $z$ of $V_\a$
so that the function $x = x(z)$ represents the 
projection $\pi:V_\a \lrar V$.
The equation $P\Psi^\a = P e^{F^\a} = 0$ reads
\begin{equation}
\label{F alpha}
\hbar^2 \frac{d^2}{dx^2}F^\a + 
\hbar^2\frac{dF^\a}{dx}
\frac{dF^\a}{dx} +s_1 \hbar \frac{dF^\a}{dx} +s_2=0.
\end{equation}
The $\hbar$-expansion of \eqref{F alpha} 
gives
\begin{align}
\label{scl}
&\hbar^0{\text{-terms}}:
\quad (S_0'(x))^2 +s_1S_0'(x) + s_2=0,
\\
\label{consistency}
&\hbar^1{\text{-terms}}:\quad
2S_0'(x)S_1'(x) + S_0''(x)+s_1S_1'(x)=0,
\\
\label{h m+1}
&\hbar^{m+1}{\text{-terms}}:\quad
S_m''(x) +\sum_{a+b=m+1}
S_a'(x)S_b'(x)+s_1S_{m+1}'(x)=0, \quad m\ge 1,
\end{align}
where $'$ denotes the $x$-derivative.
The WKB method is to solve these equations
iteratively and find $S_m(x)$ for all $m\ge 0$. Here,
\eqref{scl} is the
\textbf{semi-classical limit}
of \eqref{Sch}, and 
\eqref{consistency} is the 
\emph{consistency condition} we need to solve
the WKB expansion, the same as before.
Since the $1$-form $dS_0(x)$ is a local section
of $T^*C$, we identify $y=S_0'(x)$. Then 
\eqref{scl} is the local expression of the 
spectral curve equation 
\eqref{r=2}.
This expression is the same everywhere
for $p\in C\setminus \supp(\Delta)$. We note 
$s_1$ and $s_2$ are globally defined.
Therefore, we
recover the spectral curve $\Sigma$ from 
the differential operator of \eqref{Sch}.

\begin{thm}[Main theorem]
The differential topological recursion provides a 
formula for each $S_m(x)$, 
$m\ge 2$, and constructs a formal solution 
to the quantum curve \eqref{Sch}.
\begin{itemize}

\item The quantum curve associated with 
the Hitchin spectral curve $\Sigma$ is defined
as a differential equation  
 on $C$. On each 
coordinate neighborhood $U\subset C$ with
coordinate $x$, a generator of the quantum curve
is given by 
$$
P(x,\hbar) = \left(\hbar\frac{d}{dx}\right)^2
-\tr \, \phi(x) \left(\hbar\frac{d}{dx}\right)
+ \det \phi(x).
$$
In particular, the \textbf{semi-classical limit}
of the quantum curve recovers the singular
spectral curve $\Sigma$, not its normalization
$\widetilde{\Sigma}$.

\item 
The \textbf{all-order WKB expansion} 
\be
\label{WKB}
\Psi(x,\hbar) = \exp\left(\sum_{m=0}^\infty
\hbar^{m-1} S_m(x)\right)
\ee
of a solution to the 
Schr\"odinger equation
$$
\left(\left(\hbar\frac{d}{dx}\right)^2
-\tr \, \phi(x) \left(\hbar\frac{d}{dx}\right)
+ \det \phi(x))\right)
\Psi(x,\hbar)=0,
$$
near  each critical value of 
${\pi}:{\Sigma}\lrar C$,
can be obtained by 
the  \textbf{principal 
specialization} of the differential recursion
\eqref{differential TR}, 
after determining the first three terms.
The procedure is the following. We determine
$S_0, S_1$, and $S_2$ 
 by  solving
\be
\begin{aligned}
\label{S012}
\left(S_0'(x)\right)^2 -\tr\phi(x) S_0'(x)+\det \phi(x) 
&=0,\\
2 S_0'(x)S_1'(x)+S_0''(x)-\tr \phi(x)S_1'(x)
&=0,
\\
S_1''(x)+\sum_{a+b=2}S_a'(x)S_b'(x)
-\tr \phi(x) S_2'(x)
&=0.
\end{aligned}
\ee
Then find $F_{1,1}(z)$ and $F_{0,3}(z_1,z_2,z_3)$
so that 
$$
S_2(x) = F_{1,1}\big(z(x)\big)+\frac{1}{6}
F_{0,3}\big(z(x),z(x),z(x)\big).
$$
This can be achieved as follows. 
First  integrate
$W_{1,1}(z)$ of \eqref{integral TR} to 
construct $F_{1,1}(z)$. We do the same
for the solution 
$W_{0,3}(z_1,z_2,z_3)$ of \eqref{integral TR}.
We now define
  \be
  \label{F03}
  F_{0,3}(z_1,z_2,z_3) 
  = \int\!\!\int\!\!\int W_{0,3}(z_1,z_2,z_3)
  + 2\left(f(z_1)+f(z_2)+f(z_3)\right),
  \ee
 where 
 \be
 \label{adjust}
 f(z) := \widetilde{S}_2(z)
 - \left(F_{1,1}(z) +
 \frac{1}{6}\int^z\!\!\int^z\!\!\int^z
  W_{0,3}(z_1,z_2,z_3)\right),
 \ee
 and $\widetilde{S}_2(z)$ is the lift of $S_2(x)$ to
 $\widetilde{\Sigma}$.

 \item Suppose we have symmetric meromorphic
functions $F_{g,n}(z_1,\dots,z_n)$ 
that solve the differential 
recursion \eqref{differential TR} on
the universal covering
$\varpi:\cU\lrar \widetilde{\Sigma}$
with these $F_{1,1}$ and $F_{0,3}$
as  initial values.

\item Let
\be
\label{Sm in Fgn}
S_m(x) = \sum_{2g-2+n=m-1} \frac{1}{n!}
F_{g,n}\big(z(x)\big),\qquad m\ge 3,
\ee
where $F_{g,n}\big(z(x)\big)$ is the principal 
specialization of $F_{g,n}(z_1,\dots,z_n)$
evaluated at a local section $z=z(x)$ of
$\tilde{\pi}:\widetilde{\Sigma}\lrar C$.
Then the \textbf{wave function} $\Psi(x,\hbar)$,
a formal section of the line bundle
$K_C^{-\half}$ on $C$, solves
\eqref{Sch}.

\item The canonical ordering of the quantization of
the local functions on $T^*C$ is automatically
chosen in  \eqref{differential TR}
and the principal specialization \eqref{Sm in Fgn}.
This selects the canonical ordering in 
\eqref{Sch}.

\end{itemize}
\end{thm}

\begin{rem}
We do not have a closed formula for $F_{1,1}$ and 
$F_{0,3}$ from the given geometric data. 
Except for the case of 
$g(C)=0$, they are \textbf{not} given by integrating 
$W_{1,1}$ and $W_{0,3}$ of the integral
topological recursion.
\end{rem}

\begin{rem}
The differential recursion \eqref{differential TR}
assumes $F_{1,1}$ and $F_{0,3}$ as the initial
values. The equation itself does not give any
condition for them. The discovery of \cite{OM1,OM2}
is that the WKB equations for $S_m(x)$
are consequences of 
\eqref{differential TR} for all $m\ge 2$.
We note that there is an alternative way 
of constructing a quantization of the spectral curve.
From the geometric data, first choose
$W_{0,1}$ and $W_{0,2}$ as in \eqref{W0102},
and solve the integral topological recursion 
\eqref{integral TR}. Then define a set of alternative
free energies by
\be
\label{alt}
F_{g,n}^{\rm{alt}}(z_1,\dots,z_n)
= \int\!\!\cdots\!\!\int W_{g,n}(z_1,\dots,z_n)
\ee
for all values of $(g,n)$. Then use the same
\eqref{Sm in Fgn} and \eqref{WKB}
to define a wave function $\Psi^{\rm{alt}}(x,\hbar)$.
The differential equation that annihilates
this alternative wave function is another
quantum curve. We emphasize that 
$\Psi^{\rm{alt}}(x,\hbar)$ \textbf{does not satisfy}
our quantum curve equation \eqref{Sch}.
It is obvious because our definition 
\eqref{adjust} of $S_2$
 is different. The alternative
quantum curve is a second order differential 
equation, but it cannot be given by a 
closed formula, unlike \eqref{Sch}. It is also 
noted that every coefficient of this alternative
differential operator contains terms 
depending on all orders of $\hbar$. 
Therefore, the mechanism described in these
lecture notes provides a totally different
notion of quantum curves. We have shown 
that the differential recursion \eqref{differential TR}
is the passage from the starting spectral curve
to the quantum curve
\eqref{Sch}. This picture is consistent
with the construction of opers in \cite{DFKMMN}
and a physics point of view \cite{Teschner}.
\end{rem}

\subsection{Classical differential equations}
\label{subsect:classical}

If quantum curves are natural objects, then 
where do we see them in classical mathematics?
Indeed, they appear as classical differential equations.
Riemann and Poincar\'e 
found 
the interplay between algebraic geometry 
of  curves in a ruled surface and 
the asymptotic expansion
   of an analytic solution 
   to a differential equation defined on the 
base curve of the ruled surface.  
We look at these classical subjects from a 
new point of view.
Let us now  recall the definition of 
regular and irregular singular points  of 
a second order differential
equation.

\begin{Def}
\label{def:regular and irregular} 
Let
\begin{equation}
\label{second}
\left(\frac{d^2}{dx^2}+s_1(x)\frac{d}{dx}+s_2(x)
\right)\Psi(x) = 0
\end{equation}
be a second order differential equation
defined around a neighborhood of $x=0$ on a
small disc $|x|< \epsilon$ with meromorphic 
coefficients $s_1(x)$ and $s_2(x)$ with poles  
at $x=0$. Denote by $k$ (reps.\ $\ell$) the 
order of the pole of 
$s_1(x)$ (resp.\ $s_2(x)$) at $x=0$. 
If $k \le 1$  and $\ell\le 2$, then \eqref{second}
has a \textbf{regular singular point} at $x=0$.
Otherwise, consider the \emph{Newton polygon}
 of the order of poles of the coefficients of
 \eqref{second}. It is the upper part of
 the convex hull of three
 points $(0,0), (1, k), (2,\ell)$. As a convention,
 if $s_j(x)$ is identically $0$, then
 we assign $-\infty$ as its pole order. Let $(1,r)$
 be the intersection point of
 the Newton polygon and the line $x=1$.
 Thus
 \begin{equation}
 \label{irregular class}
r= \begin{cases}
 k \qquad 2k\ge \ell,\\
\frac{\ell}{2} \qquad 2k\le \ell.
 \end{cases}
 \end{equation}
 The differential equation \eqref{second} 
 has an \textbf{irregular singular point of class}
 $r-1$ at $x=0$ if $r>1$.
\end{Def}

To illustrate the scope of interrelations among the
geometry of meromorphic 
Higgs bundles, their spectral curves, the singularities
of  quantum curves, $\hbar$-connections,
and the quantum invariants, let us tabulate 
five examples here (see Table~\ref{tab:examples}).
The differential operators of these equations 
are listed in the third column. 
In the first three rows, the quantum curves are
examples of 
classical differential equations known as 
 Airy,  Hermite,  the Gau\ss\ hypergeometric
equations. 
The fourth and the fifth
rows are added to show that it is \emph{not}
the singularity of the spectral curve that
determines the singularity
of the quantum curve.
In each example, the Higgs bundle $(E,\phi)$ 
we are 
considering consists of the base curve $C=\bP^1$
and the rank $2$ vector bundle $E$  on ${\bP^1}$
of \eqref{E}. For this situation, the two 
topological recursions \eqref{integral TR}
and \eqref{differential TR} are equivalent.

\begin{table}[htb]
\label{tab:examples}
  \centering
  
  \begin{tabular}{|c|c|c|}

\hline 

Higgs Field & Spectral Curve  & Quantum Curve
\tabularnewline
\hline \hline
$\begin{bmatrix}
&x(dx)^2\\
1
\end{bmatrix}$ & 
$\begin{matrix}
y^2-x=0\\
w^2-u^5=0\\
\Sigma = 2C_0+5F\\
p_a=2,p_g=0
\end{matrix}
$ & 
$\begin{matrix}
\text{Airy}\\
\left(\hbar\frac{d}{dx}\right)^2 - x\\
\text{Class $\frac{3}{2}$ irregular singularity}\\
\text{at $\infty$}
\end{matrix}
$
 \tabularnewline
\hline 
$\begin{bmatrix}
-xdx&-(dx)^2\\
1
\end{bmatrix}$ & 
$\begin{matrix}
y^2+xy+1=0\\
w^2-uw+u^4=0\\
\Sigma = 2C_0+4F\\
p_a=1, p_g=0
\end{matrix}$ &
$\begin{matrix}
\text{Hermite}\\
\left(\hbar\frac{d}{dx}\right)^2 +x\hbar\frac{d}{dx}
+1\\
\text{Class $2$ irregular singularity}\\
\text{at $\infty$}
\end{matrix}$
  \tabularnewline
\hline 
 $\begin{bmatrix}
\frac{2x-1}{x(1-x)}dx&\frac{(dx)^2}{4(1-x)}\\ \\
\frac{1}{x}
\end{bmatrix}$ & 
$\begin{matrix}
y^2
+\frac{2x-1}{x(x-1)}y+\frac{1}{4x(x-1)}=0
\\
w^2+4(u-2)uw
\\
-4u^2(u-1)=0
\\
\Sigma = 2C_0+4F\\
p_a=1,p_g=0
\end{matrix}
$ &
$\begin{matrix}
\text{Gau\ss\ Hypergeometric}\\
\left(\hbar\frac{d}{dx}\right)^2
+\frac{2x-1}{x(x-1)}
\hbar\frac{d}{dx}+\frac{1}{4x(x-1)}\\
\text{Regular singular points}\\
\text{at $x=0,1,\infty$}
\end{matrix}
$
  \tabularnewline
\hline 

 $\begin{bmatrix}
-dx&-\frac{(dx)^2}{x+1}\\
1
\end{bmatrix}$ & 
$
\begin{matrix}
y^2+y+\frac{1}{x+1}=0
\\
w^2-u(u+1)w\\
+u^3(u+1)=0
\\
\Sigma = 2C_0+4F\\
p_a=1,p_g=0
\end{matrix}
$ &
$\begin{matrix}
\left(\hbar\frac{d}{dx}\right)^2+ \hbar\frac{d}{dx}
+\frac{1}{x+1}\\
\text{Regular singular point at $x=-1$}\\
\text{and a class $1$ irregular singularity}\\
\text{at $x=\infty$}
\end{matrix}
$

  \tabularnewline
\hline 

$\begin{bmatrix}
-\frac{2x^2}{x^2-1}dx&\frac{(dx)^2}{x^2-1}\\
1
\end{bmatrix}$ & 
$
\begin{matrix}
(x^2-1)y^2+2x^2y-1=0
\\
\text{non-singular}
\\
\Sigma = 2C_0+4F\\
p_a=p_g=1
\end{matrix}
$ &
$\begin{matrix}
\left(\hbar\frac{d}{dx}\right)^2+
2\frac{x^2}{x^2-1} \hbar\frac{d}{dx}
-\frac{1}{x^2-1}\\
\text{Regular singular points at $x=\pm1$}\\
\text{and a class $1$ irregular singularity}\\
\text{at $x=\infty$}
\end{matrix}
$

  \tabularnewline
\hline 

\end{tabular}
\bigskip

  \caption{Examples of quantum curves.}
\end{table}

The first column of the table shows the Higgs field
$\phi:E\lrar K_{\bP^1}(5)\tensor E$. 
Here, $x$ is the affine coordinate
of $\bP^1\setminus \{\infty\}$. 
Since our vector bundle is a
specific direct sum of 
line bundles, the quantization is simple in each case,
due to the fact that the $\hbar$-deformation
$E_\hbar$ of $E$ satisfies the
condition as described in \eqref{Ehbar}. 
Thus our quantum curves are equivalent to
$\hbar$-connections
in the trivial bundle. 
Except for the Gau\ss\ hypergeometric case, 
the connections are
given by
\begin{equation}
\label{hbar-connection intro}
\nabla^\hbar = \hbar d-\phi,
\end{equation}
where $d$ is the exterior differentiation operator
acting on  the trivial bundle $E_\hbar$, $\hbar\ne 0$.

For the third example of a
Gau\ss\ hypergeometric equation,
we use a particular choice of parameters
so that 
the $\hbar$-connection becomes  an $\hbar$-deformed
Gau\ss-Manin connection
 of \eqref{Gauss-Manin}.
 More precisely,
  for every
  $x\in \cM_{0,4}$, we consider the elliptic
  curve $E(x)$ ramified over $\bP^1$ at four points
  $\{0,1,x,\infty\}$, and its two periods 
   given by the elliptic integrals
  \cite{KZ}
  \begin{equation}
\label{elliptic periods}
\omega_1(x)=
\int_1^\infty \frac{ds}{\sqrt{s(s-1)(s-x)}},
\qquad
\omega_2(x) = 
\int_x^1 \frac{ds}{\sqrt{s(s-1)(s-x)}}.
\end{equation}
The quantum curve in this case is 
  an $\hbar$  \textbf{-deformed meromorphic
Gau\ss-Manin connection}
\begin{equation}
\label{Gauss-Manin}
\nabla^\hbar_{GM} =\hbar d -
\begin{bmatrix}
\left(-\frac{2x-1}{x(x-1)}
+\frac{\hbar}{x}\right)dx&-\frac{(dx)^2}{4(x-1)}
\\ \\
\frac{1}{x}
\end{bmatrix}
\end{equation}
in the $\hbar$-deformed 
vector bundle $K_{\Mbar_{0,4}}^{\half}
\dsum K_{\Mbar_{0,4}}^{-\half}$ of rank $2$
over $\Mbar_{0,4}$. Here, $d$ again denotes
the exterior differentiation acting on this 
trivial  vector bundle.
 The restriction $\nabla^1_{GM}$ 
 of the connection at $\hbar=1$ is equivalent to
the Gau\ss-Manin connection that
characterizes 
   the two periods of \eqref{elliptic periods},
   and the Higgs field is the classical limit
   of the connection matrix
   at $\hbar\rar 0$:
   \begin{equation}
   \label{Gauss-Manin-Higgs}
\phi =    \begin{bmatrix}
-\frac{2x-1}{x(x-1)}dx&-\frac{(dx)^2}{4(x-1)}
\\ \\
\frac{1}{x}
\end{bmatrix}.
\end{equation}
  The spectral curve 
  $\Sigma \subset \overline{T^*\Mbar_{0,4}}$ 
  as a moduli space consists
  of the data 
  $\big(E(x), \a_1(x), \a_2(x)\big)$,
  where $\a_1(x)$ and $\a_2(x)$ are the two
  eigenvalues of the  Higgs field $\phi$. 
  The spectral curve 
  $\Sigma \subset \overline{T^*\Mbar_{0,4}} = \bF^2$ 
  as a divisor in the Hirzebruch surface 
  is  determined by the characteristic 
  equation 
  \be
  \label{char intro}
  y^2 + \frac{2x-1}{x(x-1)}y
  + \frac{1}{4x(x-1)} = 0
  \ee
  of the Higgs field. Geometrically,
  $\Sigma$ is a singular rational curve with one
  ordinary double point at $x=\infty$.
  The quantum curve
  is a \textbf{quantization} of the characteristic equation
  \eqref{char intro}
  for the eigenvalues $\a_1(x)$ and $\a_2(x)$ 
 of $\phi(x)$. It is an $\hbar$-deformed
   \textbf{Picard-Fuchs equation}
   $$
   \left(\left(\hbar \frac{d}{dx}\right)^2 
   + \frac{2x-1}{x(x-1)} \left(\hbar \frac{d}{dx}\right)
   + \frac{1}{4x(x-1)}\right) \omega_i(x,\hbar) =0,
   $$
    and its semi-classical limit agrees with
   the singular spectral curve
  $\Sigma$. As a second order differential equation,
  the quantum curve has two independent
  solutions corresponding
  to the two eigenvalues. At $\hbar=1$, 
  these solutions are
  exactly the two periods $\omega_1(x)$ and
  $\omega_2(x)$ of the Legendre family of 
  elliptic curves $E(x)$. 
  The topological recursion 
  produces  asymptotic
  expansions of these periods  as
  functions in $x\in \Mbar_{0,4}$, at which the
  elliptic
  curve $E(x)$ degenerates to a nodal rational 
  curve.

This is a singular connection with simple poles
at $0,1,\infty$, 
and has an explicit $\hbar$-dependence in
the connection matrix. 
The Gau\ss-Manin connection $\nabla^1_{GM}$ at 
$\hbar = 1$ is equivalent to
the Picard-Fuchs equation
that characterizes the periods \eqref{elliptic periods}
of the Legendre family of elliptic curves $E(x)$
defined by the cubic equation
\begin{equation}
\label{Legendre}
t^2 = s(s-1)(s-x), \qquad x\in \cM_{0,4} = 
\bP^1\setminus \{0,1,\infty\}.
\end{equation}

The second column gives the spectral curve 
of the Higgs bundle $(E,\phi)$. 
Since the Higgs fields have 
poles, the spectral curves are no longer contained in 
the cotangent bundle $T^*\bP^1$. We need  
 the compactified cotangent bundle
\begin{equation*}
\overline{T^*\bP^1} = \bP(K_{\bP^1}\dsum
\cO_{\bP^1}) = \bF_2,
\end{equation*}
which is a Hirzebruch surface.
The parameter $y$
is the fiber coordinate of the cotangent line
$T^*_x\bP^1$. 
The first line of the second column is the
equation of the spectral curve in the $(x,y)$ 
affine coordinate of $\bF_2$.
All but the last 
example produce a singular
spectral curve. 
Let $(u,w)$ be a coordinate system
 on another affine
chart of $\bF_2$ defined by
\begin{equation*}
\begin{cases}
x = {1}/{u}\\
ydx = v du, \qquad w = 1/v.
\end{cases}
\end{equation*}
The singularity of  $\Sigma$ 
in the $(u,w)$-plane is given by the 
second line of the second column. 
The third line of the second column gives
$\Sigma\in \NS(\bF_2)$ as an element of 
the N\'eron-Severy group of 
$\bF_2$. Here, $C_0$ is the class of the 
zero-section of $T^*\bP^1$, and $F$ represents
the fiber class of $\pi:\bF_2\lrar \bP^1$. 
We also give the arithmetic and geometric
genera of the spectral curve.

A solution $\Psi(x,\hbar)$
to the first 
example is  given
by the \textbf{Airy function}
\begin{equation}
\label{Airy}
Ai(x,\hbar) =\frac{1}{2\pi} \hbar^{-\frac{1}{6}}
\int_{-\infty} ^\infty
\exp\left({\frac{ipx}{\hbar^{2/3}}}+{i\frac{p^3}{3}}
\right)dp,
\end{equation}
which is an entire function in $x$ for $\hbar\ne 0$,
as discussed earlier in these lectures. 
 The expansion 
coordinate $x^{\frac{3}{2}}$ of \eqref{Fgn Airy}
indicates the class of the irregular singularity of
the Airy differential equation.

The solutions to the second example are given by 
confluent hypergeometric functions, such as
${}_1F_1\left(\frac{1}{2\hbar};
\half;-\frac{x^2}{2\hbar}\right),
$
where 
\begin{equation}
\label{Kummer}
{}_1F_1(a;c;z) :=
\sum_{n=0}^\infty \frac{(a)_n}{(c)_n}\;
\frac{z^n}{n!}
\end{equation}
is the 
\textbf{Kummer confluent hypergeomtric function},
and the 
\textbf{Pochhammer symbol} $(a)_n$
is defined by
\begin{equation}
\label{Poch}
(a)_n :=a (a+1)(a+2)\cdots (a+n-1).
\end{equation}
For $\hbar>0$, 
the topological recursion 
determines the asymptotic 
expansion of a particular entire solution known as
 a \textbf{Tricomi confluent 
hypergeomtric function} 
\begin{multline*}
\Psi^\Catalan(x,\hbar) 
\\
= 
\left(-\frac{1}{2\hbar}\right)^{\frac{1}{2\hbar}}
\left(
\frac{\Gamma[\half]}{\Gamma[\frac{1}{2\hbar}
+\half]}
{}_1F_1\left(\frac{1}{2\hbar};
\half;-\frac{x^2}{2\hbar}\right)
+\frac{\Gamma[-\half]}{\Gamma[\frac{1}{2\hbar}]}
\sqrt{-\frac{x^2}{2\hbar}}
{}_1F_1\left(\frac{1}{2\hbar}
+\half;
\frac{3}{2};-\frac{x^2}{2\hbar}\right)
\right).
\end{multline*}
The expansion is given in the form
\begin{equation}
\begin{aligned}
\label{Catalan Psi expansion}
\Psi^{\Catalan}(x,\hbar)
&=
\left(\frac{1}{x}\right)^{\frac{1}{\hbar}}
\sum_{n=0}^\infty 
\frac{\hbar^n \left(\frac{1}{\hbar}\right)_{2n}}
{(2n)!!}\cdot  \frac{1}{x^{2n}}
\\
&=
\exp\left(
\sum_{g=0}^\infty 
\sum_{n=1}^\infty \frac{1}{n!}\hbar^{2g-2+n}
F_{g,n}^C(x,\dots,x)
\right),
\end{aligned}
\end{equation}
where $F_{g,n}^C$ is defined by 
\eqref{Catalan Fgn}, in terms of 
 generalized Catalan numbers.
The expansion variable $x^2$ in
\eqref{Catalan Psi expansion}
indicates the class of 
irregularity of the Hermite differential equation 
at $x=\infty$. The cases for $(g,n)=(0,1)$ and $(0,2)$
require again a special treatment, as we discussed
earlier.

The Hermite
differential equation
 becomes simple for $\hbar=1$,
and we have the asymptotic expansion
\begin{multline}
\label{error asymptotic}
i\sqrt{\frac{\pi}{2}}e^{-\half x^2}\left[
1-\erf\left(\frac{ix}{\sqrt{2}}\right)\right]
= \sum_{n=0}^\infty \frac{(2n-1)!!}{x^{2n+1}}
\\
=\exp\left(
\sum_{2g-2+n\ge -1}\frac{1}{n!}
\sum_{\mu_1,\dots,\mu_n>0}
\frac{C_{g,n}(\mu_1,\dots,\mu_n)}
{\mu_1\cdots\mu_n}
\prod_{i=1}^n x^{-(\mu_1+\cdots +\mu_n)}
\right).
\end{multline}
Here, 
$
\erf(x) :=\frac{2}{\sqrt{\pi}} \int_0^x e^{-z^2} dz
$
is the   Gau\ss\ error function.
\begin{figure}[htb]
\centerline{\epsfig{file=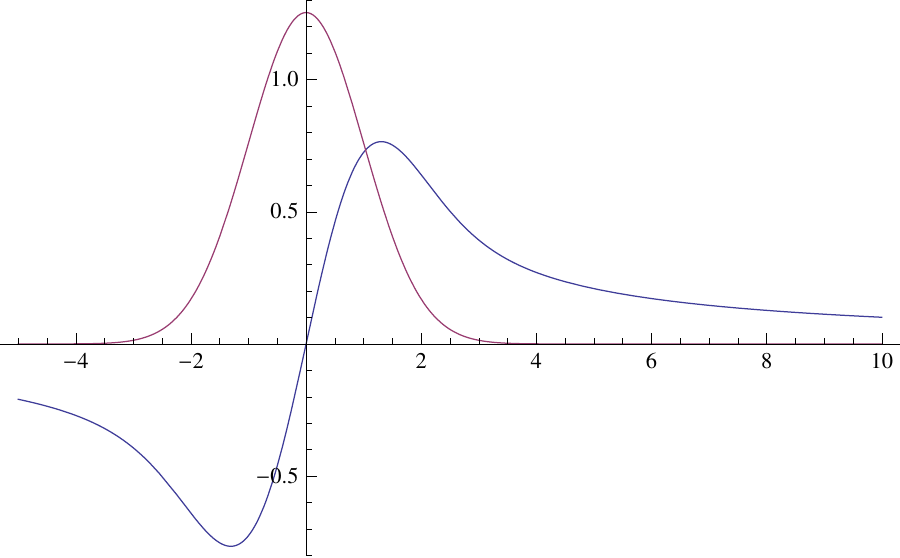, width=3in}}
\caption{The imaginary part and the real part
of $\Psi^{\Catalan}(x,1)$. For $x>>0$, the imaginary
part dies down, and only the real part has a
non-trivial asymptotic
expansion. Thus \eqref{error asymptotic} 
is a series with real coefficients. 
}
\label{fig:erf}
\end{figure}

One of the
two independent solutions to the third
example, the Gau\ss\ hypergeometric
equation, that is holomorphic
 around $x=0$, is given by 
\begin{equation}
\label{Gauss 0}
\Psi^\Gauss(x,\hbar)
= {}_2F_1
\left(-\frac{\sqrt{(h-1)(h-3)}}
   {2h}+\frac{1}{h}-
   \frac{1}{2},\frac{\sqrt{(h-1)(h-3)}}
   {2h}+\frac{1}{h}-\frac{1}{2};
   \frac{1}{h};x\right),
\end{equation}
where 
\begin{equation}
\label{Gauss}
{}_2F_1(a,b;c;x):=
\sum_{n=0}^\infty \frac{(a)_n(b)_n}{(c)_n}\;
\frac{x^n}{n!}
\end{equation}
is the \textbf{Gau\ss\ hypergeometric
function}.
The topological recursion
calculates the B-model genus
expansion of the
periods of the Legendre family of elliptic curves
\eqref{Legendre}
at the point where the 
elliptic curve degenerates to a nodal
rational curve.
For example, the procedure
applied to the spectral curve
$$
y^2
+\frac{2x-1}{x(x-1)}y+\frac{1}{4x(x-1)}=0
$$
with a choice of  
$$
\eta=\frac{-(2 x-1) - \sqrt{3x^2 - 3 x + 1}}
{2x (x-1)} dx,
$$
which is an eigenvalue $\a_1(x)$ of the Higgs
field $\phi$,
gives a genus expansion at $x=0$:
\begin{equation}
\label{Gauss expansion}
\Psi^\Gauss (x,\hbar)
=\exp\left(\sum_{g=0}^\infty
\sum_{n=1}^\infty
\frac{1}{n!}\hbar^{2g-2+n}F_{g,n}^\Gauss(x)\right).
\end{equation}
At $\hbar=1$, we have a topological recursion
expansion of the period 
$\omega_1(x)$ defined in \eqref{elliptic periods}:
\begin{equation}
\label{period expansion}
\frac{\omega_1(x)}{\pi} = \Psi^\Gauss (x,1)
=\exp\left(\sum_{g=0}^\infty
\sum_{n=1}^\infty
\frac{1}{n!}F_{g,n}^\Gauss(x)\right).
\end{equation}

A subtle point we notice here is that 
while the Gau\ss\ hypergeometric equation has
regular singular points at $x=0,1,\infty$, the
Hermite equation has an irregular singular point
of class $2$
at $\infty$. The spectral curve of each case has
an ordinary double point at $x=\infty$. But the 
crucial difference lies in the intersection of 
the spectral curve $\Sigma$ with the divisor
$C_\infty$.
For the Hermite case we have $\Sigma\cdot C_\infty
= 4$ and the intersection occurs all at once at
$x=\infty$.
For the Gau\ss\ hypergeometric case, the
intersection $\Sigma\cdot C_\infty = 4$
occurs once each at $x=0,1$, and twice at $x=\infty$. 
This \emph{confluence} of regular singular
points is the source of the irregular singularity
 in the Hermite differential equation.
 
 The fourth row indicates an example of 
 a quantum curve that has one regular singular
 point at $x=-1$ and one irregular singular point
 of class $1$ at $x = \infty$. The spectral curve
 has an ordinary double point at $x=\infty$, the same
 as the Hermite case. As 
 Figure~\ref{fig:spectral 2 and 4} shows, 
 the class of the irregular singularity at $x=\infty$
 is determined by how the spectral curve
 intersects with $C_\infty$. 
 
 \begin{figure}[htb]
\centerline{\epsfig{file=figcatalanspectral.pdf, height=1.2in}\qquad
\epsfig{file=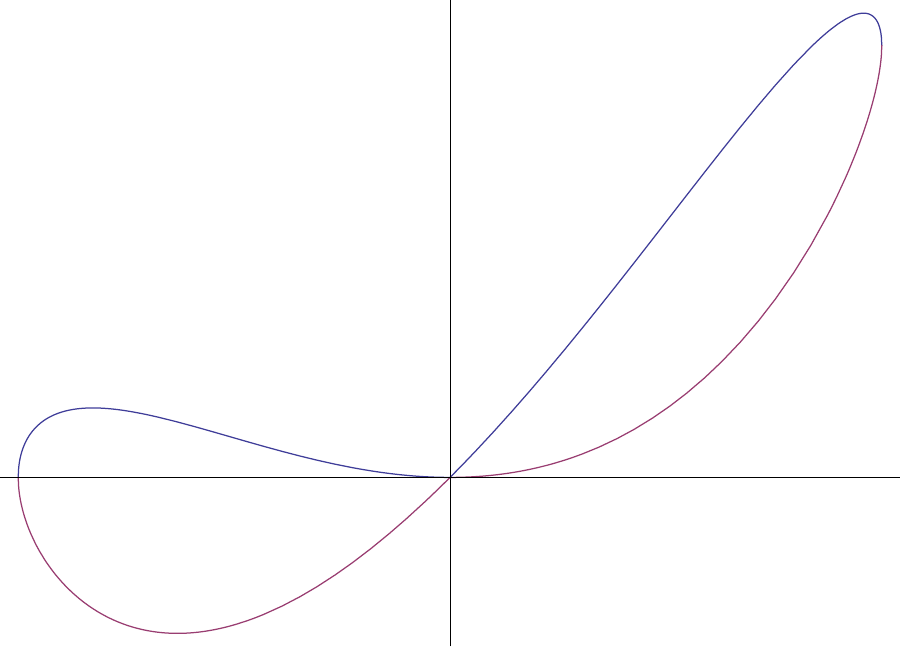, height=1.2in}}
\caption{The spectral curves  of the second
and the fourth examples.
The horizontal 
line is the divisor $C_\infty$, and the vertical
line is the fiber class $F$ at 
$x=\infty$. The spectral curve
intersects with $C_\infty$ a total of four times.
The curve on the right has a triple intersection 
at $x=\infty$, while the one on the left intersects
all at once.}
\label{fig:spectral 2 and 4}
\end{figure}

The existence of the irregular singularity in 
the quantum curve associated with a spectral
curve has nothing to do with the singularity
of the spectral curve. The fifth example shows
a non-singular spectral curve of genus $1$
(Figure~\ref{fig:spectral 5}), for which the quantum
curve has a
class $1$ irregular singularity at $x=\infty$.

\begin{figure}[htb]
\centerline{\epsfig{file=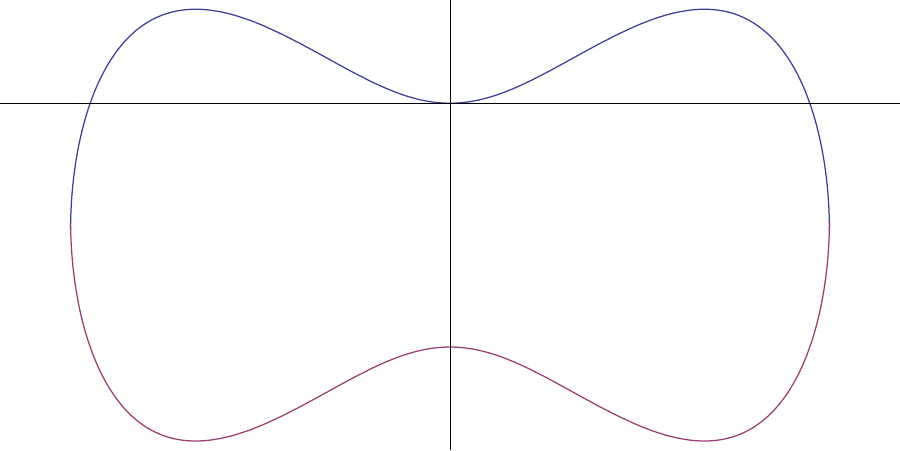, height=1.2in}}
\caption{The spectral curve  of the fifth
 example, which is non-singular. The corresponding
 quantum curve has two regular singular points
 at $x=\pm1$, and a class $1$ irregular singular 
 point at $x=\infty$.
}
\label{fig:spectral 5}
\end{figure}

\section{Difference operators as quantum 
curves}
\label{sect:difference}

Quantum curves often appear as  infinite-order
differential operators, or \emph{difference}
 operators. In this section we present three
 typical examples: simple Hurwitz numbers,
 special double Hurwitz numbers, and 
 the Gromov-Witten invariants of $\bP^1$.
 These examples do not come from the usual
 Higgs bundle framework, because the 
 rank of the Higgs bundle corresponds to the
 order of the quantum curves as
 a differential operator. Therefore, we ask:
 
 \begin{quest}
 What is the geometric structure generalizing
 the Hitchin spectral curves that correspond
 to difference operators as their quantization?
 \end{quest}

In these lectures, we do not address this question,
leaving it for a future investigation.
We are content with giving examples here.
We refer to \cite{KMar}
for a new and different perspective for 
the notion of quantum curves for difference
operators.

\subsection{Simple and 
orbifold Hurwitz numbers}
\label{sub:Hurwitz}

The \emph{simple Hurwitz number}
 $H_{g,n}(\vec{\mu})$ counts the automorphism
weighted number of the topological types of 
simple Hurwitz covers of $\bP^1$ of type 
$(g,\vec{\mu})$. 
A holomorphic map $\varphi:C\lrar \bP^1$ is a
\emph{simple Hurwitz cover}
 of type $(g,\vec{\mu})$  
if $C$ is a complete nonsingular algebraic curve
defined
over $\bC$ of genus $g$, $\varphi$ has $n$ labeled 
poles of orders 
$\vec{\mu}=(\mu_1,\dots,\mu_n)$, and all 
other critical 
points of $\varphi$ are unlabeled
simple ramification points.

In a similar way, we consider the \emph{orbifold
Hurwitz number} $H_{g,n}^{(r)}(\vec{\mu})$
 for every positive integer
$r>0$ to be the automorphism weighted 
count of the topological types of 
smooth orbifold morphisms $\varphi:C\lrar \bP^1[r]$
with the same pole structure as the simple
Hurwitz number case. 
Here, $C$ is a connected $1$-dimensional
orbifold (a \emph{twisted} curve)
modeled on a nonsingular 
curve of genus $g$ with $(\mu_1+\cdots+\mu_n)/r$
stacky points of the type $\big[p\big/(\bZ/r\bZ)\big]$.
We impose  that the inverse image of the
morphism 
$\varphi$ of the unique stacky point 
$\big[0\big/(\bZ/r\bZ)\big]\in \bP^1[r]$
coincides with the set of stacky points of $C$.
When $r=1$ we recover the 
simple Hurwitz number:
$H_{g,n}^{(1)}(\vec{\mu}) = 
H_{g,n}(\vec{\mu})$.

\begin{thm}[Cut-and-join equation, \cite{BHLM}]
The orbifold
Hurwitz numbers
$H_{g,n}^{(r)}(\mu_1,\dots,\mu_n)$
satisfy the following equation.
\begin{multline}
\label{eq:CAJ}
s H_{g,n}^{(r)}(\mu_1,\dots,\mu_n)
= \half\sum_{i\ne j} (\mu_i+\mu_j)
H_{g,n-1}^{(r)}
\left(
\mu_i+\mu_j,\mu_{[\hat{i},\hat{j}]}
\right)
\\
+\half\sum_{i=1}^n
\sum_{\a+\b=\mu_i}
\a\b 
\left[
H_{g-1,n+1}^{(r)}
\left(
\a,\b,\mu_{[\hat{i}]}
\right)
+\sum_{\substack{g_1+g_2=g\\
I\sqcup J=[\hat{i}]}}
H_{g_1,|I|+1}^{(r)}\big(\a,\mu_I\big)
H_{g_2,|J|+1}^{(r)}\big(\b,\mu_J\big)
\right].
\end{multline}
Here 
\begin{equation}
\label{eq:s}
s=s(g,\vec{\mu}) = 2g-2+n +
 \frac{\mu_1+\cdots+\mu_n}{r}
\end{equation}
 is the number of 
simple ramification point given by the
Riemann-Hurwitz formula.
As before, we use the convention that for any
subset $I\subset [n] = \{1,2,\dots,n\}$, 
$\mu_I=(\mu_i)_{i\in I}$. The hat notation
$\hat{i}$ indicates that the index $i$ is removed. 
The last summation is over all 
partitions of $g$ and set partitions of
$[\hat{i}] = \{1, \dots, i-1,i+1,\dots,n\}$. 
\end{thm}

\begin{rem}
There is a combinatorial description, in the same
manner we have done for the Catalan numbers in 
Section~\ref{sect:Catalan}, for simple and 
orbifold Hurwitz numbers. The cut-and-join
equation is derived as the edge-contraction 
formula, exactly in the same way for the
Catalan recursion \eqref{Catalan recursion}. 
See 
\cite{OM3} for more detail.
\end{rem}

We regard $H_{g,n}^{(r)}(\vec{\mu})$ 
as a function in $n$ integer 
variables $\vec{\mu}\in\bZ_+^n$. 
Following the recipe of \cite{DMSS, EMS, MS}
that is explained in the earlier sections, 
we define the \emph{free energies} as the Laplace transform
\begin{equation}
\label{Fgn}
F^{(r)}_{g,n}(z_1,\dots,z_n) =
\sum_{\vec{\mu}\in \bZ_+^n}
H_{g,n}^{(r)}(\vec{\mu})\; e^{-\la \vec{w},\vec{\mu}\ra}.
\end{equation}
Here, $\vec{w}=(w_1,\dots,w_n)$ is the 
vector of the Laplace
dual coordinates of $\vec{\mu}$, 
$\la \vec{w},\vec{\mu}
\ra=w_1\mu_1+\cdots +w_n\mu_n$,
and  variables $z_i$ and $w_i$
for each $i$
are related by the $r$-\emph{Lambert}
function
\begin{equation}
\label{r-Lambert}
e^{-w} = z e^{-z^r}.
\end{equation}
It is often convenient to use a different variable
$
x= z e^{-z^r}$,
with which the plane
analytic curve called the \textbf{$r$-Lambert curve} is 
given by 
\be
\label{rLamb curve}
\begin{cases}
x = ze^{-z^r}\\
y=z^r.
\end{cases}
\ee
Then 
the free energies $F^{(r)}_{g,n}$ of \eqref{Fgn}
are  generating functions of the 
orbifold Hurwitz numbers. By abuse
of notation, we also write
\begin{equation}
\label{Fgn in x}
F^{(r)}_{g,n}(x_1,\dots,x_n)
 = \sum_{\vec{\mu}\in \bZ_+^n}
H_{g,n}^{(r)}(\vec{\mu}) 
\prod_{i=1}^n x_i^{\mu_i}
= \sum_{\vec{\mu}\in \bZ_+^n}
H_{g,n}^{(r)}(\vec{\mu}) 
\prod_{i=1}^n \left(z_i e^{-z_i ^r}\right)^{\mu_i}.
\end{equation}
For 
every $(g,n)$, the power series \eqref{Fgn in x}
in $(x_1,\dots,x_n)$
is convergent and defines an analytic function.

\begin{thm}[Differential recursion for Hurwitz
numbers, \cite{BHLM}]
\label{thm:Fgn recursionH}
In terms of the $z$-variables, the 
free energies are calculated as follows.
\begin{align}
\label{F01}
F^{(r)}_{0,1}(z) &= \frac{1}{r}z^r-\half z^{2r},
\\
\label{F02}
F^{(r)}_{0,2}(z_1,z_2) &=\log\frac{z_1-z_2}{x_1-x_2}
-(z_1^r+z_2^r),
\end{align}
where $x_i = z_i e^{-z_i ^r}$.
 For $(g,n)$ in the stable range, 
 i.e., when $2g-2+n>0$, the free energies 
 satisfy the differential recursion equation
 \begin{multline}
 \label{diffrecursion}
 \left(
 2g-2+n+\frac{1}{r}\sum_{i=1}^n
 z_i \frac{\partial}{\partial z_i}
 \right)
 F^{(r)}_{g,n}(z_1,\dots,z_n) 
 \\
 =
 \half\sum_{i\ne j}\frac{z_iz_j}{z_i-z_j}
 \left[
 \frac{1}{(1-rz_i^r)^2}\frac{\partial}{\partial z_i}
 F^{(r)}_{g,n-1}\big(z_{[\hat{j}]}\big) -
  \frac{1}{(1-rz_j^r)^2}\frac{\partial}{\partial z_j}
 F^{(r)}_{g,n-1}\big(z_{[\hat{i}]}\big)
 \right]
 \\
 +
 \half \sum_{i=1}^n \frac{z_i^2}{(1-rz_i^r)^2}
 \left.
 \frac{\partial^2}{\partial u_1\partial u_2}
 F^{(r)}_{g-1,n+1}\big(
 u_1,u_2,z_{[\hat{i}]}
 \big)\right|_{u_1=u_2=z_i}
 \\
 +
 \half \sum_{i=1}^n \frac{z_i^2}{(1-rz_i^r)^2}
 \sum_{\substack{
 g_1+g_2=g\\I\sqcup J=[\hat{i}]}} ^{\rm{stable}}
 \left(
 \frac{\partial}{\partial z_i} 
 F^{(r)}_{g_1,|I|+1}(z_i,z_I)
 \right)
 \left(
  \frac{\partial}{\partial z_i} 
 F^{(r)}_{g_2,|J|+1}(z_i,z_J)
 \right).
 \end{multline}
\end{thm}

\begin{rem}
\label{rem:Fgn recursion}
Since $F^{(r)}_{g,n}(z_1,\dots,z_n)\big|_{z_i=0}=0$
for every $i$, the differential recursion 
\eqref{diffrecursion}, which is a 
linear first order partial differential 
equation, uniquely determines
$F^{(r)}_{g,n}$  inductively for all
$(g,n)$ subject to $2g-2+n>0$. 
This is a generalization of the  result of \cite{MZ}
to the orbifold case.
\end{rem}

\begin{rem}
The differential recursion 
of Theorem~\ref{thm:Fgn recursionH} is 
obtained by taking the Laplace 
transform of the cut-and-join 
equation for $H_{g,n}^{(r)}(\vec{\mu})$. 
The $r$-Lambert curve itself, \eqref{rLamb curve}, 
is obtained by
computing the Laplace transform of 
$H_{0,1}^{(r)}(\mu)$, and solving the
differential equation that arises from the 
cut-and-join 
equation. See also \cite{OM3}
for a different formulation of the $r$-Hurwitz
numbers using the graph enumeration formulation
and a universal mechanism to obtain the spectral 
curve.
\end{rem}

The differential recursion produces
two results, as we have seen for the 
case of the Catalan numbers.
One is the quantum curve by taking the 
principal specialization, and the other
the topological recursion of \cite{EO2007}.

\begin{thm}[Quantum curves for $r$-Hurwitz
numbers, \cite{BHLM}]
\label{thm:quantum curve}
We introduce the partition function,
or the wave function, of the orbifold
Hurwitz numbers  as
\begin{equation}
\label{partition}
\Psi^{(r)}(z,\hbar) = 
\exp\left(
\sum_{g=0} ^\infty \sum_{n=1} ^\infty
\frac{1}{n!}\hbar^{2g-2+n}F^{(r)}_{g,n}(z,z,\dots,z)
\right).
\end{equation}
It satisfies the following system of (an infinite-order)
linear differential equations.
\begin{align}
\label{P}
 \left(
 \hbar D - e^{r \left(-w+\frac{r-1}{2}\hbar\right)}
 e^{r\hbar D}
 \right)
 \Psi^{(r)}(z,\hbar) &=0,
\\
\label{Q}
\left(
\frac{\hbar}{2}
D^2-\left(\frac{1}{r}+\frac{\hbar}{2}
\right) D
-\hbar\frac{\partial}{\partial \hbar}
\right)
\Psi^{(r)}(z,\hbar) &=0,
\end{align}
where
\begin{equation*}
D=\frac{z}{1-rz^r}\frac{\partial}
{\partial z}=x\frac{\partial}{\partial x}
= -\frac{\partial}{\partial w}.
\end{equation*}
Let the differential operator of \eqref{P}
(resp.~\eqref{Q})
 be denoted by $P$  (resp.~$Q$). Then 
we have the commutator relation
\begin{equation}
\label{[P,Q]}
[P,Q]=P.
\end{equation}
The semi-classical limit of each of the 
equations \eqref{P} or \eqref{Q}
recovers the $r$-Lambert curve
\eqref{rLamb curve}.
\end{thm}

\begin{rem}
\label{rem:MSS}
The Schr\"odinger equation \eqref{P}
is first established in \cite{MSS}.
\end{rem}

\begin{rem}
\label{rem:r=1}
The above theorem
is a generalization of
  \cite[Theorem~1.3]{MS}
  for an arbitrary $r>0$. The restriction
  $r=1$ reduces to the simple Hurwitz case.
\end{rem}

\begin{rem}
Unlike the situation of Hitchin spectral curves,
the results of the quantization of the 
analytic spectral curves are a 
\textbf{difference-differential equation},
and a PDE containing the differentiation 
with respect to the deformation parameter
$\hbar$.
\end{rem}

Now let us define
\begin{equation}
\label{Wgnr}
W_{g,n}^{(r)}(z_1,\dots,z_n):=
d_1 d_2\cdots d_n F_{g,n}^{(r)}(z_1,\dots,z_n).
\end{equation}
Then we have

\begin{thm}[Topological recursion for orbifold
Hurwitz numbers, \cite{BHLM}]
\label{thm:r-EO}
For the stable range $2g-2+n>0$, the symmetric
differential forms \eqref{Wgnr}
satisfy the following
integral recursion 
formula.
\begin{multline}
\label{r-EO}
W_{g,n}^{(r)}(z_1,\dots,z_n)
=\frac{1}{2\pi i}\sum_{j=1} ^{r}
\oint_{\gam_j}K_j(z,z_1)
\Bigg[
W_{g-1,n+1}^{(r)}\big(z,s_j(z),z_2,\dots,z_n\big)
\\
+
\sum_{i=2}^n 
\left(
W_{0,2}^{(r)}(z,z_i)
\tensor
W_{g,n-1}^{(r)}
\big(s_j(z),z_{[\hat{1},\hat{i}]}\big)
+
W_{0,2}^{(r)}\big(s_j(z),z_i\big)
\tensor
W_{g,n-1}^{(r)}
\big(z,z_{[\hat{1},\hat{i}]}\big)
\right)
\\
+
\sum_{\substack{g_1+g_2=g\\
I\sqcup J=\{2,\dots,n\}}}
^{\rm{stable}}
W_{g_1,|I|+1}^{(r)}\big(z,z_I\big)
\tensor
W_{g_2,|J|+1}^{(r)}\big(s_j(z),z_J\big)
\Bigg].
\end{multline}
Here, the integration is taken with respect to 
$z$ along a small simple closed loop $\gam_j$
around $p_j$, and 
$\{p_1,\dots,p_r\}$ are the critical points of
the $r$-Lambert function
$x(z)=z e^{-z^r}$ at $1-rz^r=0$.
Since $dx=0$ has a simple
zero at each $p_j$,  the map
$x(z)$ is locally a double-sheeted covering
around $z=p_j$. 
We denote by $s_j$ the deck transformation 
on a small neighborhood of $p_j$. Finally,
the integration kernel is defined by
\begin{equation}
\label{r-kernel}
K_j(z,z_1)= \half\;
\frac{1}{W_{0,1}^{(r)}\big(s_j(z_1)\big)
-W_{0,1}^{(r)}(z_1)}
\tensor
\int_{z}^{s_j(z)}W_{0,2}^{(r)}(\;\cdot\;,z_1).
\end{equation}
\end{thm}

\begin{rem}
As mentioned earlier,
the significance of the integral formalism
is its universality. The differential equation
\eqref{diffrecursion} takes a 
 different form
depending on the counting problem, whereas
the integral formula \eqref{r-EO}
depends only on the 
choice of the spectral curve. 
\end{rem}

\begin{rem}
The proof is based on the idea of \cite{EMS}. 
The notion of the \emph{principal part}
of  meromorphic differentials
 plays a key role in converting the
Laplace transform of the cut-and-join
equation into a residue formula. 
\end{rem}

\subsection{Gromov-Witten invariants
of the projective line}
\label{sub:P1}

Hurwitz numbers and Gromov-Witten invariants
of $\bP^1$ are closely related \cite{OP}.
However, their relations to the topological recursion
is rather different. For example,
the topological recursion for stationary 
Gromov-Witten invariants of $\bP^1$
was conjectured by Norbury and Scott
\cite{NS} as a concrete formula,
but its proof \cite{DOSS} is done in a 
very different way than that of \cite{EMS,MZ}.
This is based on the fact that 
we do not have a counterpart of the 
cut-and-join equation for the case of  
the Gromov-Witten invariants of $\bP^1$.
Nonetheless, the quantum curve exists.

Let $\Mbar_{g,n}(\bP^1,d)$ denote the 
moduli space of stable maps of degree $d$
from an $n$-pointed genus $g$ curve to
$\bP^1$. This is an algebraic stack of 
dimension $2g-2+n +2d$. The dimension
reflects the fact that a generic map
from an algebraic curve to $\bP^1$ has 
only simple ramifications, and the number
of such ramification points,
which we
can derive from the Riemann-Hurwitz 
formula, gives the dimension 
of this stack. The  descendant 
Gromov-Witten invariants of 
$\bP^1$ are defined by 
\begin{equation}
\label{GW}
\left< \prod_{i=1}^n
\tau_{b_i}(\alpha_i)\right>_{g,n} ^d
:=
\int _{[\Mbar_{g,n}(\bP^1,d)]^{vir}}
\prod_{i=1}^n \psi_i ^{b_i} ev_i^*(\alpha_i),
\end{equation}
where 
$[\Mbar_{g,n}(\bP^1,d)]^{vir}$ is the virtual 
fundamental class of the moduli space,
\begin{equation*}
ev_i:\Mbar_{g,n}(\bP^1,d)\lrar \bP^1
\end{equation*}
is a natural morphism defined by
evaluating a stable map at the $i$-th marked 
point of the source curve, $\alpha_i\in H^*(\bP^1,\bQ)$
is a cohomology class of the target $\bP^1$,
and $\psi_i$ is  the tautological cotangent 
class in $H^2(\Mbar_{g,n}(\bP^1,d),\bQ)$.
We denote by $1$ the generator of
$H^0(\bP^1,\bQ)$, and by $\omega\in
H^2(\bP^1,\bQ)$ the Poincar\'e dual to the 
point class.
We assemble the Gromov-Witten invariants
into particular generating functions as follows.
For every $(g,n)$ in the stable sector
$2g-2+n>0$, we define the \emph{free energy} of 
type $(g,n)$
by
\begin{equation}
\label{FgnP1}
F_{g,n}(x_1,\dots,x_n)
:= 
\left< \prod_{i=1}^n \left(
-\frac{\tau_0(1)}{2} - \sum_{b=0}^\infty
\frac{b!\tau_b(\omega)}{x_i^{b+1}}
\right)
\right>_{g,n}.
\end{equation}
Here the degree $d$ is determined by the 
dimension condition of the 
cohomology classes to be integrated
over the virtual fundamental class.
We note that \eqref{FgnP1} contains 
the class $\tau_0(1)$. 
For unstable geometries, we introduce two 
functions
\begin{align}
\label{S0P1}
S_0(x) &:=
x-x\log x +
\sum_{d=1}^\infty 
\left< -\frac{(2d-2)!\tau_{2d-2}(\omega)}
{x^{2d-1}}
\right>_{0,1}^d,
\\
\label{S1P1}
S_1(x) &:=
-\half \log x+
\half \sum_{d=0}^\infty 
\left< 
\left(
-\frac{\tau_0(1)}{2} - \sum_{b=0}^\infty
\frac{b!\tau_b(\omega)}{x^{b+1}}
\right)^2
\right>_{0,2}^d,
\end{align}
utilizing an earlier work of \cite{DMSS}.
Then we have

\begin{thm}[The quantum curve for the
Gromov-Witten invariants of $\bP^1$, \cite{DMNPS}]
\label{thm:P1}
The wave function
\begin{equation}
\label{Psi-P1}
\Psi(x,\hbar)
:=\exp\left(
\frac{1}{\hbar}S_0(x) + S_1(x)
+\sum_{2g-2+n>0}\frac{\hbar^{2g-2+n}}{n!}
F_{g,n}(x,\dots,x)
\right)
\end{equation}
satisfies the quantum curve equation
of an infinite order
\begin{equation}
\label{qcP1}
\left[
\exp\left(
\hbar\frac{d}{dx}
\right)
+
\exp\left(
-\hbar\frac{d}{dx}
\right)
-x
\right]
\Psi(x,\hbar) = 0.
\end{equation}
Moreover, 
the free energies $F_{g,n}(x_1,\dots,x_n)$
as functions in $n$-variables,
 and hence 
all the Gromov-Witten invariants \eqref{GW},
can be recovered 
from the equation
\eqref{qcP1} alone,
using the mechanism of the
\textbf{topological recursion}
of \cite{EO2007}.
\end{thm}

\begin{rem}
The appearance of the extra terms in
$S_0$ and $S_1$, in particular,
the $\log x$ terms, is trickier than the 
cases studied in these lectures.
We refer to \cite[Section 3]{DMNPS}.
\end{rem}

\begin{rem}Put 
\begin{equation}
\label{SmP1}
S_m(x) := \sum_{2g-2+n=m-1}\frac{1}{n!}
F_{g,n}(x,\dots,x).
\end{equation}
Then our wave function is of the form
\begin{equation}
\label{WKB-P1}
\Psi(x,\hbar) = \exp\left( \sum_{m=0}^\infty
\hbar^{m-1} S_m(x)\right),
\end{equation}
which provides the WKB approximation 
of the quantum curve equation 
\eqref{qcP1}. Thus the significance of 
\eqref{FgnP1} is that the 
exponential generating function 
\eqref{Psi-P1} of the
descendant Gromov-Witten invariants of 
$\bP^1$ gives the  solution to the 
WKB analysis in a closed formula
for the difference equation
\eqref{qcP1}.
\end{rem}

\begin{rem}
For the case of Hitchin spectral curves 
\cite{OM1,OM2},
the Schr\"odinger-like equation
\eqref{qcP1} is a direct consequence 
of the generalized topological recursion. In the
$GW(\bP^1)$ context, the topological recursion does 
not play any role in establishing 
\eqref{qcP1}. 
\end{rem}

We can
recover the classical mechanics  corresponding 
to \eqref{qcP1} by taking its 
semi-classical limit, which is the 
singular perturbation limit
\begin{multline}
\label{SCL-P1}
\lim_{\hbar\rar 0}
\left(
e^{-\frac{1}{\hbar}S_0(x)}
\left[
\exp\left(
\hbar\frac{d}{dx}
\right)
+
\exp\left(
-\hbar\frac{d}{dx}
\right)
-x
\right]
e^{\frac{1}{\hbar}S_0(x)}
e^{ \sum_{m=1}^\infty
\hbar^{m-1} S_m(x)}
\right)
\\
=
\left(e^{S_0'(x)}+e^{-S_0'(x)}-x
\right)
e^{S_1(x)}=0.
\end{multline}
In terms of new variables
$y(x) = S_0'(x)$ and $z(x) = e^{y(x)}$, 
the semi-classical limit  gives us 
an equation for the spectral curve
$$
z\in \Sigma = \bC^*\subset \bC \times \bC^*
\overset{\exp}{\longleftarrow} T^*\bC = \bC^2
\owns (x,y)
$$ by
\begin{equation}
\label{spectralP1}
\begin{cases}
x = z+\frac{1}{z}\\
y=\log z
\end{cases}.
\end{equation}
This is the reason we consider \eqref{qcP1}
as the quantization of the Laudau-Ginzburg
model 
$$
x=z+\frac{1}{z}.
$$

It was conjectured in \cite{NS} that the
stationary Gromov-Witten theory of $\bP^1$ should
satisfy the topological recursion
with respect to the spectral curve
\eqref{spectralP1}. 
The conjecture
is solved in \cite{DOSS} as a corollary to 
its main theorem, which 
establishes the correspondence between the 
topological 
recursion and the Givental formalism.

The key discovery of \cite{DMNPS} is that
the quantum curve equation \eqref{qcP1}
is equivalent to a  recursion equation
\begin{equation}
\label{XP1}
\frac{x}{\hbar}
\left(
e^{-\hbar\frac{d}{dx}} -1
\right) X_d(x,\hbar)
+\frac{1}{1+\frac{x}{\hbar}}
e^{\hbar\frac{d}{dx}}X_{d-1}(x,\hbar)=0
\end{equation}
for a rational function
\begin{equation}
\label{XdP1}
X_d(x,\hbar) = \sum_{\lambda\vdash d}
\left(
\frac{\dim \lambda}{d!}
\right)^2
\prod_{i=1}^{\ell(\lambda)}
\frac{x+(i-\lambda_i)\hbar}{x+i\hbar}.
\end{equation}
Here $\lambda$ is a partition of $d\ge 0$
with parts $\lambda_i$ and 
$\dim \lambda$ denotes the dimension 
of the irreducible representation of 
the symmetric group $S_d$ 
characterized  by $\lambda$.

\begin{ack}
The present article is based on the series of lectures
that the authors have given in Singapore, Kobe,  
Hannover, Hong Kong, and Leiden in 2014--2015.
They are indebted to Richard Wentworth
and Graeme Wilkin for their kind invitation to 
 the Institute for Mathematical
Sciences at the National University of Singapore,
where these lectures were first delivered at the 
IMS Summer Research Institute, 
\emph{The Geometry, Topology and Physics of Moduli Spaces of Higgs Bundles}, in July 2014.
The authors are also grateful to  
Masa-Hiko Saito 
for his kind invitation to run 
the Kobe Summer School
consisting of an undergraduate
and a graduate courses  on the
related topics at Kobe University, Japan, in 
July--August, 2014. 
A part of these lectures was also delivered 
at the \emph{Advanced Summer School: Modern Trends in Gromov-Witten Theory}, 
Leibniz Universit\"at Hannover, in September 2014.

The authors' 
special thanks are due  to Laura P.\ Schaposnik,
whose constant interest  in the subject
made these lecture notes possible. 

The authors are grateful to
the American Institute of Mathematics in California, 
the Banff International Research Station,
 Max-Planck-Institut f\"ur Mathematik in Bonn,
  and the Lorentz Center for Mathematical 
 Sciences, Leiden,
for their hospitality and financial support for
the collaboration of the authors. 
They  thank
J\o rgen Andersen,
Philip Boalch, 
Ga\"etan Borot,
Vincent Bouchard,
Andrea Brini,
Leonid Chekhov,
Bertrand Eynard,
Laura Fredrickson, 
Tam\'as Hausel,
Kohei Iwaki,
Maxim Kontsevich,
Andrew Neitzke,
Paul Norbury,
Alexei Oblomkov,
Brad Safnuk,
Albert Schwarz, 
Sergey Shadrin,
Yan Soibelman, 
and
Piotr Su\l kowski
for useful discussions.
They also thank the referees for numerous 
suggestions to improve these lecture notes.
O.D.\ thanks the Perimeter Institute for 
Theoretical Physics,
Canada, and
M.M.\  thanks 
the University of Amsterdam, 
l'Institut des Hautes \'Etudes Scientifiques,  
Hong Kong University of 
Science and Technology, and
the Simons Center for Geometry 
and Physics,
 for
financial support and 
hospitality.
The research of O.D.\ has been supported by
GRK 1463 of Leibniz Universit\"at 
 Hannover and MPIM in Bonn.
The research of M.M.\ has been supported 
by MPIM in Bonn, 
NSF grants DMS-1104734 and DMS-1309298, 
and NSF-RNMS: Geometric Structures And 
Representation Varieties (GEAR Network, 
DMS-1107452, 1107263, 1107367).
\end{ack}


\providecommand{\bysame}{\leavevmode\hbox to3em{\hrulefill}\thinspace}

\bibliographystyle{amsplain}

\end{document}